\documentclass[12pt]{amsart}
\usepackage[utf8]{inputenc}
\usepackage{graphicx}
\usepackage{color}
\usepackage{xcolor}
\usepackage{times}
\usepackage[hidelinks]{hyperref}
\usepackage{enumerate,latexsym}
\usepackage{latexsym}
\usepackage{amsmath,amssymb}
\usepackage{graphicx}

\usepackage{amsmath,amsthm,amsfonts,amssymb}
\usepackage{graphicx}
\usepackage{dsfont}

\makeatletter
\def\namedlabel#1#2{\begingroup
 #2%
 \def\@currentlabel{#2}%
 \phantomsection\label{#1}\endgroup
}
\makeatother

\theoremstyle{plain}
\newtheorem*{theorem*}{Theorem}
\newtheorem*{thmex*}{Theorem~\ref{example}}
\newtheorem*{thmasymp*}{Theorem~\ref{thmAsymp}}
\newtheorem{theorem}{Theorem}[section]

\newtheorem{corollary}[theorem]{Corollary}

\newtheorem{lemma}[theorem]{Lemma}

\newtheorem*{eg*}{Example}

\theoremstyle{definition}
\newtheorem{definition}[theorem]{Definition}

\definecolor{rrr}{rgb}{.9,0,.1}

\definecolor{rr}{rgb}{.8,0,.3}

\title{The Rest of the Tilings of the Sphere by Regular Polygons}
\author{Colin Adams, Cameron Edgar, Peter Hollander, and Liza Jacoby}
\date{January 2021}

\begin{document}
\begin{abstract}
    We determine all non-edge-to-edge tilings of the sphere by regular spherical polygons of three or more sides.
\end{abstract}

\maketitle

\section{Introduction}

The Platonic solids were known to the ancient Greeks. When they are projected out to a circumscribing sphere, they become the five regular tilings of the sphere, each comprised of a single type of regular spherical polygon called its prototile.

Other tilings of the sphere by regular spherical polygons of three or more sides glued edge-to-edge include the infinite families of the prisms and anti-prisms. As reported by Pappus of Alexandria \cite{Pappus}, Archimedes listed thirteen polyhedra with regular polygonal faces and symmetries identifying any pair of vertices, which when projected to the sphere, yield  additional edge-to-edge tilings of the sphere by regular spherical polygons.

In 1966, Johnson listed the 92 Johnson solids \cite{Johnson}, and conjectured this was a complete list of all the additional polyhedra with regular polygonal faces. This was proved  by Zalgaller \cite{Zalgaller} in 1967.  Of these, 25 are circumscribable by a sphere and can thus be projected out onto the surface of the sphere to obtain tilings by spherical regular polygons. This completed the classification of edge-to-edge tilings of the sphere by regular spherical polygons, with a total of 43 possibilities beside the two infinite classes corresponding to the prisms and anti-prisms.

However, the classification of tilings of the sphere by regular spherical polygons was yet incomplete, as non-edge-to-edge tilings were left unconsidered. In this paper, we complete the classification of tilings by regular polygons by determining exactly the non-edge-to-edge tilings of the sphere by regular polygons of three or more sides. This completes the work begun by Archimedes over two millennia ago and also provides some beautiful new families of tilings.

Polygonal tilings such that at least two tiles intersect in the interior of sides but do not overlap in the entirety of the interiors of both sides are said to be non-edge-to-edge. 
There are a great many non-edge-to-edge tilings of the plane. Even if we restrict to regular polygons, there has been no attempt to classify all of the possibilities (see Section 2.4 of \cite {GS2}). If we restrict to uniform tilings, so the symmetry group is transitive on vertices, there are eight families. But when not uniform, the possibilities are only known in special cases. (See \cite{Bolcskei} and \cite{Schattschneider} for instance).
    


As with the plane, we can also consider non-edge-to-edge tilings of the sphere. 
In this paper, we obtain a complete classification of all non-edge-to-edge tilings of the sphere by regular spherical $n$-gons for $n \geq 3$. Tilings with bigonal tiles will be considered in a future paper. The tilings fall into six classes which we now describe.

\medskip

Those in the first class are called the {\bf kaleidoscope tilings} and fall into five infinite families. In Figure \ref{tet-tet}, we see an example of the first infinite family, which is a tiling of the sphere by regular triangles with two different sizes of triangles represented, each appearing four times. We can shrink the side-length of the set of smaller triangles and expand the side-length of the set of larger triangles. In the limit, we obtain the tetrahedral tiling of the sphere. If we were to expand the side-length of the set of smaller triangles and shrink the others, the limit yields the octahedral tiling. So, as we expand the first set and shrink the second set, we pass through the octahedral edge-to-edge tiling with all triangles of the same size, and as we continue to expand the size of the first set and shrink the second, in the limit we again arrive at the tetrahedral tiling. We call these the triangle-triangle kaleidoscope tilings.


\begin{figure}[htpb]
\begin{center}
\includegraphics[width=.4\textwidth]{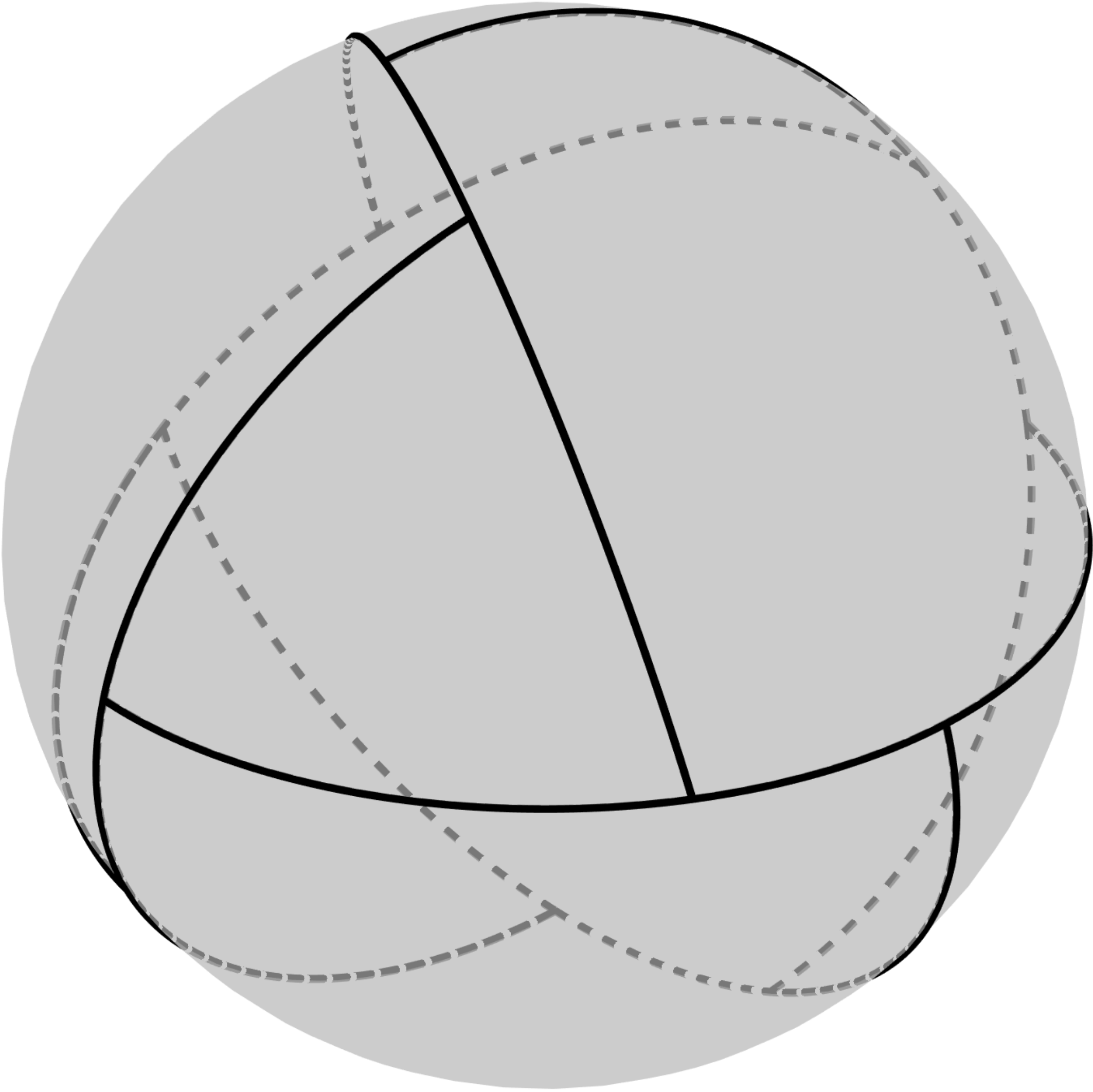}
\caption{A representative of the triangle-triangle continuum of kaleidoscope tilings.}
\label{tet-tet}
\end{center}
\end{figure}

In the case of Figure \ref{cubocto}, we see a tiling with two prototiles, one a square and the other a triangle, with the side-length of the square less than that of the triangle. This is our second family of kaleidoscope tilings, called the square-triangle kaleidoscope tilings.

\begin{figure}[htpb]
\begin{center}
\includegraphics[width=.4\textwidth]{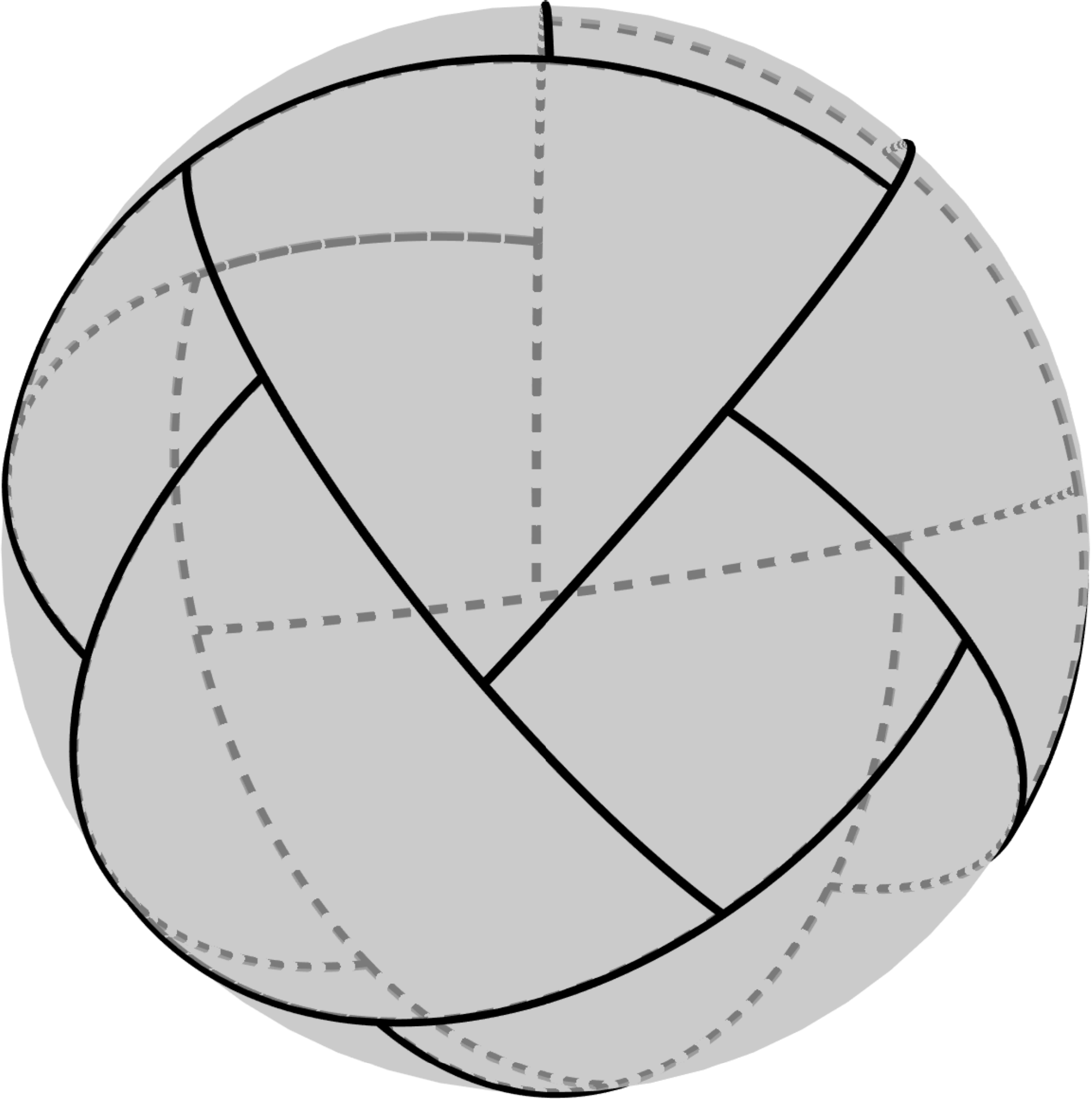}
\caption{A representative of the square-triangle continuum of kaleidoscope tilings.}
\label{cubocto}
\end{center}
\end{figure}

 If we shrink the size of the squares and expand the size of the triangles, we ultimately obtain an octahedron tiling in the limit. If instead we expand the squares and shrink the triangles, then when they have the same length edges, we obtain the cuboctahedron tiling. Then as we continue to expand the squares and shrink the triangles so that the squares have side-length greater than that of the triangles, we obtain our third family of kaleidoscope tilings, called the triangle-square kaleidoscope tilings, an example of which appears in Figure \ref{trisquaretiling}. If we continue to expand the side-length of the squares, in the limit, as the triangles disappear, we end at the cube tiling.

\begin{figure} [htpb]
    \centering
    \includegraphics[width = 0.4\textwidth]{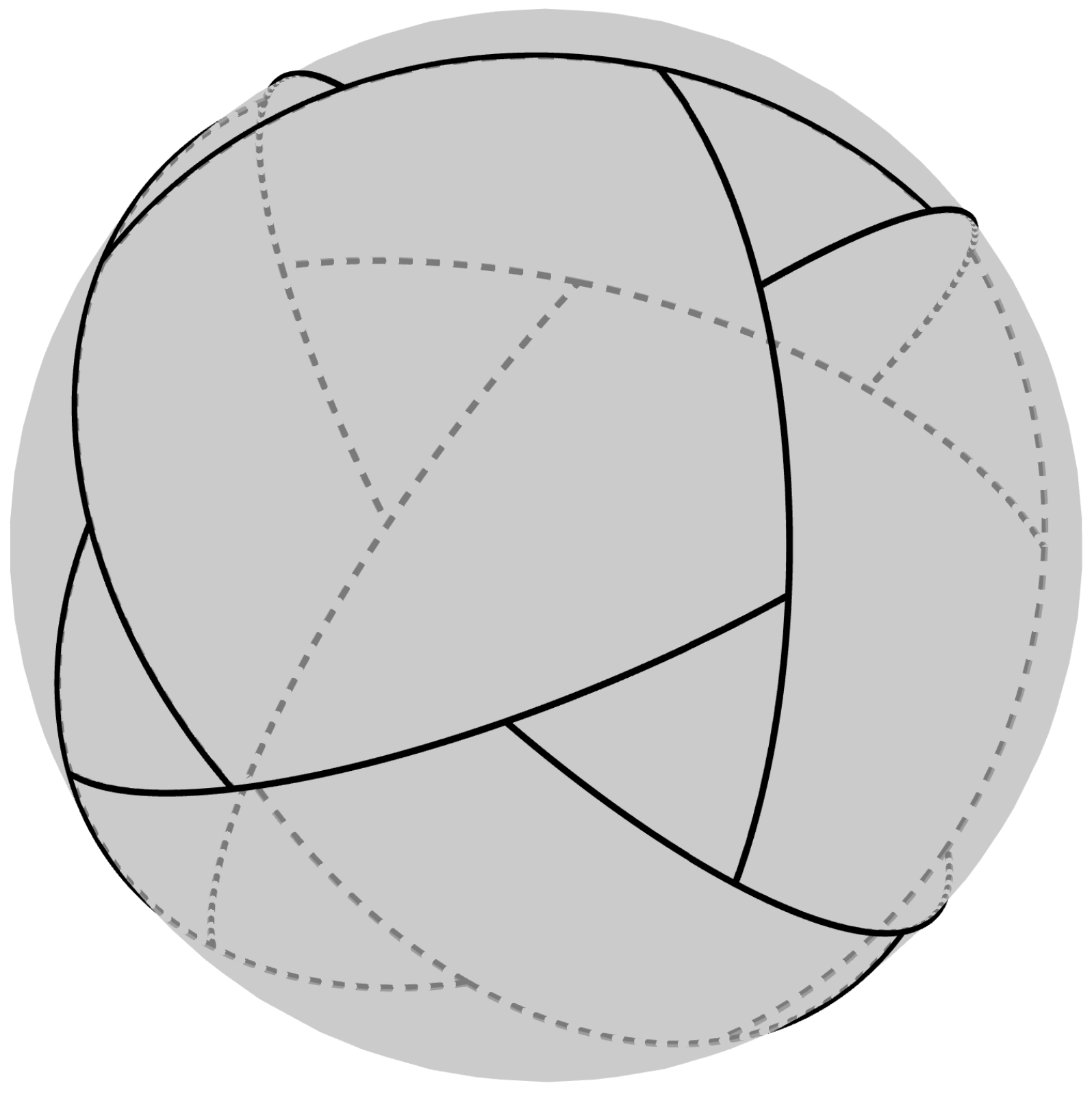}
    \caption{A representative of the triangle-square continuum of kaleidoscope tilings.}
    \label{trisquaretiling}
\end{figure}

In Figure \ref{dodicos}, we have a tiling with one prototile a pentagon and the other prototile a triangle, with the pentagon side-length less than that of the triangle. This is our fourth family, the pentagon-triangle kaleidsocope tilings. Shrinking the pentagons and expanding the triangles limits to the icosahedron tiling. Expanding the pentagons while shrinking the triangles until they have the same side-length yields the  icosidodecahedron tiling. 
\begin{figure}[htpb]
\begin{center}
\includegraphics[width=.4\textwidth]{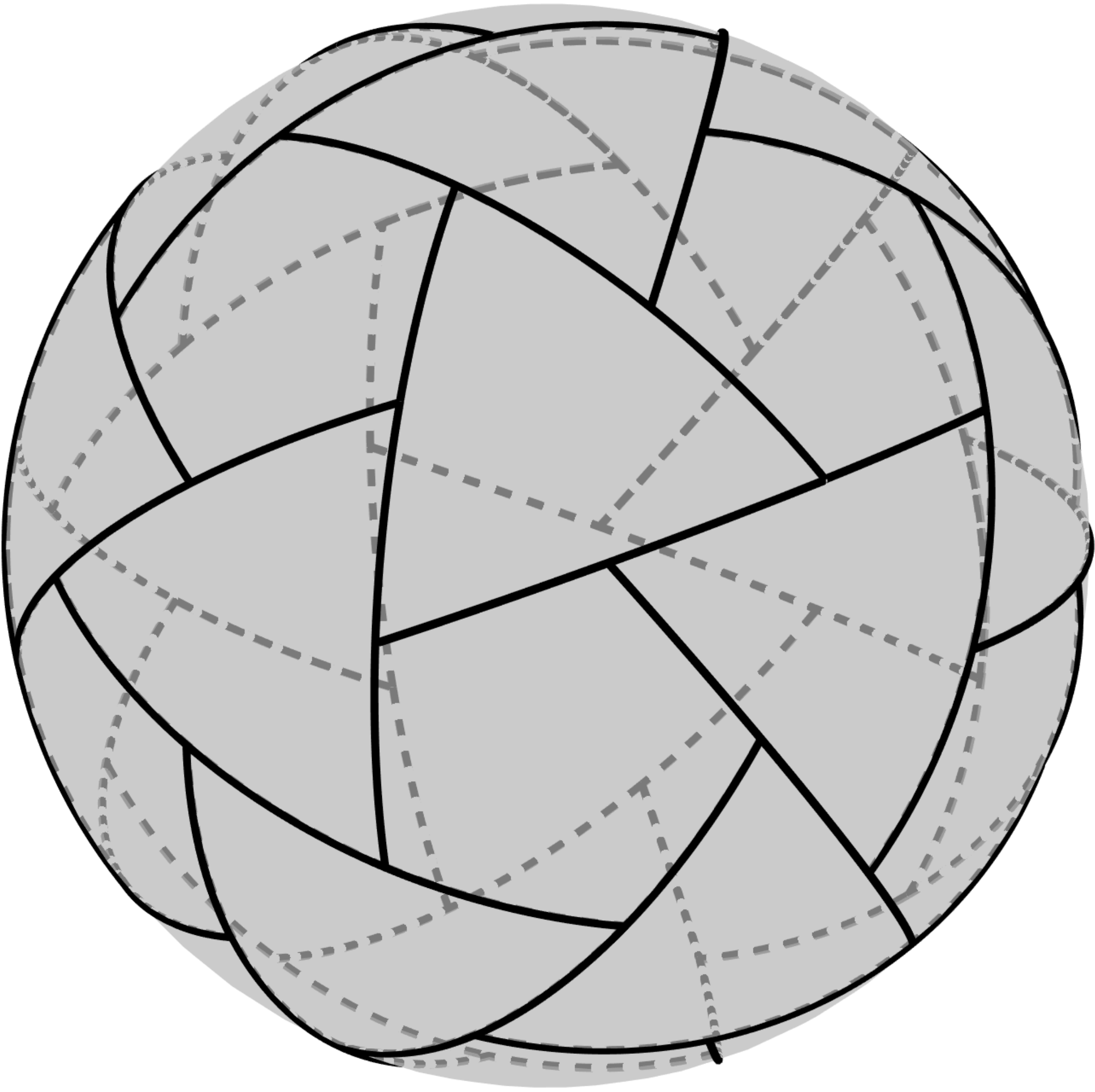}
\caption{A representative of the pentagon-triangle continuum of kaleidoscope tilings.}
\label{dodicos}
\end{center}
\end{figure}

Continuing to expand the pentagons while shrinking the triangles leads to our fifth family, the triangle-pentagon kaleidoscope tilings, as in Figure \ref{tripenttiling}. In the limit, we obtain the dodecahedron tiling.

\begin{figure} [htpb]
    \centering
    \includegraphics[width = 0.4\textwidth]{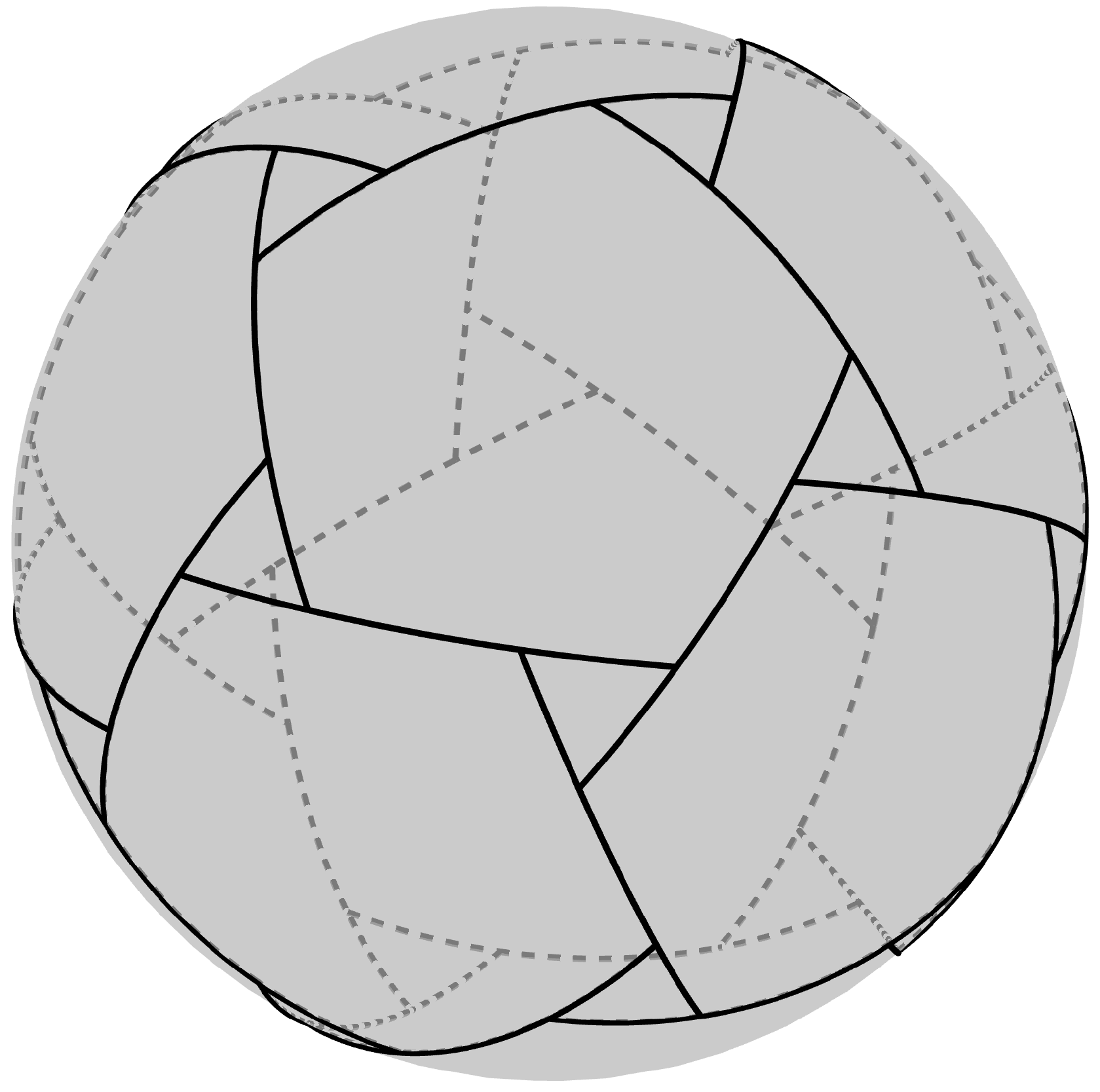}
    \caption{A representative of the triangle-pentagon continuum of kaleidoscope tilings.}
    \label{tripenttiling}
\end{figure}





The second class of tilings is called the {\bf 2-hemisphere tilings}. We arrive at these by considering the edge-to-edge tilings of $S^2$,  which correspond to the Platonic tilings, the prism and anti-prism tilings, the thirteen Archimedean tilings and the twenty-five Johnson tilings. In the case that there is a great circle contained in the union of the edges, we can cut along it to obtain edge-to-edge tilings of the corresponding hemispheres. Rotating along this equator, we can make the tilings of the hemispheres align in a non-edge-to-edge fashion. We can also take a hemisphere from each of two such distinct tilings to obtain non-edge-to-edge tilings of $S^2$, each subject to similar rotations, as shown in Figure \ref{2-hemispheres}. In fact, all of the edge-to-edge tilings of hemispheres come from the octahedron, the cuboctahedron or the icosidodecahedron, or they consist of a single tile. Thus by pairing the various options with themselves or other ones, we obtain ten families of these tilings, keeping in mind that we could consider a single-tile hemisphere to have any number of sides $n \geq 3$.

\begin{figure}[htpb]
\begin{center}
\includegraphics[width=.29\textwidth]{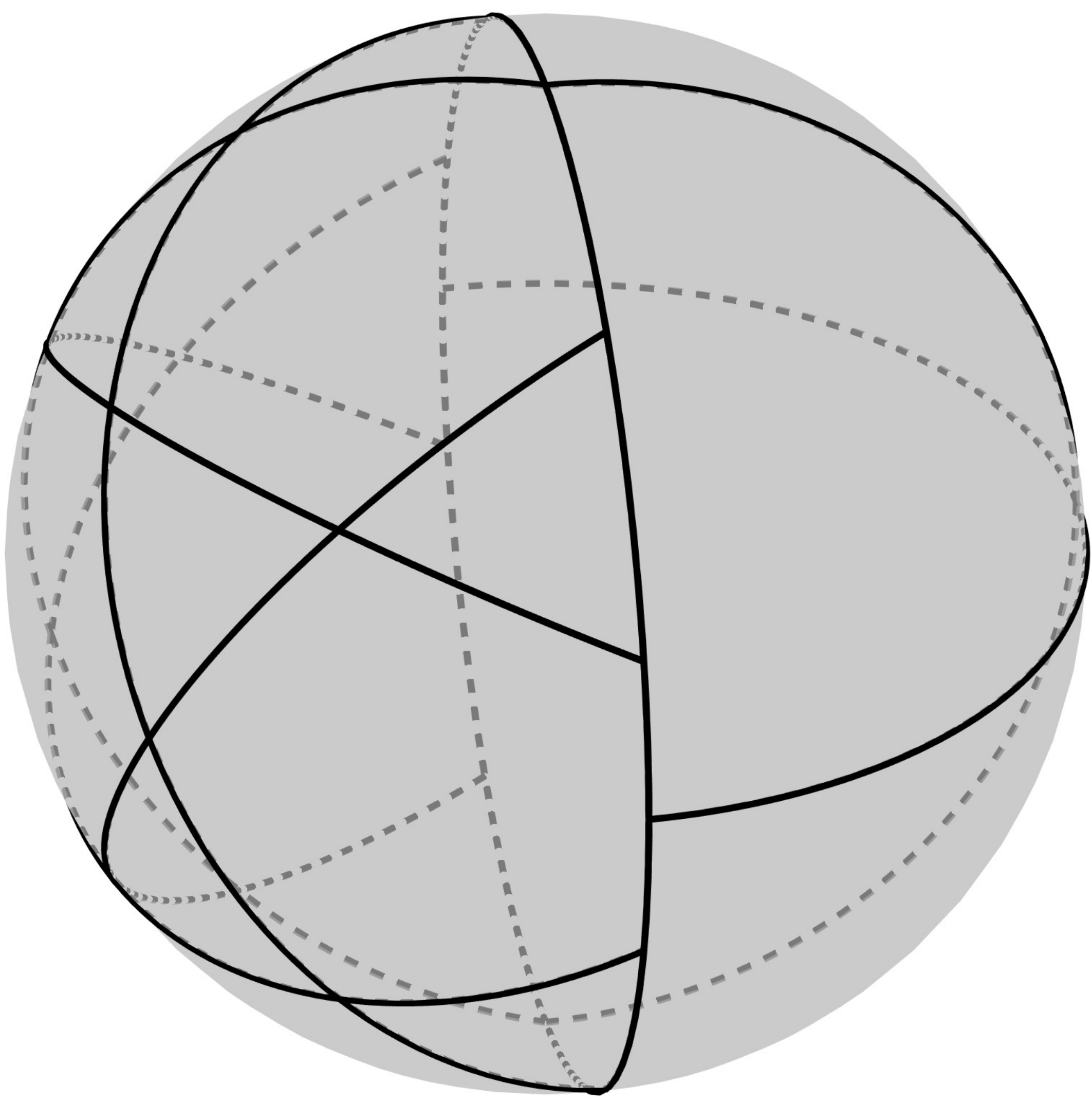}
\hspace{5mm}
\includegraphics[width=.29\textwidth]{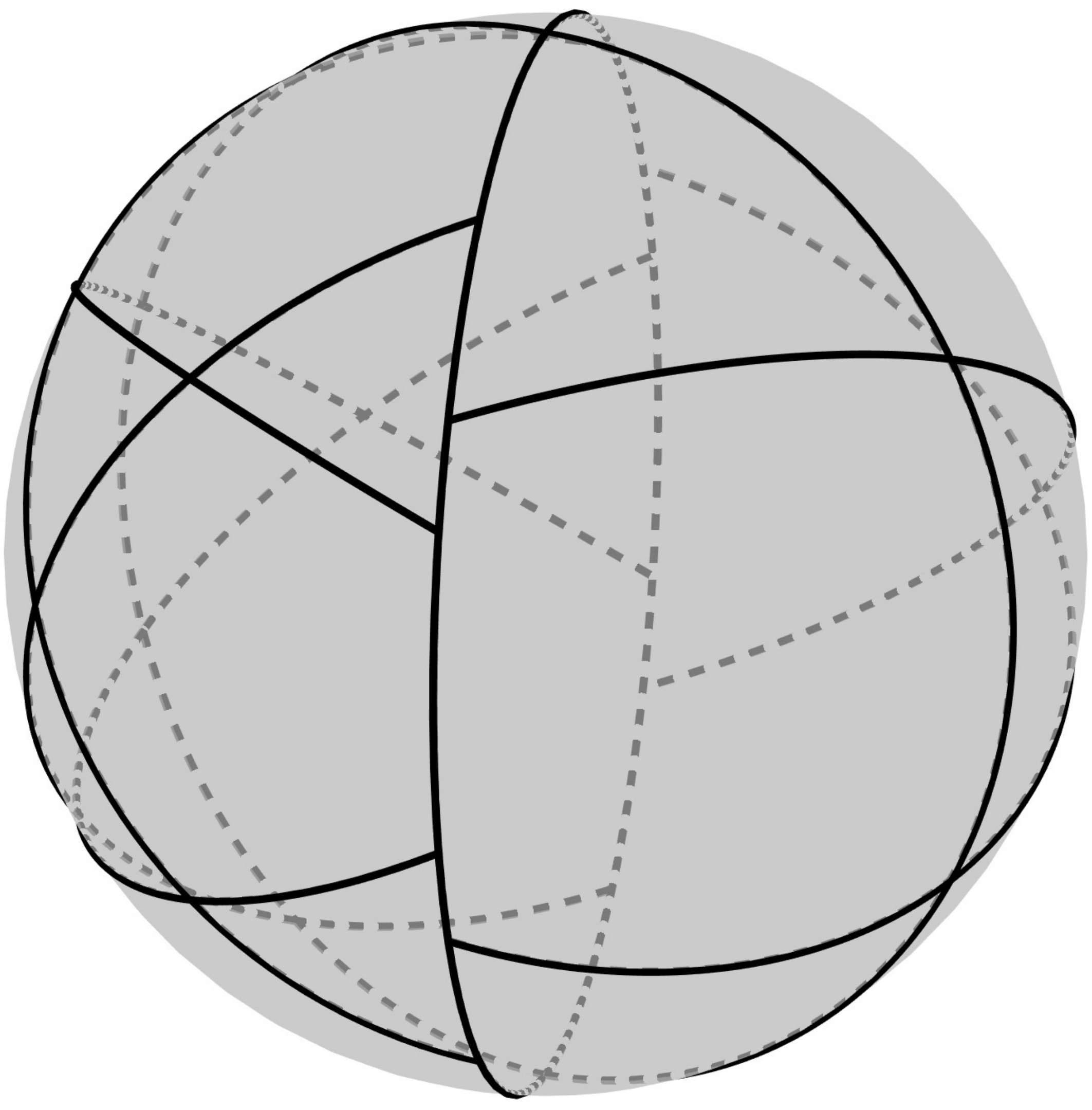}
\hspace{5mm}
\includegraphics[width=.29\textwidth]{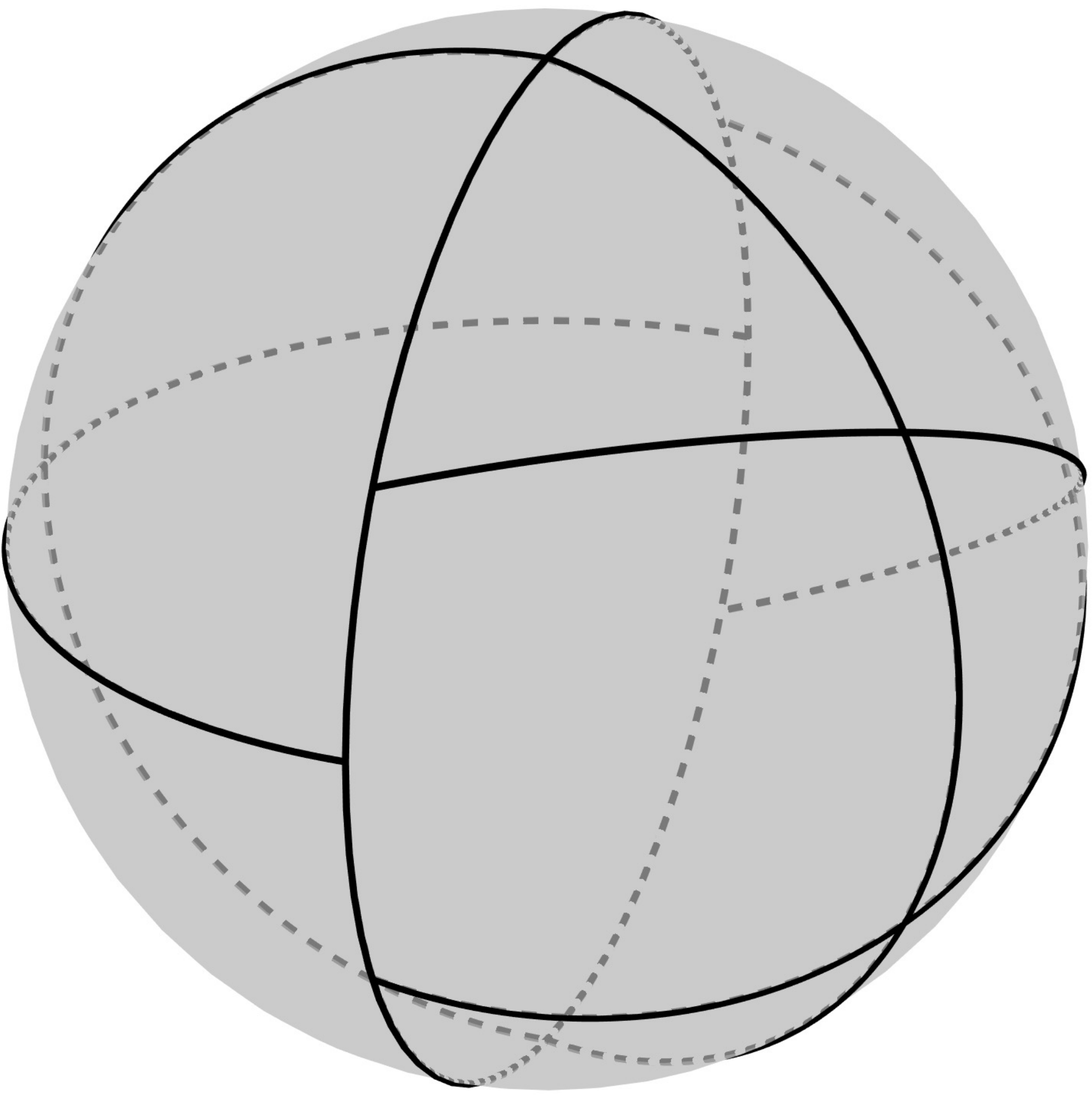}
\caption{Some examples of the 2-hemisphere tilings.}
\label{2-hemispheres}
\end{center}
\end{figure}

The third class of examples are the {\bf lunar tilings}.  From the Archimedean tilings, we can consider the edge-to-edge patches that designate lunes on $S^2$. There are five such lunes, as in Figure \ref{fig:bigontilings}. 

\begin{figure}[htpb]
\begin{center}
\includegraphics[scale=.5]{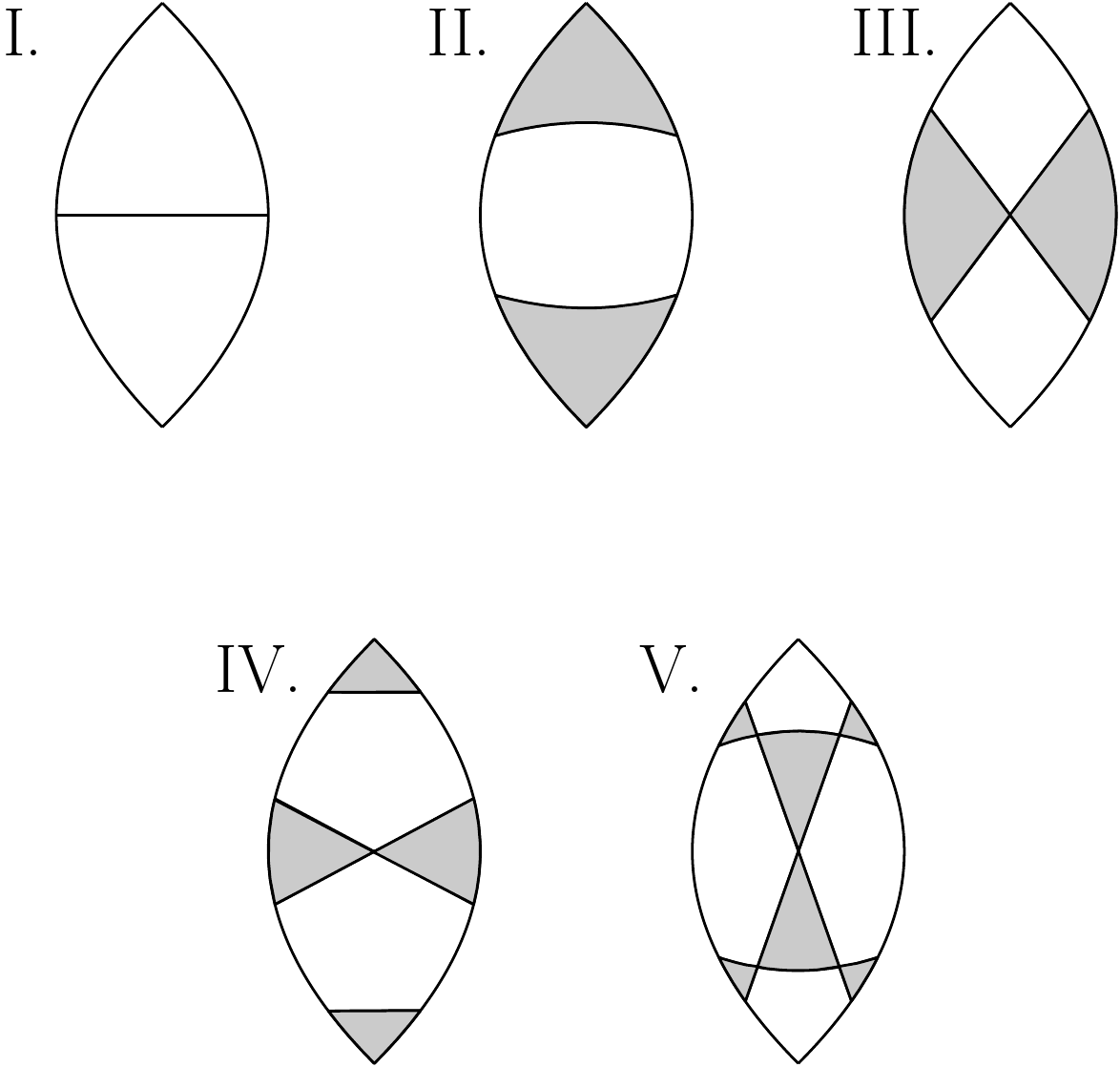}
\caption{Lunes coming from the dodecahedron, the cuboctahedron and the icosidodecahedron.}
\label{fig:bigontilings}
\end{center}
\end{figure}

The first bigon is called a Type I bigon and comes from two triangles in the octahedron tiling. Two of the bigons, denoted Type II and Type III bigons,  are from the cuboctahedron tiling. Two bigons, denoted Type IV and Type V, are from the icosadodecahedron tiling. 

In the case of Type II and Type IV bigons, we can glue a collection of copies of the same bigon type together in a non-edge-to-edge manner in a cycle that leaves space for a regular polygon at the top and bottom of the sphere, called the {\bf polar polygons}. In the Type II case, the polygon at the top can be a triangle or pentagon of fixed size. In the Type IV case, the polygon at the top and bottom can be either a triangle or square of fixed size.  This yields four possible lunar tilings, as in Figure \ref{lunartilings}. Note that unlike the preceding cases, these are not families, but rather are rigid.

\begin{figure}[htpb]
\begin{center}
\hspace{8mm} \includegraphics[width=.29\textwidth]{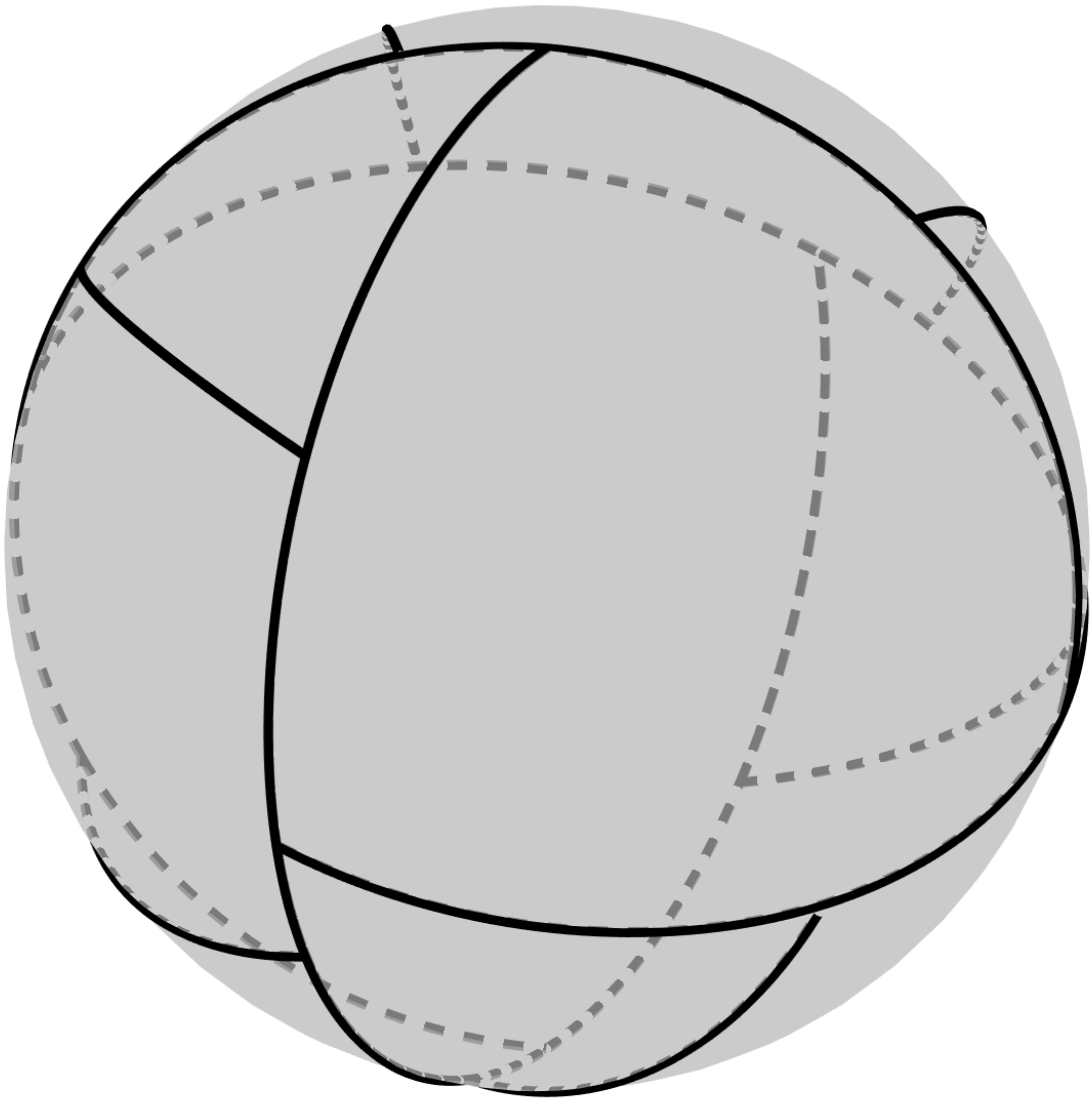}
\hspace{10mm}
\includegraphics[width=.29\textwidth]{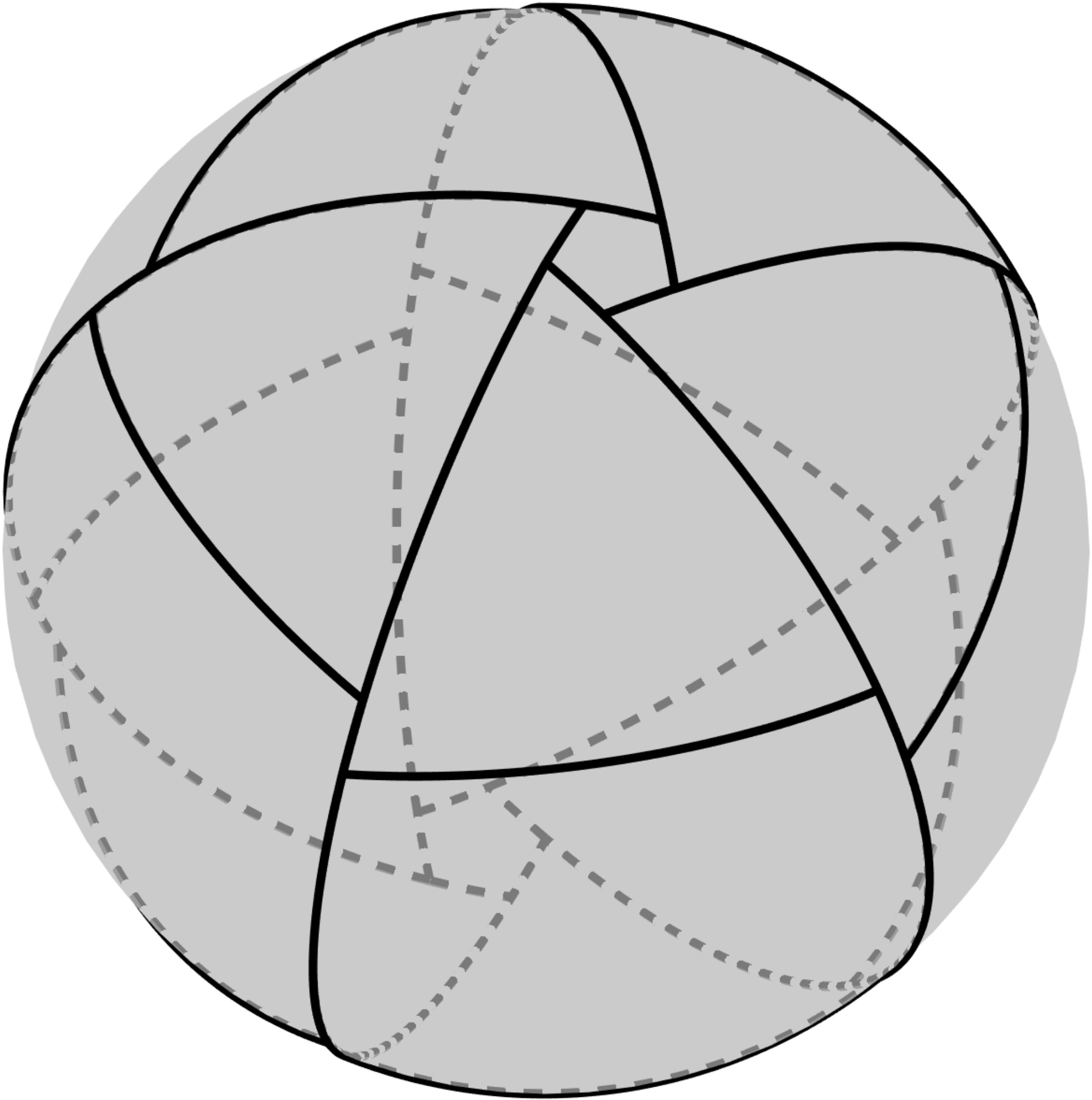} \vspace{7mm} \newline
\includegraphics[width=.29\textwidth]{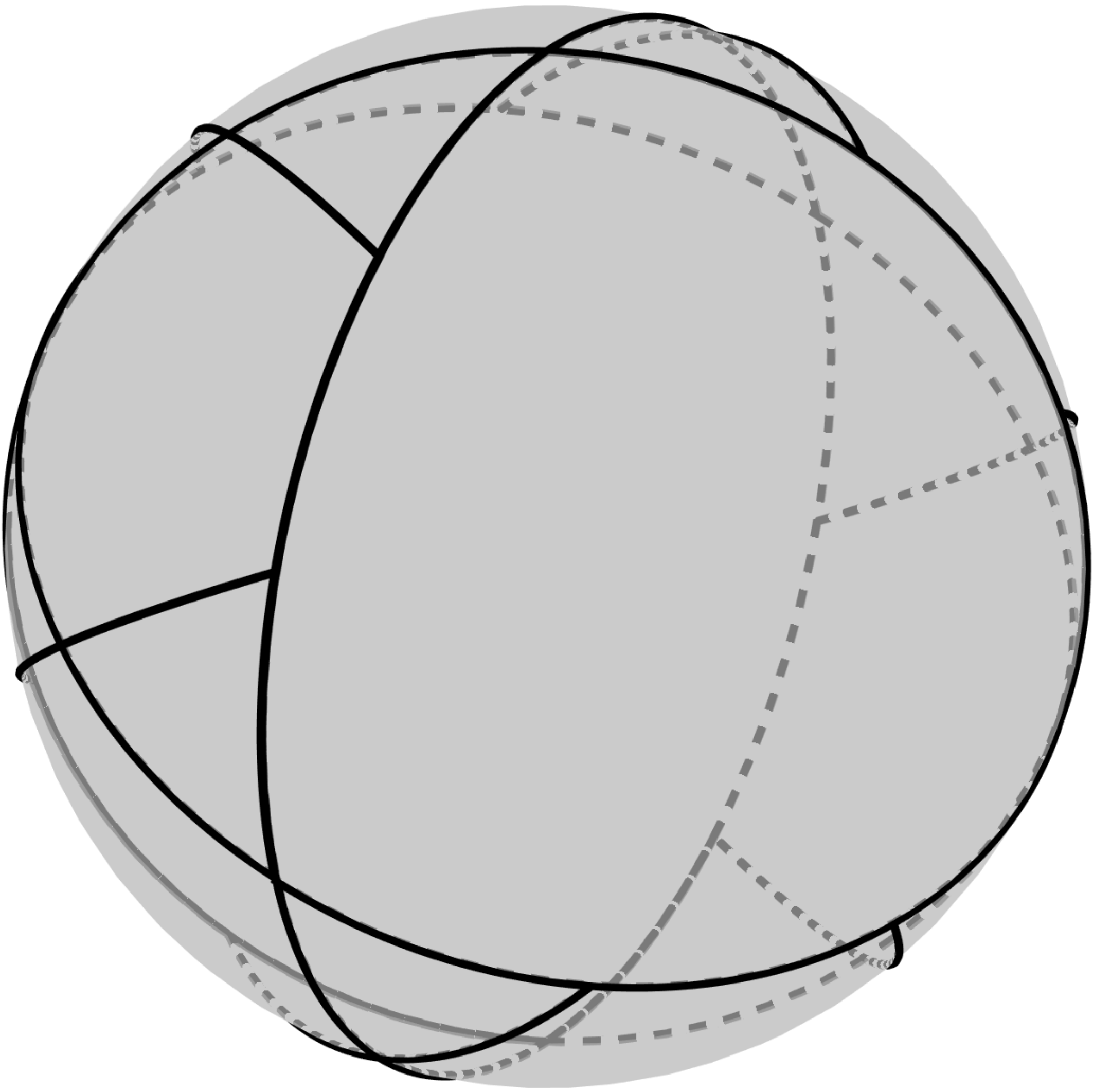} \hspace{10mm}
\includegraphics[width=.29\textwidth]{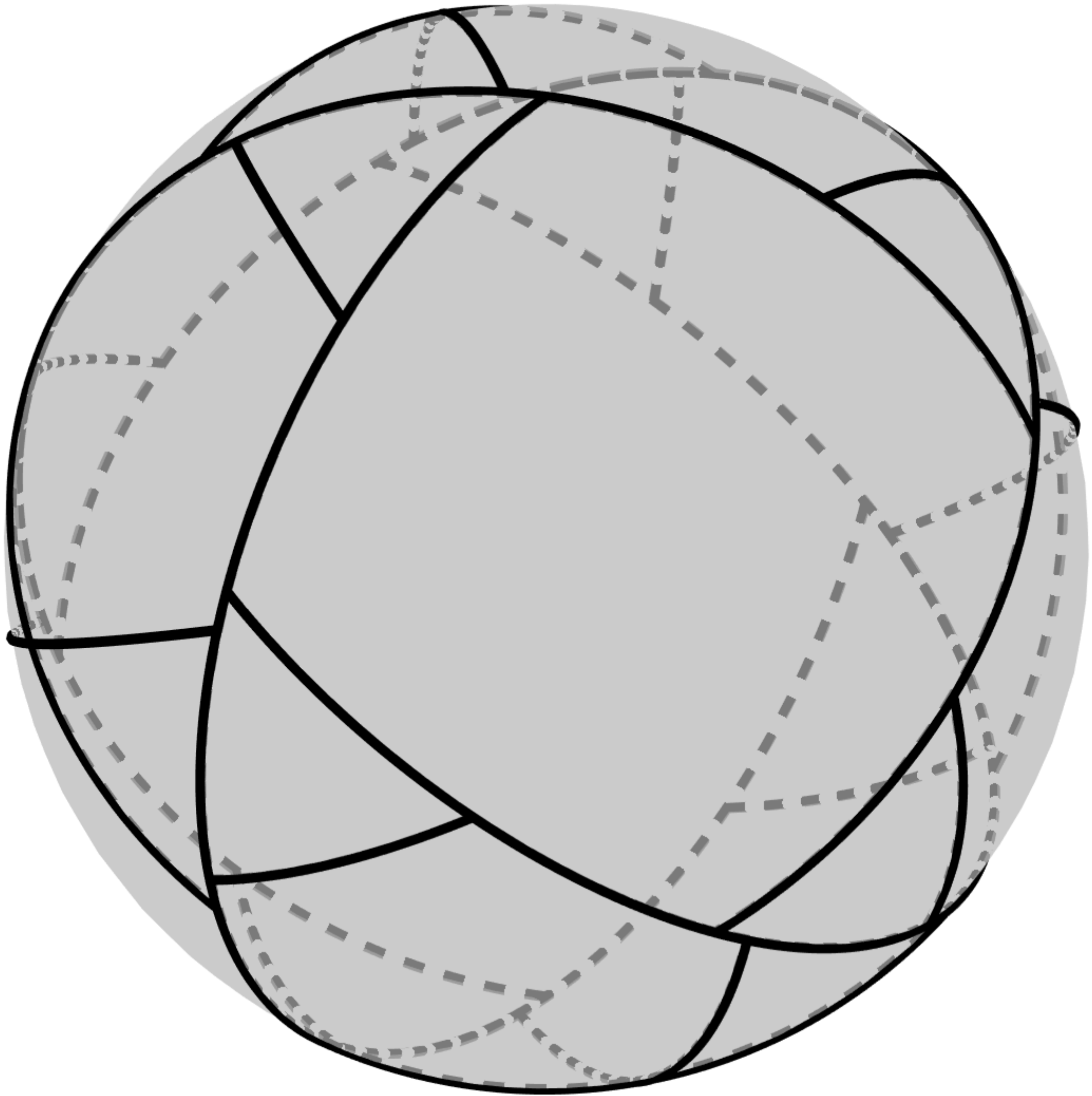}
\caption{The lunar tilings.}
\label{lunartilings}
\end{center}
\end{figure}

The fourth class consists of three rigid tilings that we call the {\bf sporadic tilings}. Each is obtained by gluing together edge-to-edge four lunes of the various types appearing as in Figure \ref{fig:bigontilings}, all sharing the north pole and south pole as vertices. We can glue these bigons together in the following cyclic orders to obtain non-edge-to-edge tilings of the sphere as in Figure \ref{sporadic}:

\begin{enumerate}
    \item I-II-I-III
    \item I-IV-I-V
    \item II-IV-III-V
\end{enumerate}

\begin{figure}[htpb]
\begin{center}
\includegraphics[width=.29
\textwidth]{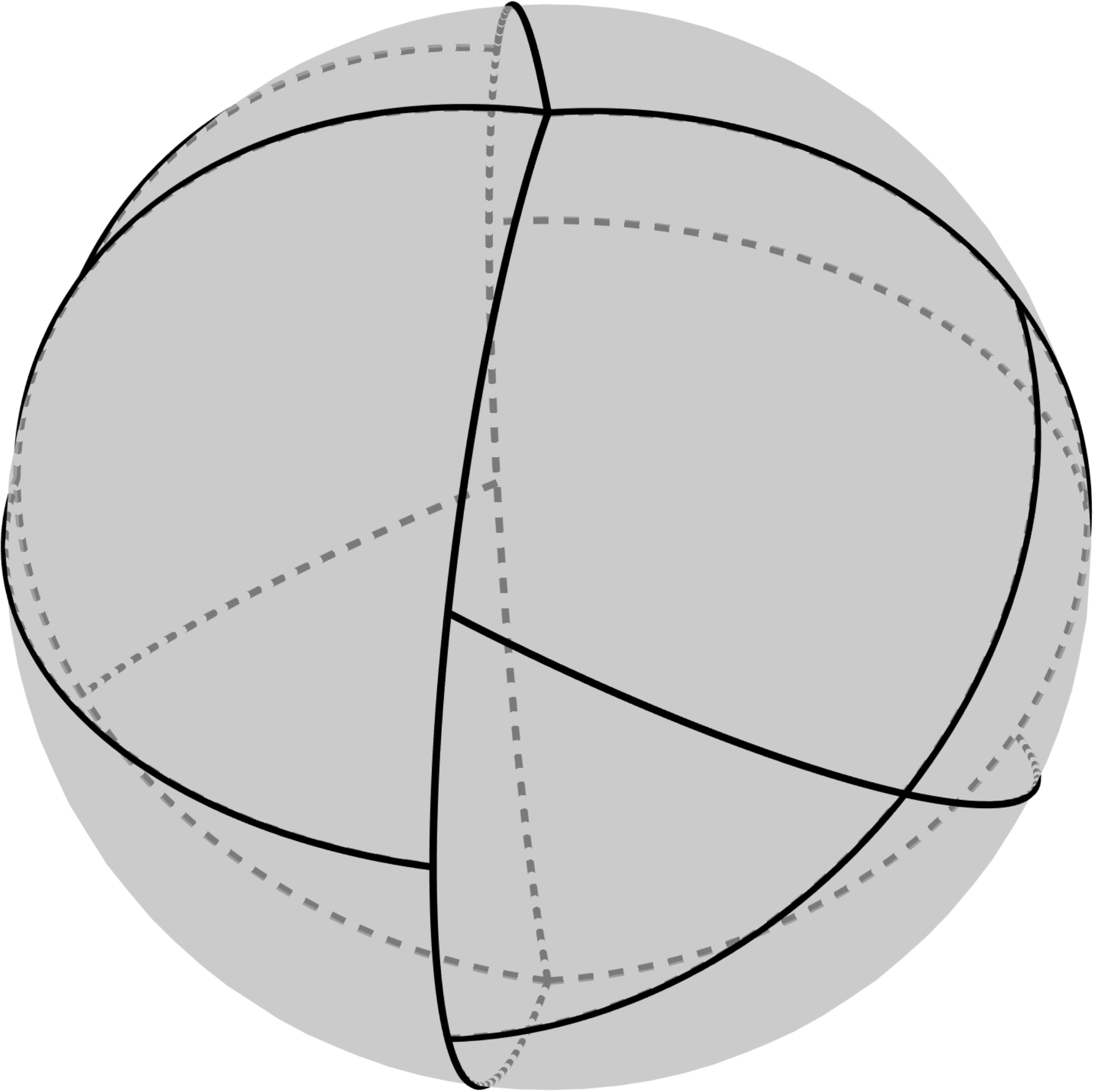} \hspace{5mm}
\includegraphics[width=.29
\textwidth]{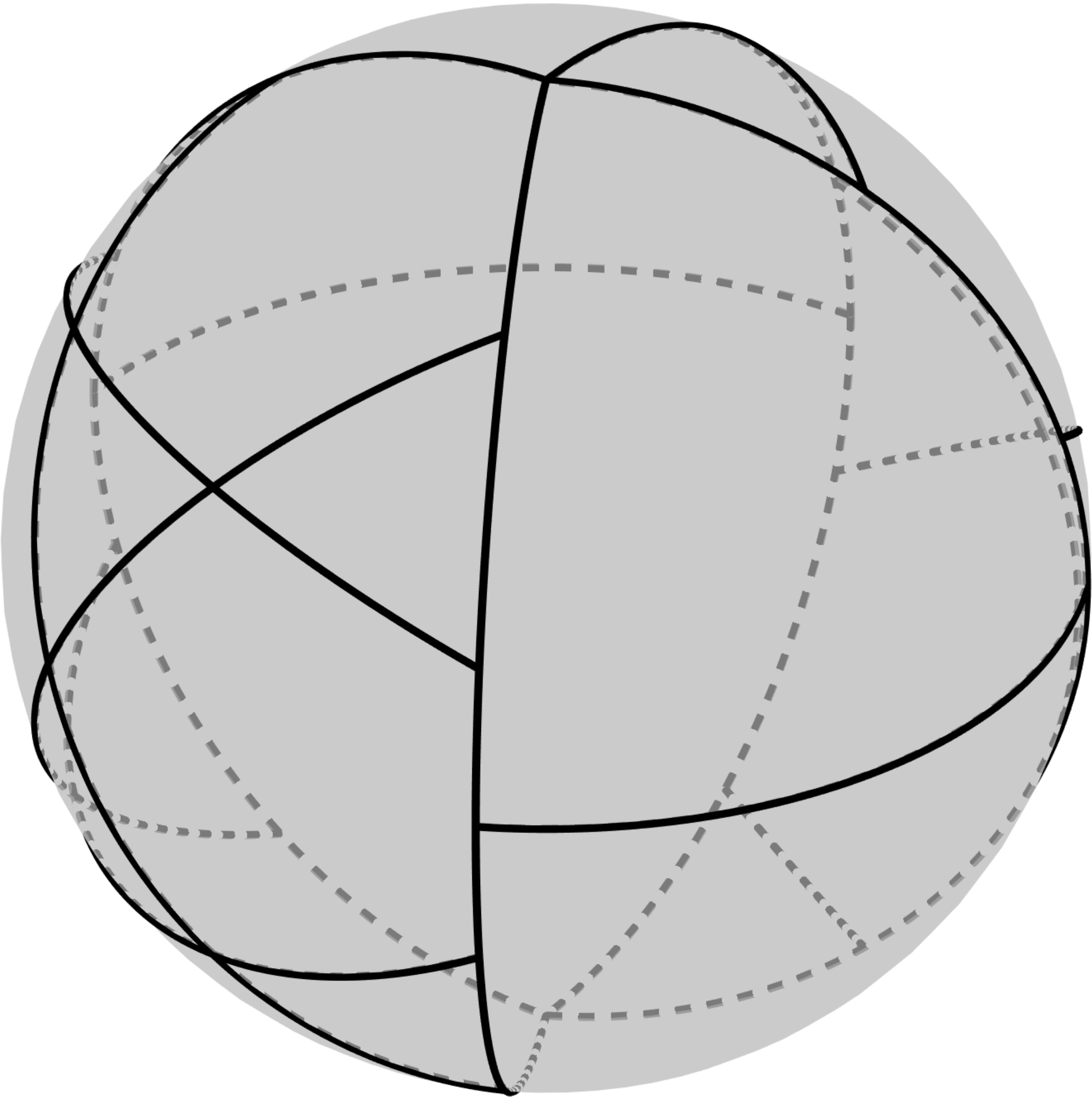} \hspace{5mm}
\includegraphics[width=.29
\textwidth]{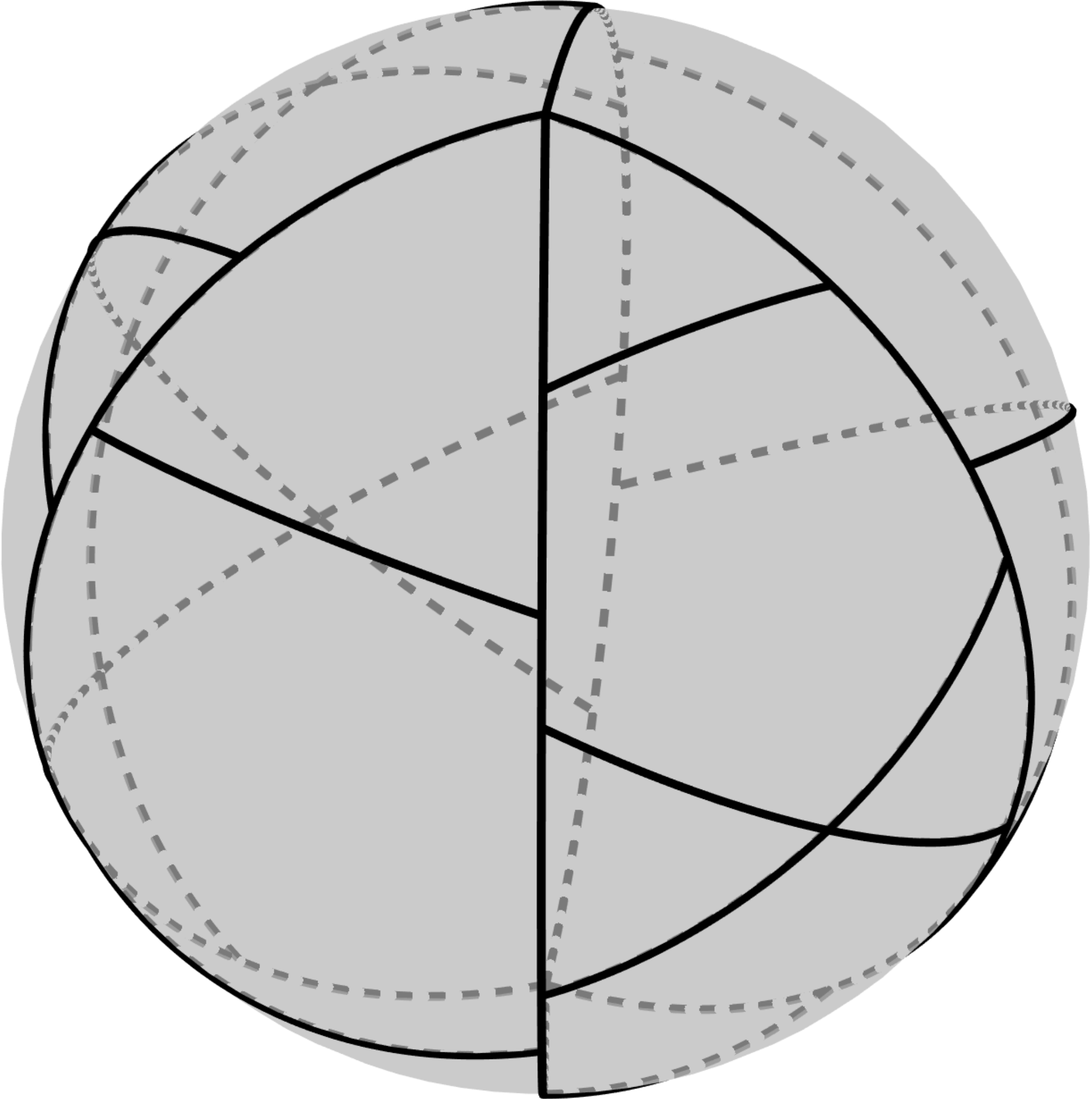}
\caption{The sporadic tilings.}
\label{sporadic}
\end{center}
\end{figure}

The fifth class of tilings are called {\bf composed tilings} and are obtained from the edge-to-edge tilings and the previously listed non-edge-to-edge tilings by composing sets of tiles that form a so-called {\bf magic triangle} as in Figure \ref{magictriangle}.  The resulting triangle has side-length exactly equal to $3\pi/5$.

\begin{figure}
    \centering
    \includegraphics[scale=1.5]{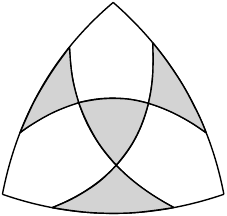}
    \caption{A magic triangle can be decomposed into smaller regular polygons.}
    \label{magictriangle}
\end{figure}

When such a collection of tiles is present in a tiling, we can compose the tiles into a single magic triangle and obtain a new non-edge-to-edge tiling. These sets of triangles exist in the icosidodecahedron, where there are up to four such collections of tiles that are disjoint on their interiors. We can choose to compose one of them.  Or we can compose two of them. The two magic triangles that result can either appear opposite one another or overlapping on their sides a length of $2 \pi /5$, or just touching at a vertex. The first of these last three options yields one of the lunar tilings.

     We can compose three magic triangles on the icosidodecahedron, either such that any two of them overlap a length of $2\pi/5$ on their boundaries, or such that two of them touch at just a vertex and the third, overlaps on its boundary with the boaundaries of the other two by $2\pi/5$. If we compose four, we get a Type I kaleidoscope tiling. So, this method applied to the icosidodecahedron generates five new non-edge-to-edge tilings which appear in Figure \ref{composedtilings}.

\begin{figure}
    \centering
    \includegraphics[width = .29\textwidth]{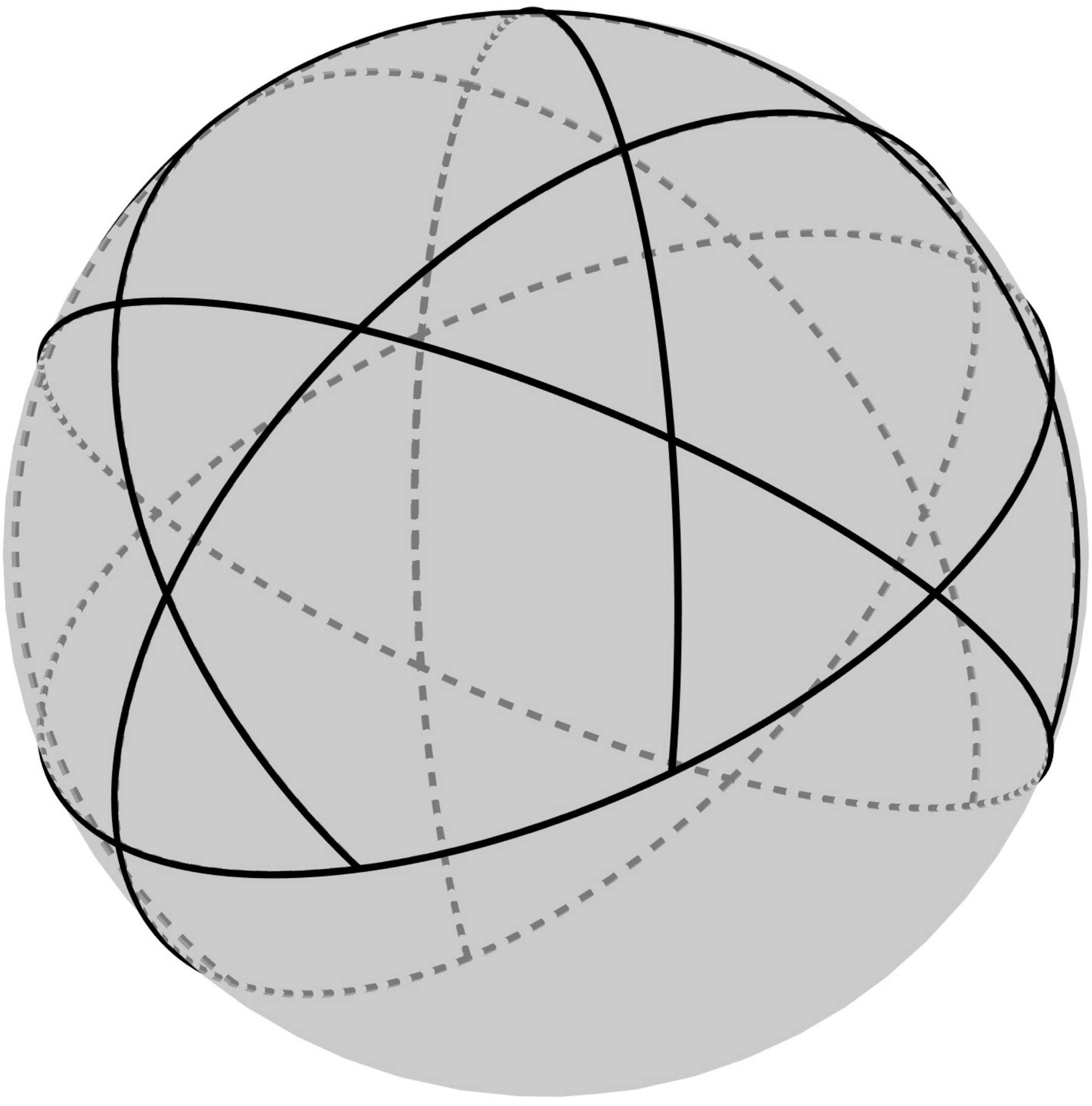}
    \hspace{5mm}
    \includegraphics[width = .29\textwidth]{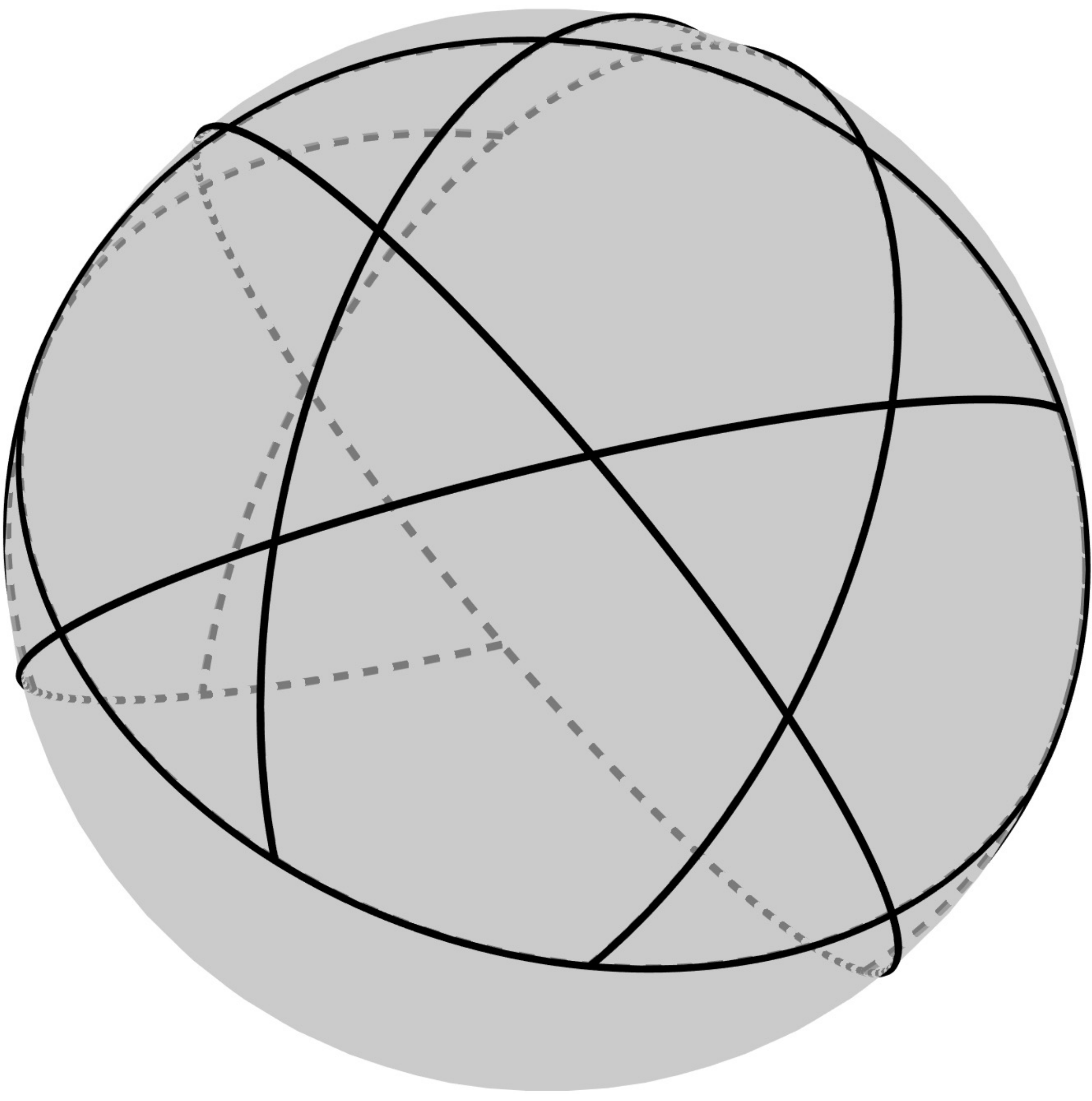}
    \hspace{5mm}
    \includegraphics[width = .29\textwidth]{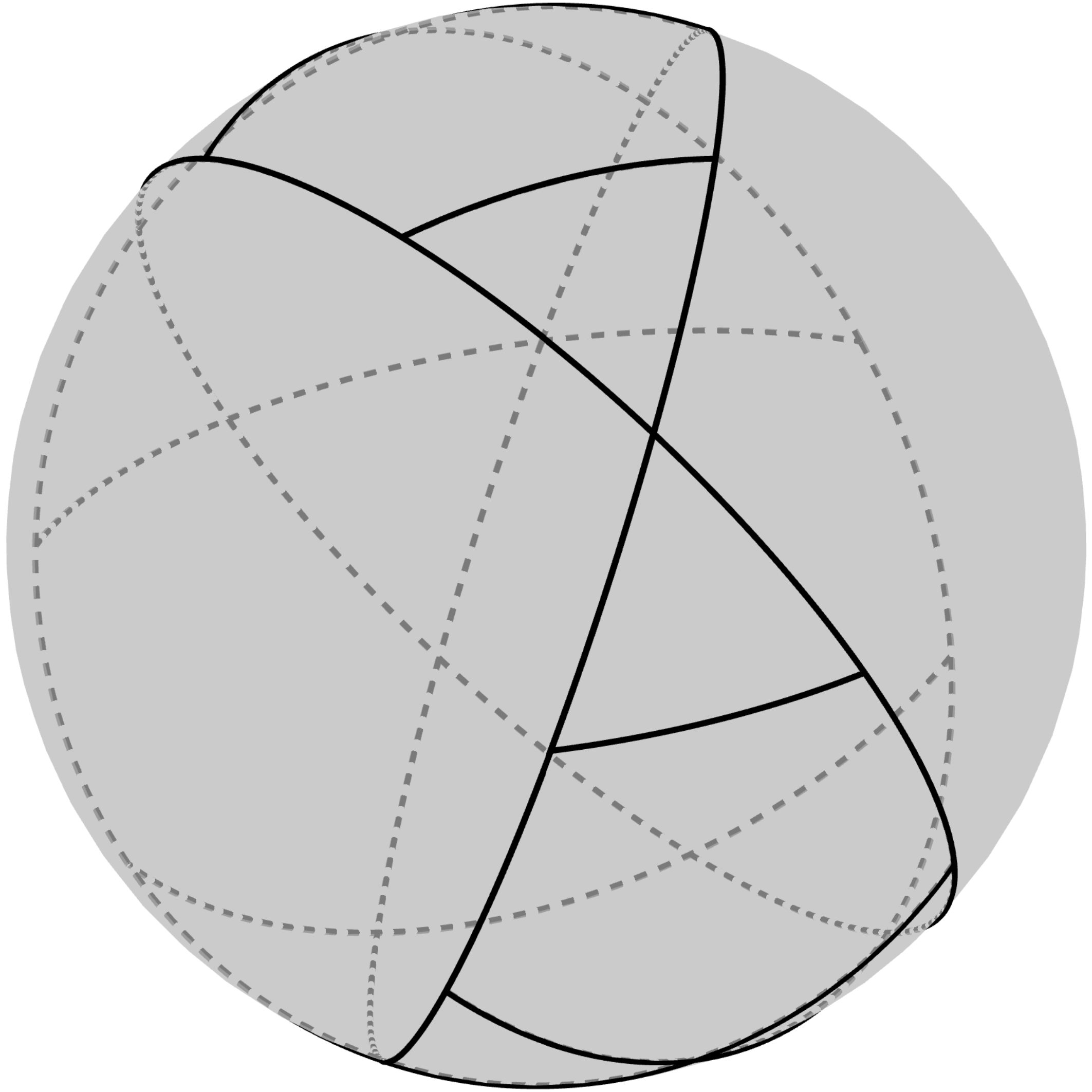}
    \vspace{7mm} \newline
    \includegraphics[width = .29\textwidth]{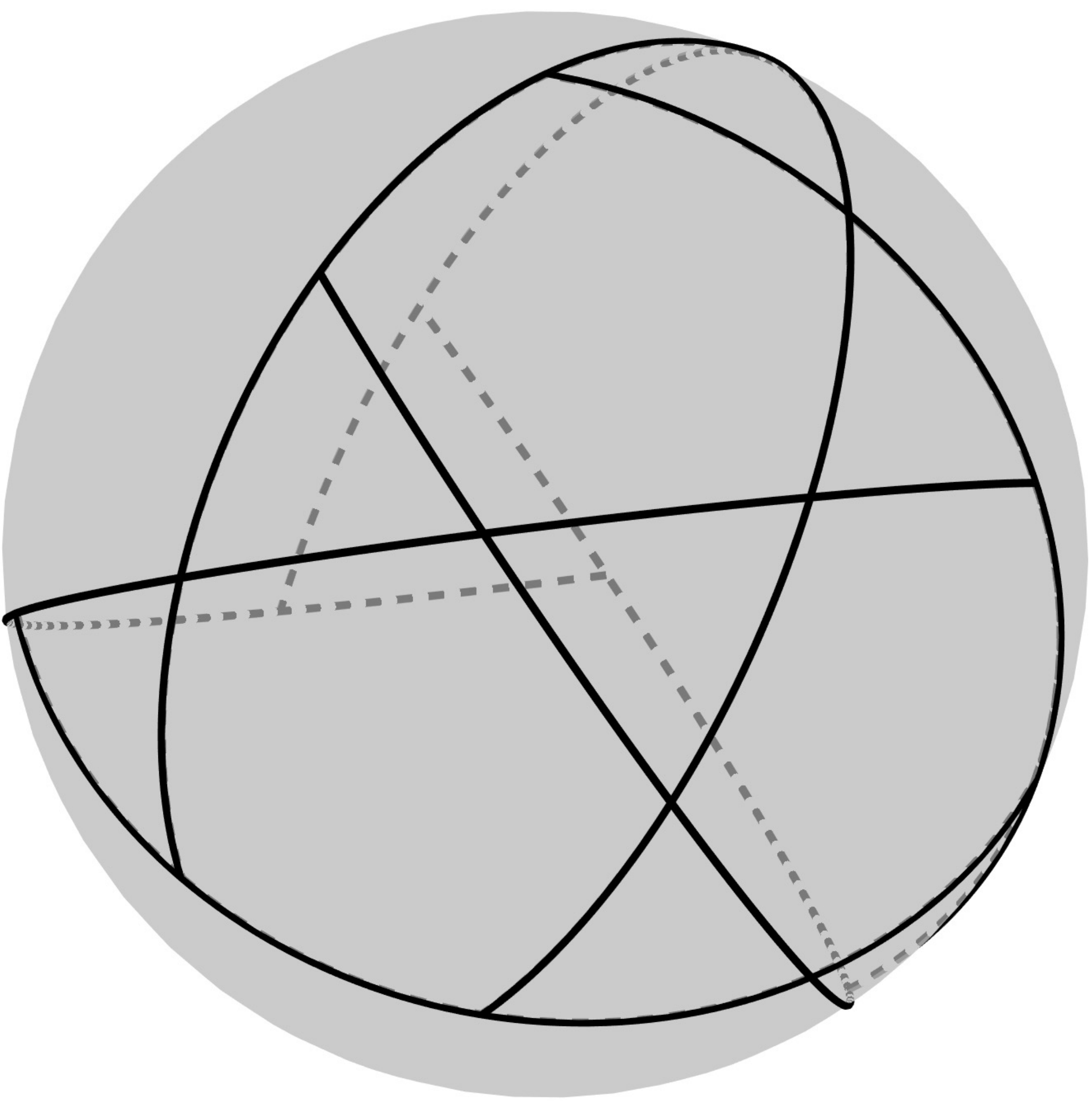}
    \hspace{5mm}
    \includegraphics[width = .29\textwidth]{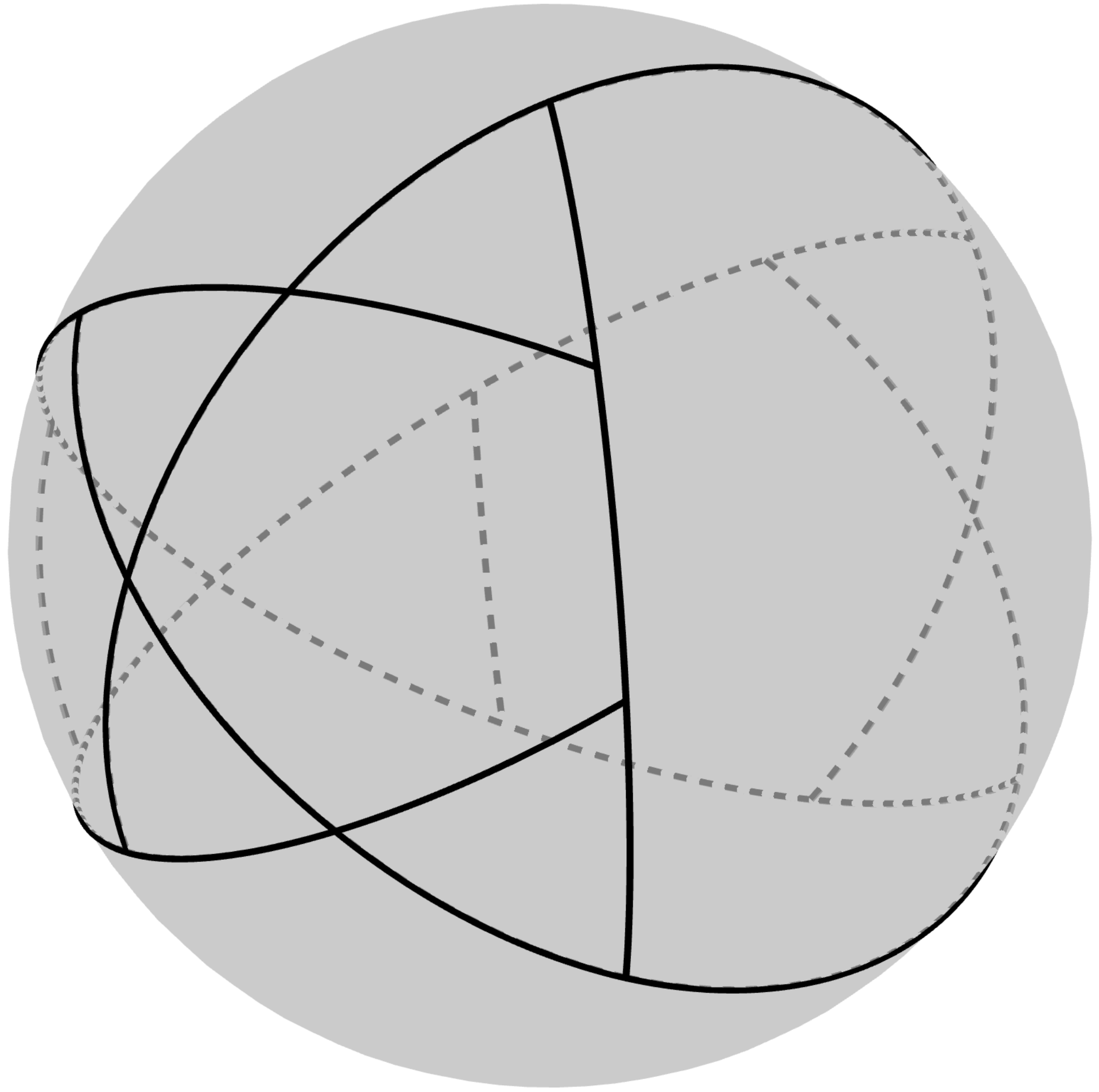}
    \caption{The five new composed tilings coming from the icosidodecahedron, one with one magic triangle, two with two magic triangles, and two with three magic triangles.}
    \label{composedtilings}
\end{figure}

In the case of the 2-hemisphere tilings, the hemisphere coming from the icosidodecahedron can have a single composed triangle, with one edge on its boundary. This generates one new hemispherical tiling, which can be paired with itself or any of the four previous hemispherical tilings.  Thus, we generate five additional composed tiling 2-hemisphere families. However, note that when we pair one copy of the new composed hemispherical tiling with an icosidodecahedral hemisphere tiling, the family contains the previously mentioned tiling obtained by composing one magic triangle in the icosidodecahedral tiling. When we take two copies of the new hemispherical tiling, the family contains all three of the tilings obtained by composing two magic triangles in the icosidodecahedral tiling, one of which was a lunar tiling. So the total count of tiling families only goes up by one.

If we choose a kaleidoscope tiling of Type I, with one set of triangles of side-length $3\pi/5$, then we can decompose a subset of these four triangles into pentagons and triangles to obtain non-edge-to-edge tilings. But these are the same as the ones we obtained by composing magic triangles in the icosidodecahedron. 
The other kaleidoscope tilings cannot be set to a size where the triangles are magic triangles, so they do not generate any such additional tilings.

Two of the sporadic tilings contain a lune that has a collection of tiles that can be composed into a magic triangle. So these are the last two composed tilings.

The final class consists of a single tiling that is a decomposed magic triangle together with the complementary tile on the sphere. This is the only non-edge-to-edge tiling that contains a polygon of angle greater than $\pi$. We call this the \textbf{magic triangle tiling} as in Figure \ref{magictriangletiling}.

\begin{figure}[htpb]
    \centering
    \includegraphics[width = .4\textwidth]{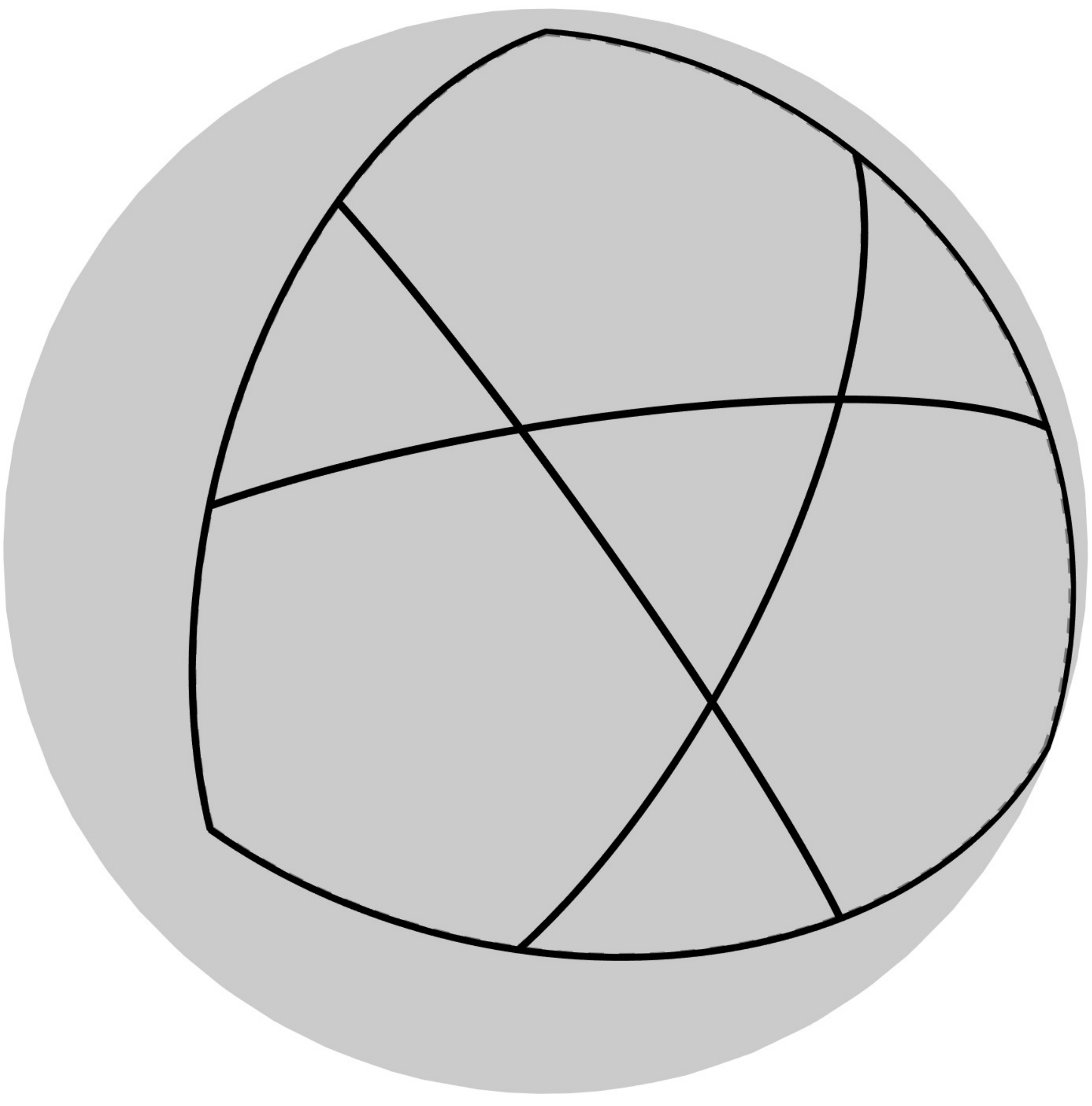}
    \caption{The magic triangle tiling.}
    \label{magictriangletiling}
\end{figure}

\begin{theorem}\label{maintheorem}

Every non-edge-to-edge tiling of the 2-sphere by regular $n$-gons with $n \geq 3$   is one of the following:
\begin{enumerate}
    \item A kaleidoscope tiling
    \item A 2-hemisphere tiling
    \item One of the four lunar tilings
    \item One of the three sporadic tilings
    \item One of the composed tilings
    \item The magic triangle tiling.
\end{enumerate}

\end{theorem}

This is a list of 31 distinct possibilities, each a family or rigid tiling, with certain tilings having been absorbed into families.
In this list, we have not distinguished between enantiomorphic pairs. That is, we do not distinguish between a tiling and its reflection. Some are equivalent to their reflection but some are not.

This paper is organized as follows. In Section 2, we outline definitions and proofs requisite for constructing non-edge-to-edge tilings of $S^2$. 
In Section 3, we consider tilings in which the smallest side-length tiles only appear as singletons, meaning they do not meet any other tile edge-to-edge. The resulting tilings are the kaleidoscope tilings and one of the lunar tilings. In Section 4, we consider tilings in which the smallest side-length tiles appear in edge-to-edge patches that consist of more than one tile. At this point, we assume all magic triangles have been decomposed. We then show that those patches are bigons, which include hemispheres, and show that the possible tilings are three of the lunar tilings, the 2-hemisphere tilings and the sporadic tilings. 

In Section 5, we consider the additional tilings we obtain by re-composing any possible magic triangles. Further, we determine that there is only one non-edge-to-edge tiling with a polygon of angle greater than $\pi$. We sum up in Section 6, pulling together the pieces to prove the main theorem.

In the appendix, we  classify all tilings of a regular polygon of angle less than $\pi$ by regular polygons. This is useful at various points throughout the paper, and may be of independent interest.

Thanks to George Hart, Casey Mann, Doris Schattschneider, and Jeff Weeks for helpful information and feedback.

\section{Background Tools and Definitions}

In this section we define relevant terms and prove necessary conditions on the vertices in non-edge-to-edge tilings of the sphere.

We need to be careful to distinguish between the vertices of a tile and the vertices of a tiling. Hence for a polygonal tile, we define its \textbf{sides} to be the geodesic segments which combine to create the boundary of the tile and its \textbf{corners} to be the points where the sides meet. 

On the other hand, a \textbf{vertex} of a tiling is a point where three or more tiles intersect. An {\bf edge} is the intersection of two tiles along a side of each and is bounded by vertices. 

If all tiles incident to a vertex have a corner there, then we refer to the vertex as a \textbf{full vertex}. Alternatively, if a vertex lies in the interior of the side of a polygonal tile, we call it a \textbf{half vertex}. We refer to the tiles whose corners meet at a half vertex as \textbf{supplementary tiles}, since their angles must be supplementary, and we refer to the tile which contains a half vertex on the interior of a side as an \textbf{edge-vertex tile}. Note that there can be at most two non-edge-vertex tiles at a half-vertex, assuming no bigon tiles, since the angle of a spherical regular $n$-gon for $n \geq 3$ is always greater than $\pi/3$. An example of supplementary tiles meeting at a half vertex appears in Figure \ref{halfvertex}.

\begin{figure}[htpb]
\begin{center}
\includegraphics[width=.3\textwidth]{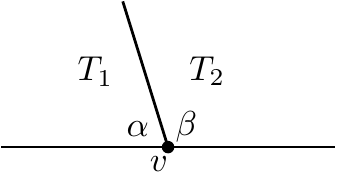}
\caption{Tiles $T_1$ and $T_2$ meeting at half vertex $v$. Note that the angles $\alpha$ and $\beta$ must add to $\pi$.}
\label{halfvertex}
\end{center}
\end{figure}

\begin{definition}
Let $T$ be a regular spherical polygon. Then the \textbf{side-length} of $T$, denoted $|T|$, is the length of each of the geodesics which make up the sides of $T$. The \textbf{area} of $T$ is denoted $a(T)$. The angle of $T$ is denoted $\angle(T)$.
\end{definition}

\begin{lemma}\label{sameanglepolygons} If $T_m$ and $T_n$ are regular spherical polygons with $m$ and $n$ edges, $m > n$ and their angles are identical, then $|T_m| < |T_n|$.
\end{lemma}

\begin{proof} If we cut each polygon $T_n$ of angle $\alpha$ into $n$ triangles, all sharing the central point, then the side length of $T_n$ satisfies $|T_n|= \cos^{-1}(\frac{\cos(2\pi/n) + \cos^2(\alpha/2)}{\sin^2(\alpha/2)})$. Since $\cos(2\pi/n)$ is increasing with $n$, and $\cos^{-1}$ is a decreasing function, then for fixed angle $\alpha$, $|T_n|$ decreases as $n$ increases.
\end{proof}


\begin{corollary}\label{samelengthpolygons} If $T_m$ and $T_n$ are regular spherical polygons with $m$ and $n$ sides, $m > n$, angles no more than $\pi$, and their side-lengths identical, then $\angle(T_m) > \angle(T_n)$.
\end{corollary}

\begin{proof}

Begin with $T_m$ and $T_n$ having identical angles so that $|T_m| < |T_n|$. As we increase $|T_m|$ to equal $|T_n|$, we necessarily increase $\angle (T_m)$ as well, since the length of the geodesic segment which constitutes a side of $T_m$ is equal to the angle at the origin of $S^2$. Therefore, $\angle (T_m) > \angle (T_n)$. 
\end{proof}




\begin{definition}
We use the term \textbf{singleton} to refer to a tile which meets no other tile in an edge-to-edge manner in a tiling.
\end{definition}

\begin{definition}
Let $T_1$ and $T_2$ be tiles which overlap in the interior of their sides $e_1$ and $e_2$ and either meet corner-to-corner at a vertex corresponding to these sides or with one playing the role of an edge-vertex tile at a corner of the other. We use the word \textbf{crevice} to refer to the angle of the unoccupied portion of the vertex at which the tiles intersect, as in Figure \ref{crevice}. Note that crevices are only defined for a partial tiling. 
If $T_1$ is an edge-vertex tile at a vertex, so the side of $T_1$ extends beyond the side of $T_2$, then we say that $T_1$ {\bf overhangs} $T_2$ at that vertex.
\end{definition}

\begin{figure}[htpb]
\begin{center}
\includegraphics[scale=.7]{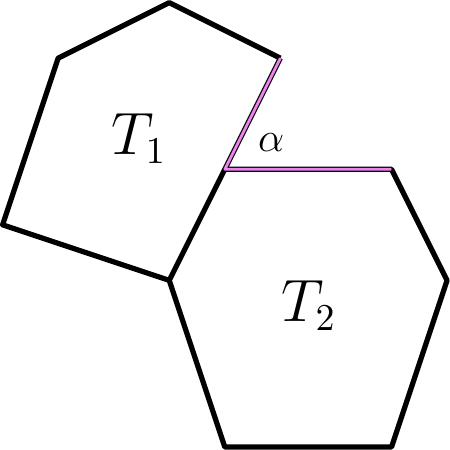}
\caption{A crevice of angle $\alpha$ formed by $T_1$ overhanging $T_2$.}
\label{crevice}
\end{center}
\end{figure}

\begin{definition}
Let $T_2$ be a tile which meets tile $T_1$ at corner $c_2$ and tile $T_3$ at corner $c_3$ such that $c_2 \neq c_3$ and $c_2$ and $c_3$ are endpoints of some side $e$ of $T_2$. Let $T_1$ have interior angle $\alpha$, $T_2$ have interior angle $\beta$, and $T_3$ have interior angle $\gamma$. If $T_2$ does not yet meet a different tile at side $e$, and if $\alpha + \beta > \pi$ and $\beta + \gamma > \pi$, then the unoccupied region at edge $e$ is referred to as a \textbf{well}. An example of a well appears in Figure \ref{fig:exWell}.
\end{definition}

\begin{figure}[htpb]
\begin{center}
\includegraphics[scale=.8]{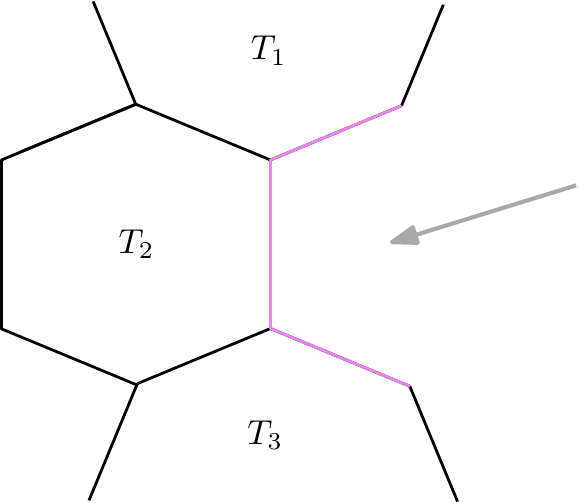}
\caption{A well formed by tiles $T_1$, $T_2$, and $T_3$}
\label{fig:exWell}
\end{center}
\end{figure}

\begin{definition}
A \textbf{patch} $\mathcal{P}$ is a set of at least two tiles on the sphere, none overlapping in their interiors, with a connected union such that we can get from any one tile  to any other by passing through the interiors of a  sequence of edges shared by the tiles. 
\end{definition}

Note that this differs from the usual definition that assumes the patch is topologically a disk. 

We note in Figure \ref{halfvertex} that when we have two tiles meeting at a half vertex, since they meet at their  corners along the side of another tile, the sum of their angles must be $\pi$. We will explicate what must hold about non-edge-to-edge tilings with what follows, though we first need the following well-known facts.

\begin{lemma}

Let $T$ be a spherical triangle with angles $\alpha$, $\beta$, and $\gamma$. Then the area of $T$ measured in radians is $\alpha + \beta + \gamma - \pi$.

\end{lemma}

\begin{corollary}
Let $T_1, T_2$ be two adjacent regular triangles with interior angles $\alpha, \beta$ such that $\alpha + \beta = \pi$. Then $a(T_1) + a(T_2) = \pi$.
\end{corollary}



\begin{corollary}
Let $\rho$ be the sum of the interior angles of a spherical triangle. Then $\rho > \pi$.
\end{corollary}


\begin{lemma}
Let $\alpha$ be the interior angle of a regular spherical $n$-gon, with $n \geq 3$. Then $\alpha > \frac{(n-2)\pi}{n}$.
\end{lemma}

\begin{proof}

In the Euclidean plane, triangulating a regular $n$-gon yields interior angles of $\frac{(n-2)\pi}{n}$. But a spherical $n$-gon has angles greater than the Euclidean case, which is the limit as the spherical regular $n$-gon shrinks in size. \qedhere
\end{proof}

\begin{lemma} Given a regular $n$-gon $T$, $0 < |T| \leq \frac{2\pi}{n}$.
\end{lemma}

\begin{proof} We can start with a regular $n$-gon centered at the north pole of the sphere that is arbitrarily small.  As we expand the $n$-gon, its side length will reach a maximum when all of its edges are along the equator, at which point its angles are all $\pi$. At that instant, its side-length is $2\pi/n$.
\end{proof}

We wish to note that on the sphere we can use polygonal tiles with fewer than three edges. These polygonal tiles are \emph{bigons} –– otherwise referred to as \emph{lunes} –– and they are bounded by edges formed from two intersecting great circles, with their corners at opposite points on the sphere. We include hemispheres as examples of bigons, and we allow the angle at the apex of the bigon to be any angle strictly between 0 and $2\pi$. In this paper, we exclude bigon tiles from our tilings and put off the classification of non-edge-to-edge tilings when bigons are allowed to a later paper. Although bigons decomposed into regular polygons play a key role in this paper, we will only allow regular spherical tiles of three  or more geodesic edges.

\begin{lemma}\label{tri.sq.pent}
At every half vertex, the only possible combinations for supplementary tiles are triangle-triangle, triangle-square, and triangle-pentagon.
\end{lemma}

\begin{proof}
Suppose we have three tiles $T_1$, $T_2$, and $T_3$ on $S^2$ such that all three are regular spherical polygons and such that $T_1$ and $T_2$ meet at a half vertex on an edge of $T_3$. Let $\angle(T_1)=\alpha$ and $\angle(T_2)=\beta$. Then $\alpha + \beta = \pi$.

To see that a half vertex may not contain an $n$-gon for $n \geq 6$, note that the interior angle of such an $n$-gon is strictly greater than $\frac{2\pi}{3}$. Since there does not exist a supplementary $m$-gon ($m \geq 3$) with interior angle strictly less than $\frac{\pi}{3}$, no such $n$-gon may exist at a half vertex, restricting to the use of only triangles, squares, and pentagons at a half vertex. However, for the same reason as before, a half vertex may also not contain two $k$-gons for $k \geq 4$ because the sum of their angles is strictly greater than $\frac{\pi}{2} + \frac{\pi}{2} = \pi$. 
Therefore, we see that at a half vertex we are restricted to the combinations triangle-triangle, triangle-square, and triangle-pentagon. For each of these combinations, their interior angles may sum to $\pi$.
\end{proof}


\begin{definition} A pair of tiles are called {\bf supp-same} if both tiles have the same side-length and their angles are supplementary.
\end{definition}

\begin{lemma}\label{suppsame} The only pairs of supp-same tiles that can appear in a valid tiling of the sphere are two right-angled triangles with side-length $\pi/2$, a triangle and square with side-lengths $\pi/3$ and a triangle and pentagon with side-lengths $\pi/5$. 
\end{lemma}

\begin{proof} By Lemma \ref{tri.sq.pent}, the only possibilities must be a triangle-triangle pair, a triangle-square pair and a triangle-pentagon pair. In all cases, both tiles must have angle strictly less than $\pi$. In this angle range,  the side-length of a tile grows monotonically with the angle.

In the case of two triangles, the fact their side-length is the same implies their angles are the same. Since their angles are supplementary, both have angle $\pi/2$. 

For either a triangle-square pair or a triangle-pentagon pair, increasing the angle of the first decreases the angle of the second. The fact that side-length grows with angle implies there is a unique pair that is supp-same. Note that the three examples given exist by considering the octahedral tiling, the cuboctahedral tiling and the icosadodecahedral tiling.
\end{proof}

\begin{definition} A supp-same pair that is two triangles, each of side-length $\pi/2$ is called a {\bf tri-tri pair}. A supp-same pair that is a triangle and square, each of side-length $\pi/3$ is called a {\bf tri-square pair}. And a supp-same pair that is a triangle and pentagon, each of side-length $\pi/5$ is called a {\bf tri-pent pair}. Any tile congruent to some tile in these pairs is known as a {\bf special tile}. 
\end{definition}




\section{Minimal length tiles appearing only as singletons}\label{singletonsection}


In this section, we consider the case of tilings such that all smallest side-length tiles appear as singletons, never sharing a full side with another tile of smallest side-length.

We  first consider the possibilities for such a singleton  $T$. Any tile glued to a side of $T$ is either the same side-length or longer. Since $T$ is a singleton of smallest side length, 
the edge of an adjacent tile $T'$ must overhang a corner of $T$, creating a half-vertex at that corner. Since $T$ is in a half-vertex, it can only be a triangle, square, or pentagon to satisfy Lemma \ref{tri.sq.pent}. The following lemma explains how $T$ and $T'$ must be positioned relative to one another. 

\begin{lemma} \label{cornersmatch} Let $T$ be a minimal length singleton and let $T'$ intersect $T$ at interior points of sides on each. Then $T$ must be glued to $T'$ such that one corner of $T'$ touches one corner of $T$.
\label{meetAtCorner}
\end{lemma}

\begin{proof}
Suppose that we try to surround $T$ such that no corner of $T$ meets a corner of $T'$ so that $T$ lies along a side of $T'$ (see Figure \ref{fig:wrongengulf}). If we attempt to surround $T$, as we cover each side with some other polygon, we arrive at a crevice that can only be filled by either having tiles overlap, or by using a polygon with a side length the same as that of $T$ glued edge-to-edge to $T$. But this contradicts the fact that $T$ is a singleton. It follows that in order to obtain a non-trivial tiling, a corner of $T'$ must meet a corner of $T$.
\end{proof}

\begin{figure}[h]\label{fig:wrongengulf}
\centering
\includegraphics[scale=0.6]{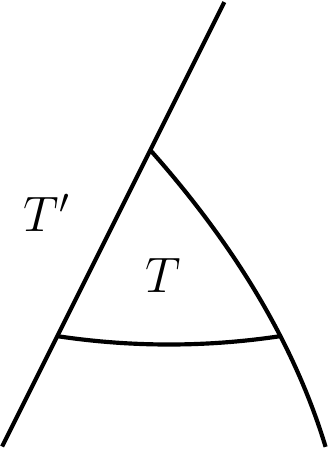}\caption{Improper engulfment}
\end{figure}

Thus we join $T$ with $T'$ such that one corner of $T$ lines up with one corner of $T'$. Let $T$ have interior angle measure $\alpha$ and $T'$ have interior angle measure $\beta$.



\begin{lemma}\label{singletonneighborssame} A minimal length  tile $T$ that only appears in singletons must always be surrounded by tiles, all of which are congruent.
\end{lemma}

\begin{proof}
Because $T$ is a singleton, Lemma \ref{cornersmatch} implies that each of its corners meets neighboring tiles at a half vertex. Those tiles can therefore be squares, pentagons, or larger triangles with angles $\beta$ which supplement $\alpha$.

Up to reflection, we can assume that the edges of the adjacent tiles extend out clockwise around $T$.
Each such tile must have angle $\beta$, supplementary to $\alpha$. If $T$ is a square or pentagon, then the fact every half-vertex must include a triangle implies all of the surrounding tiles are triangles with the same angle, and therefore are congruent. 

So, we need only prove it in the case $T$ is a triangle. Suppose there is more than one type of polygon represented in the three surrounding tiles. Then there must be a tile $T_2$ with fewer edges that is immediately clockwise from a tile $T_1$ of more edges. Together they create a crevice with angle $\alpha$. Since $T$ is the smallest side-length tile of angle $\alpha$ and a square or pentagon with angle $\alpha$ would have a smaller side-length, the only tile $T_3$ that can fill this crevice is a copy of $T$. But, as shown in Figure \ref{well1}, this creates a well with a side-length equal to that of $T$. This can only be filled using another tile of minimal side-length, contradicting the fact all copies of $T$ appear as singletons.

Thus, we see that $T$ must be surrounded by tiles all congruent to a single prototile $T'$. \qedhere

\end{proof}

\begin{figure}[h]
    \centering
    \includegraphics[scale=.97]{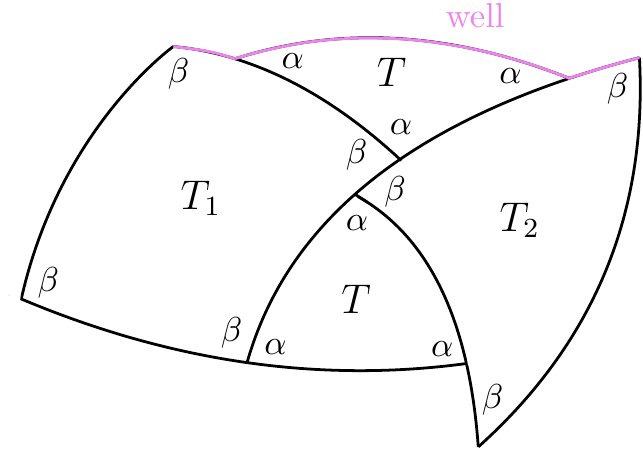}
    \caption{If tiles surrounding a singleton triangle are not congruent, we obtain a contradiction.}
    \label{well1}
\end{figure}

\subsection{Building from a Central Triangle}

\begin{lemma} If there is a minimum side-length prototile that only appears as a singleton triangle, then the tiling is either a triangle-triangle kaleidoscope tiling, a triangle-square kaleidoscope tiling or a triangle-pentagon kaleidoscope tiling. 
\end{lemma}

\begin{proof}

Let $T$ be a singleton triangle with angles $\alpha$. By Lemmas \ref{cornersmatch} and \ref{singletonneighborssame}, $T$ is surrounded by tiles, all congruent to a prototile $T'$, such that every tile shares a corner with $T$. We call this patch $\mathcal{P}$.

As we tile outward from this point, we may consider different cases depending on what $T'$ is. If $T'$ is a triangle with angle $\beta$, it follows that triangle $T$ is the only tile that can be placed at the second-layer half vertices since $T$ must have angles $\alpha < \frac{\pi}{2}$, in order to be smaller than the other triangle $T'$, and spherical pentagons and squares cannot have angle measures this small. In this case, we are left with three crevices, each of angle $\alpha$, and each of which is filled with one additional copy of $T$, yielding the tiling appearing in Figure \ref{tet-tet}. Note that since the area of a pair of the two sizes of triangles has area $\pi$, these four triangles congruent to $T$ and four triangles congruent to $T'$ do yield the total surface area of $4\pi$ for the sphere. 

If $T'$ is a square or pentagon, we must place $T$ at all second-layer half vertices since a triangle must always appear at a half vertex, and we know $T$ is the unique triangle whose angle $\alpha$ supplements that of $T'$. Each such copy of $T$, being itself a singleton, must be surrounded by copies of $T'$. 

In the case $T'$ is a square, this yields three additional squares at the next level,  touching one another on their boundaries to leave holes for three copies of $T$. The boundary of the resultant patch has three crevices of angle $\alpha$, yielding a hole for a final copy of $T$. The resulting tiling consists of  eight copies of $T$ and six copies of $T'$, with a total area of $8(3 \alpha - \pi) + 6(4\beta - 2\pi) = 4\pi$, as desired, and yielding a tiling as in Figure \ref{trisquaretiling}.

In the case $T'$ is a pentagon, we have a first layer of pentagons congruent to $T'$ surrounding $T$. The three crevices they generate require copies of $T$. Each such copy must be surrounded by pentagons, and each resulting crevice generated by those pentagons must be filled with a copy of $T$. We now have 10 triangle and 6 pentagons. Each of the six triangles on the boundary of the patch require a third pentagon to surround them, yielding a total number of 12 pentagons. These new pentagons leave gaps for an additional ten triangles, yielding a total of 20 triangles, and generating a total area of $20 (3\alpha - \pi) + 12(5\beta -3\pi) = 4\pi$. This yields a tiling as in Figure \ref{tripenttiling}.

\end{proof}

\subsection{Building from a Central Square}

\begin{lemma} If there is a minimum side-length prototile that only appears as a singleton square, then the tiling is a square-triangle kaleidoscope tiling. 
\end{lemma}

\begin{proof}
If our smallest tile $T$ appears only as a singleton square, then Lemma \ref{singletonneighborssame} implies it is surrounded by congruent tiles, and the need for a triangle at each half-vertex implies the surrounding tiles are all triangles congruent to a tile $T'$. Let $\mathcal{P}$ be the resultant patch.

If we add to the crevices of this patch another layer of four squares, they will all be congruent to $T$. Since we are assuming each square congruent to $T$ is a singleton, they must all be surrounded by three triangles. Adding those in leaves space for just one more square congruent to $T$, which by Theorem \ref{polygondecomposition} cannot be filled with smaller polygons. As in Figure \ref{centralsquarenicetiling}, we obtain a kaleidoscope tiling  with 6 squares and 8 triangles. 


\begin{figure}[h]
    \centering
    \includegraphics[scale=.35]{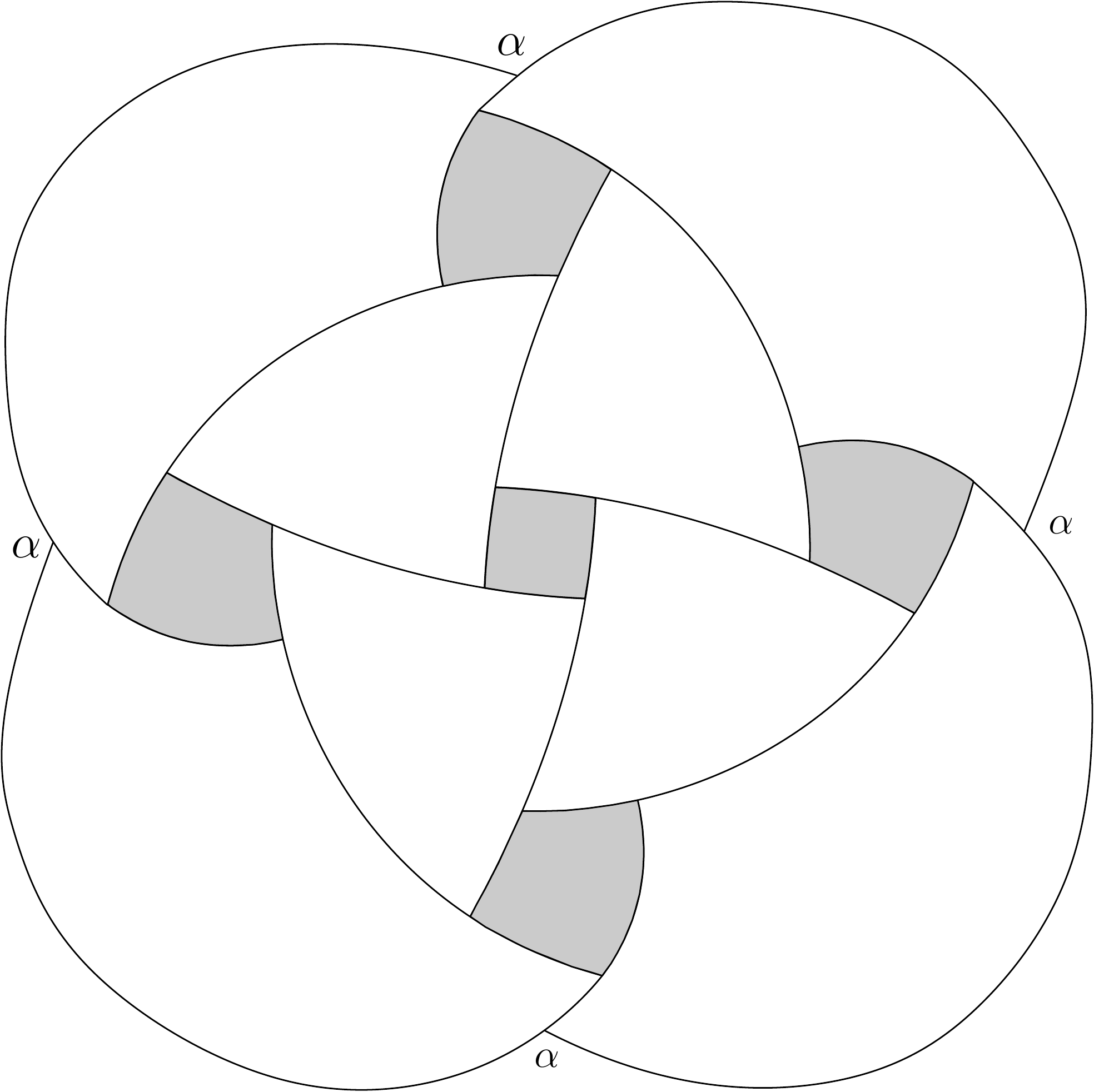}
    \caption{Adding a layer of squares to the initial patch around a singleton square yields a kaleidoscope tiling.}
    \label{centralsquarenicetiling}
\end{figure}

On the other hand, if we take the patch $\mathcal{P}$ consisting of the singleton $T$ and the four triangles surrounding it, all congruent to $T'$, and we add a triangle $T''$ to one of the crevices formed, then because $\alpha > \pi/2$ and hence $\alpha > \beta$, the new triangle is larger than the triangle $T'$. As in Figure \ref{centralsquarecontradiction}, this creates a well, which has angles $\alpha$ and base length $|T|$. Hence it must be filled with a copy of $T$. But this copy of $T$ touches both $T'$ and $T''$ along its sides,  a contradiction to Lemma \ref{singletonneighborssame}. 


\begin{figure}[h]
    \centering
    \includegraphics[scale=.3]{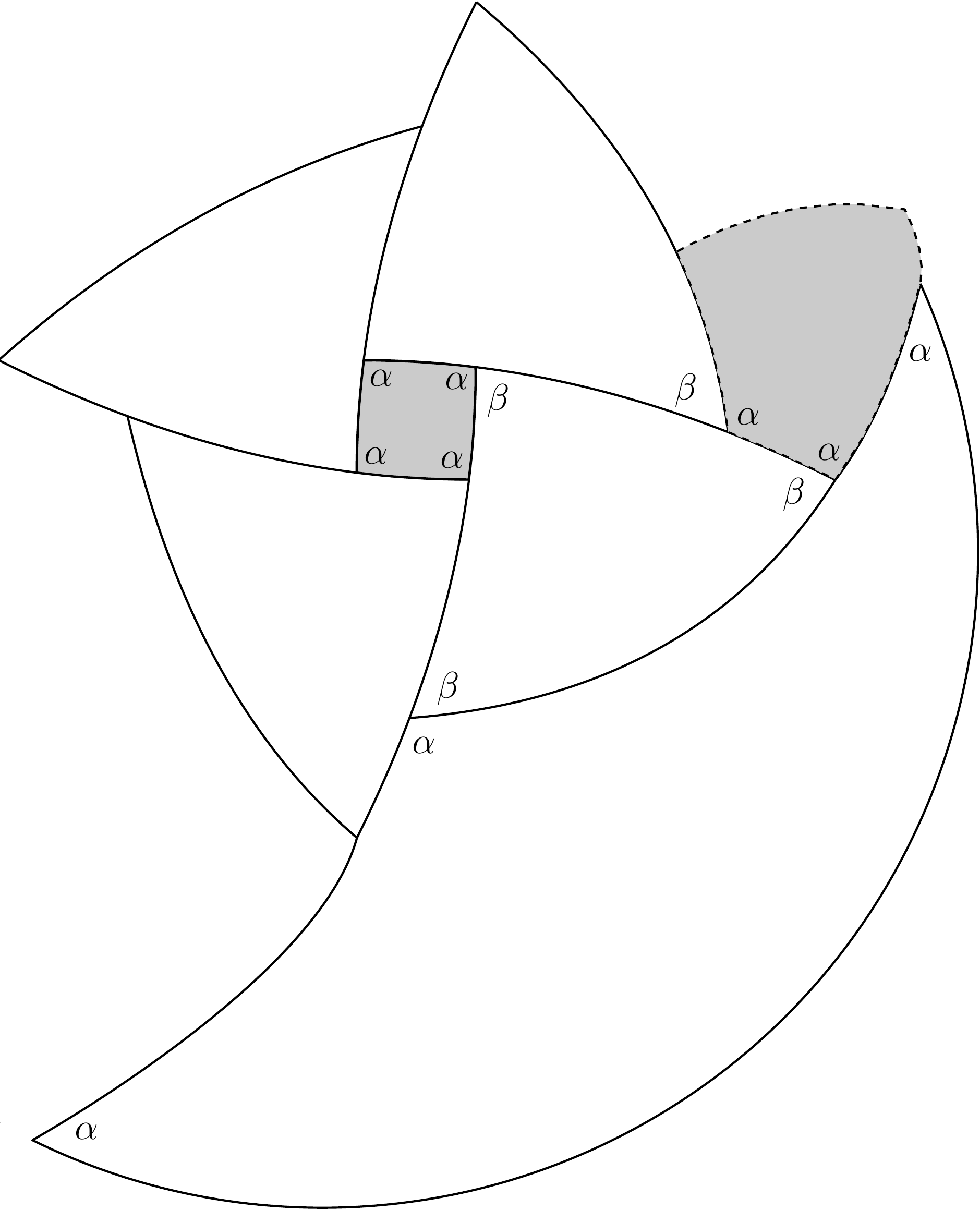}
    \caption{Adding a triangle to the initial patch around a singleton square yields a contradiction.}
    \label{centralsquarecontradiction}
\end{figure}

We now consider the case of adding a pentagon $T''$ to a crevice of the patch $\mathcal{P}$. But by Lemma \ref{sameanglepolygons},  $|T''| < |T|$, which contradicts the fact $|T|$ is minimal length.
\end{proof}

\subsection{Building from a Central Pentagon} \label{Building from a Central Pentagon}

\begin{lemma} If there is a minimum side-length prototile that only appears as a singleton pentagon, then the tiling is either a pentagon-triangle kaleidoscope tiling or a lunar tiling with polar pentagons and Type IV bigons. 
\end{lemma}

\begin{proof}
If our smallest tile $T$ appears only as a singleton pentagon, then all of the surrounding tiles must be triangles by Lemma \ref{singletonneighborssame}. Call this patch $\mathcal{P}$. See Figure \ref{fig:construct} for example.

\begin{figure}[h]
\centering
\includegraphics[scale=0.4]{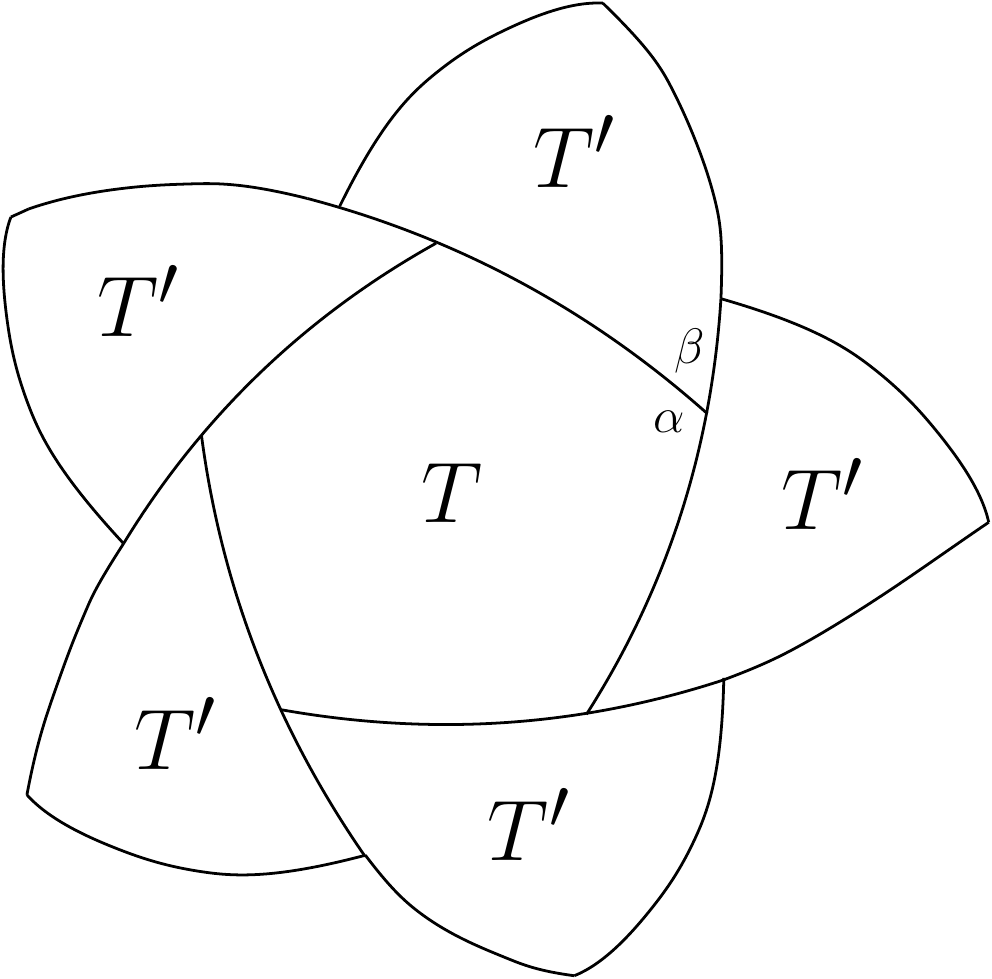}
\caption{Clockwise orientation of $T$ engulfed by copies of $T'$.}
\label{fig:construct}
\end{figure}

Building on $\mathcal{P}$, we see that we could insert another pentagon congruent to $T$ in each of the five crevices. if we did so, then since each of those must be surrounded by triangles congruent to $T'$, we would now have 15 triangles, as in Figure \ref{pentagoncentercase}(a). This creates five wells that can only be filled with pentagons congruent to $T$, as in Figure \ref{pentagoncentercase}(b). 

\begin{figure}
    \centering
    \includegraphics[scale=.55]{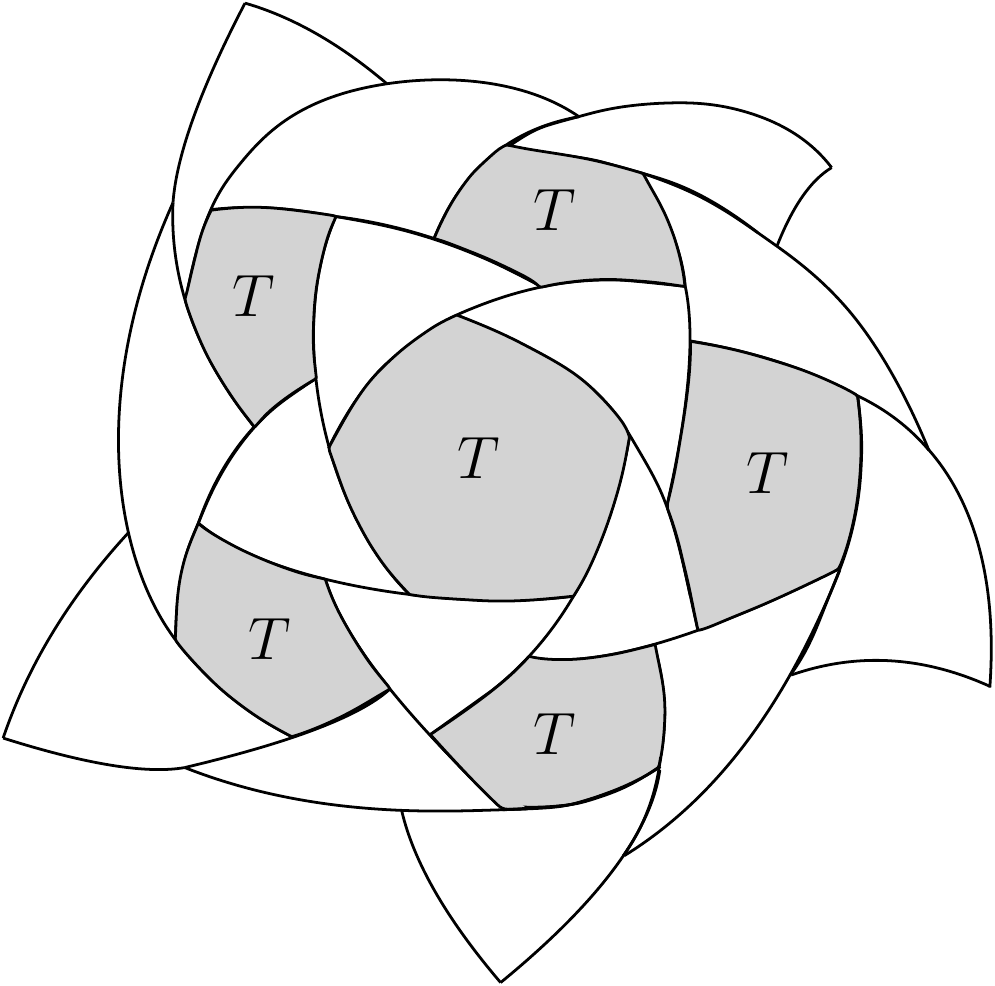} \hspace{4mm}
    \includegraphics[scale=.5]{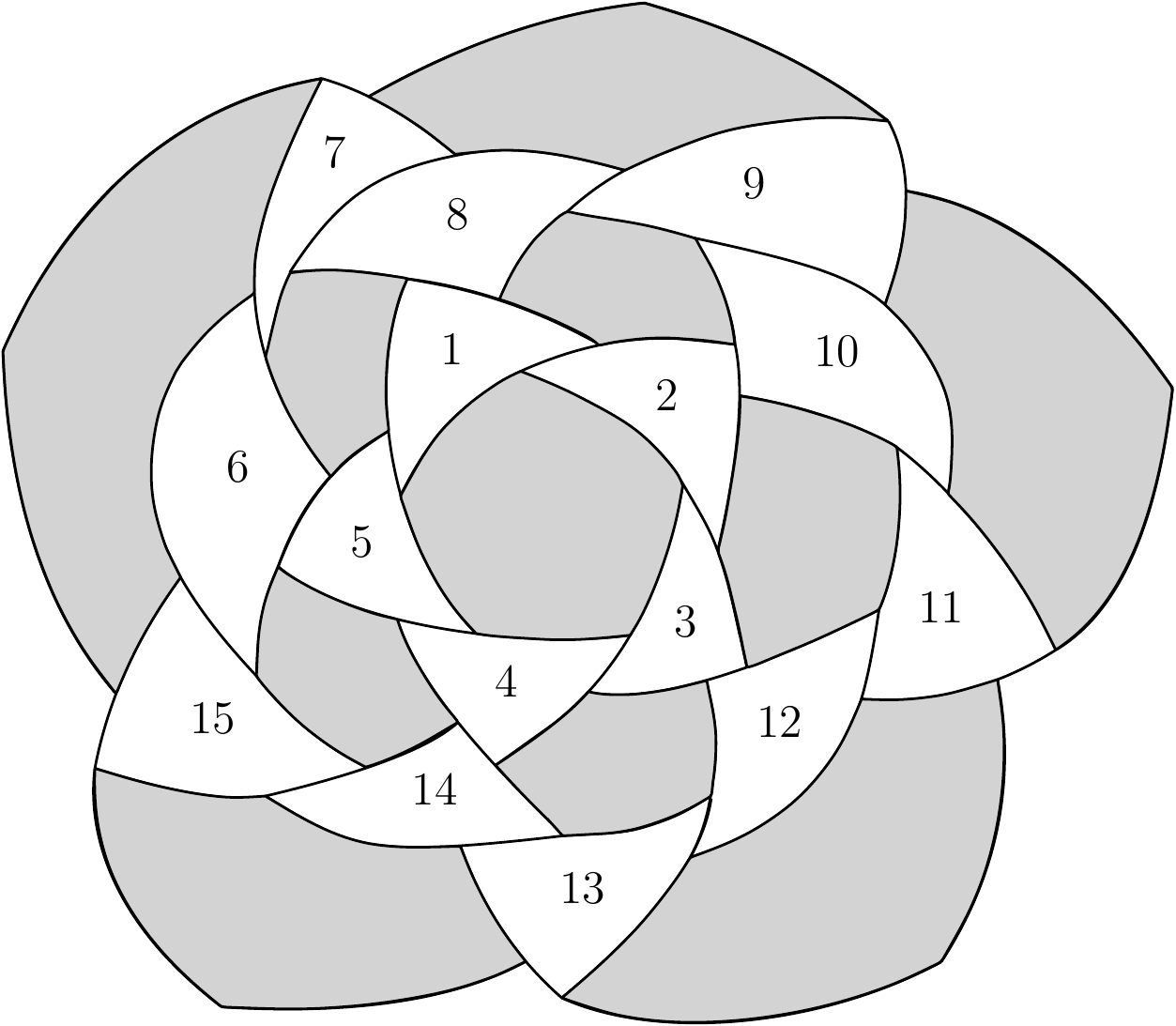}
    \caption{Adding layers of pentagons and triangles to the initial patch around a singleton pentagon.  }
    \label{pentagoncentercase}
\end{figure}

Thus we now have 11 pentagons. This leaves five wells that can only be accommodated by triangles congruent to $T'$. And finally, we have space left over for one more pentagon, yielding  a tiling as shown in Figure \ref{dodicos}. That the last pentagon cannot be itself decomposed further follows from Theorem \ref{polygondecomposition}, since the only regular pentagon that decomposes would have supplementary triangles with angle less than $\pi/3$, a contradiction.

Now, suppose instead we wish to continue tiling from the patch $\mathcal{P}$ in Figure \ref{fig:construct}, but not by adding a layer of pentagons congruent to $T$. Note that at the half-vertices created by these triangles, we may potentially place squares or larger triangles with angles equivalent to those of $T$. We begin with the case of a triangle.

So, suppose we have a triangle $T''$ with angles $\alpha$ that we place in the crevice formed by two copies of $T'$. Since $\alpha > 3\pi/5$ and $\alpha +\beta = \pi$, $\alpha > \beta$.  Since $T''$ and $T'$ are both triangles, we know that $|T''| > |T'|$. Therefore, we form a well with angle $\alpha$ between $T'$ and $T''$ that has base exactly the length of $|T|$. Thus, it must be filled with a copy of the pentagon $T$, as in Figure \ref{pentagoncenterwithtriangle}. But by Lemma \ref{singletonneighborssame}, since this copy of $T$ has two copies of $T'$ on its boundary, all the tiles on its boundary must be copies of $T'$, which contradicts the fact it touches $T''$ on its boundary.

\begin{figure}
    \centering
    \includegraphics[scale=.5]{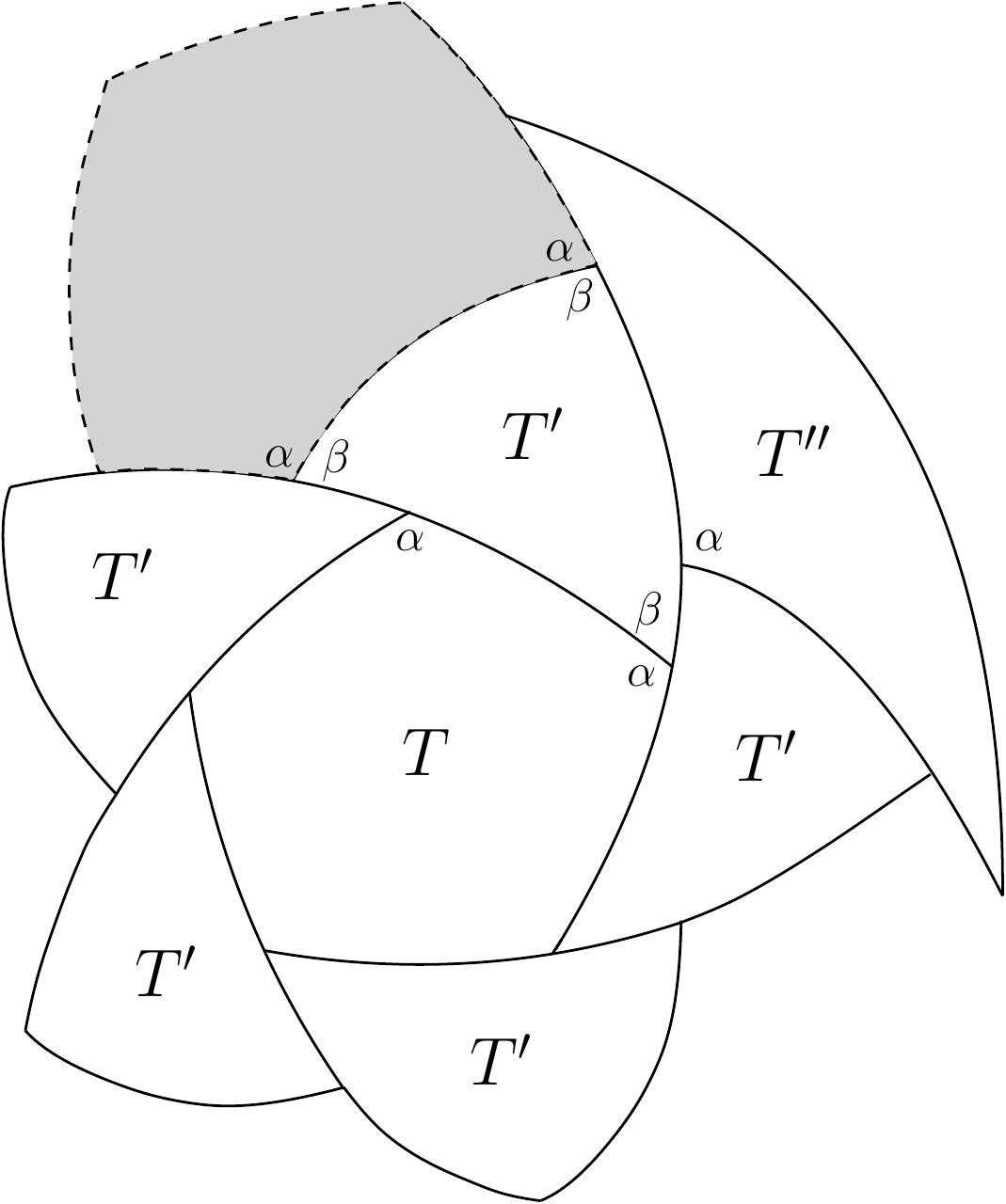}
    \caption{Adding a triangle to the first layer of triangles around a central pentagonal singleton leads to a contradiction.}
    \label{pentagoncenterwithtriangle}
\end{figure}


We now move to the case of placing a square $S$ with angle $\alpha$ into one of the crevices of $\mathcal{P}$. Note that $|S| > |T|$ by Lemma \ref{sameanglepolygons}, since they are a square and pentagon with the same angle measure. Thus, the addition of $S$ generates a well, as in Figure \ref{pentagonwithsquare}, with angles $\alpha$ and base of length $|T'|$. Since $|T| < |T'|$, this cannot be filled with a copy of $T$. Since $\alpha \neq \beta$, it cannot be filled with a triangle. Therefore, it must be filled with a square that has angle $\alpha$ and side-length the same as $T'$. By Lemma \ref{suppsame}, this means that $T'$ and $S$  are a tri-square supp-same pair of angles approximately $70.528^{\circ}$ and $109.472^{\circ}$. This forces $T$ to have angle $109.472^{\circ}$ as well.

As we move clockwise, each square added on forces another well and another square of the same size. We now have a non-edge-to-edge patch with a pentagon, triangles, and squares, as shown in Figure \ref{pentagontrianglessquares}.   

\begin{figure}
    \centering
    \includegraphics[scale=.45]{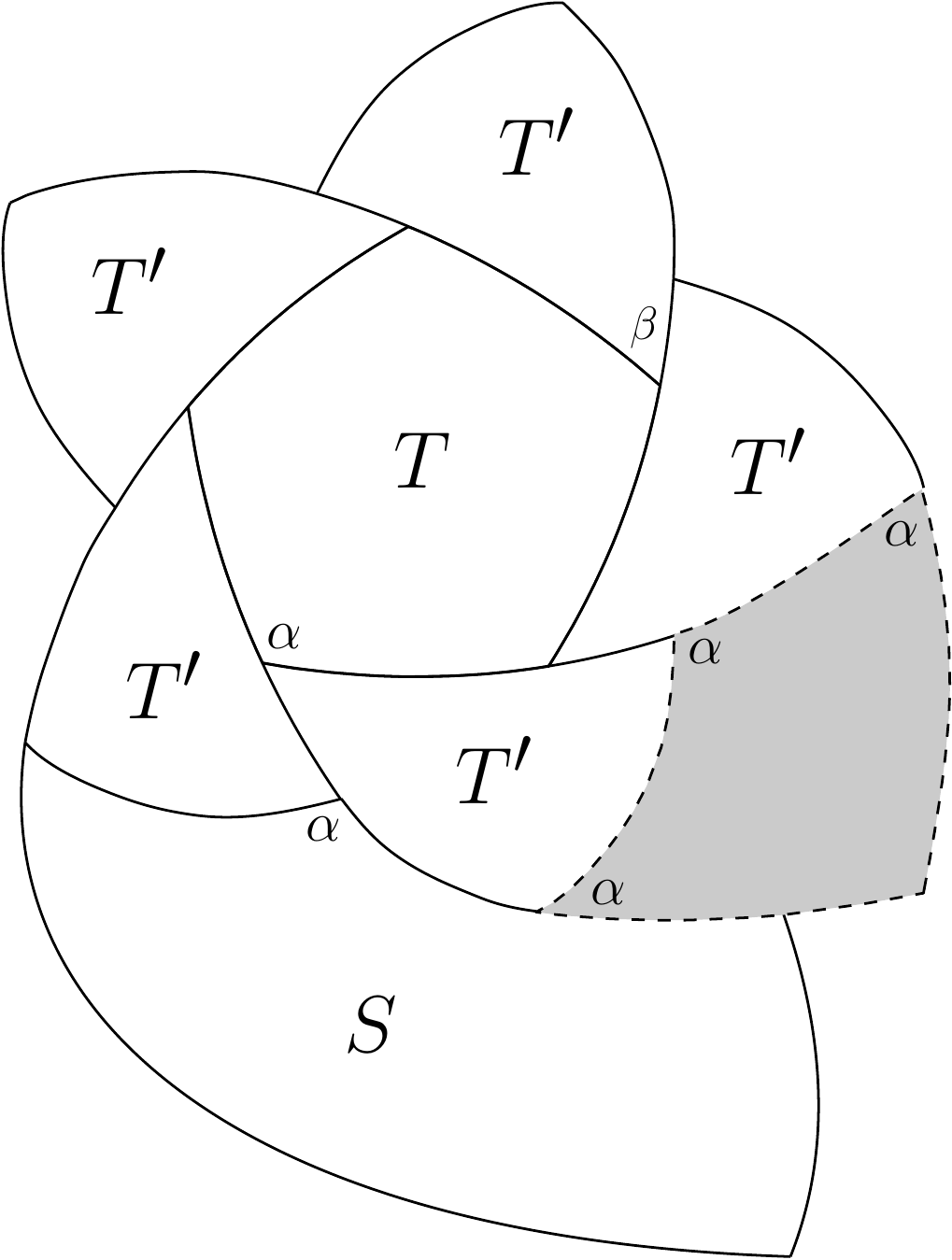}
    \caption{Adding a square to the first layer of triangles around a central pentagonal singleton forces further squares.}
    \label{pentagonwithsquare}
\end{figure}



\begin{figure}
    \centering
    \includegraphics[scale=.43]{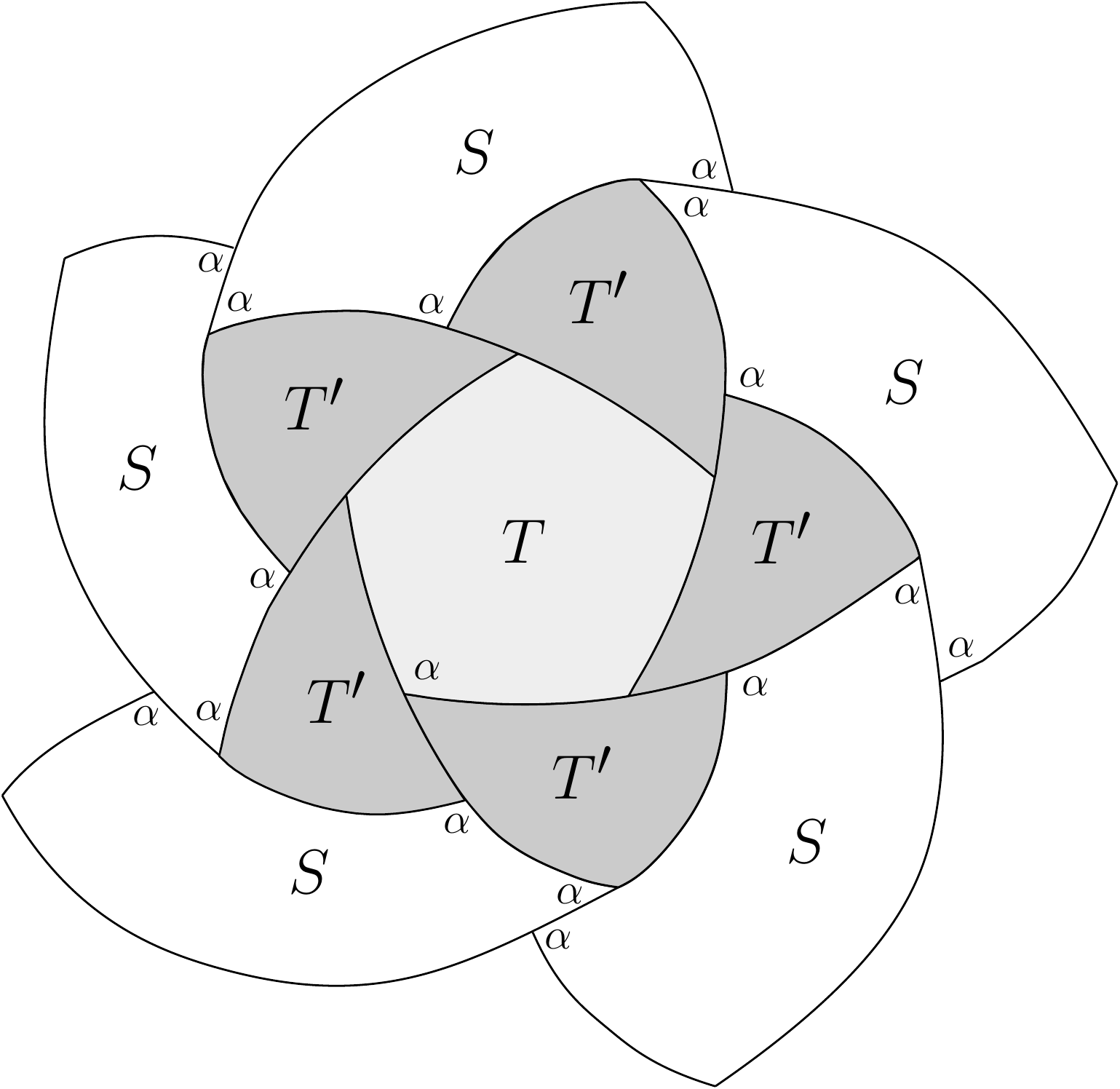}
    \caption{Adding a square to the initial patch around a singleton pentagon forces all squares around the patch.}
    \label{pentagontrianglessquares}
\end{figure}

Continuing to tile outward, we see that we must place more copies of $T$ in the crevices formed by copies of $S$, as we form half vertices using a square, and $T$ is the triangle which supplements $S$. We note, then, that we have surrounded our pentagon $T$ with edge-to-edge patches of triangle-square-triangle, which are the Type II lunes that appear in the cuboctahedral tiling. 

There remains a regular pentagonal slot to fill with angle exactly the same as the initial pentagon. 
By Theorem \ref{polygondecomposition}, $T$ cannot be decomposed further into regular polygons, so the only possible result is the lunar tiling with polar pentagons. 
\end{proof}

We have thus completed our classification of tilings with protoset containing  a prototile of least side-length that only appears as a singleton of minimal side-length.

\section{Tiles of minimal side-length that are not singletons} \label{bigonsection}

We now consider the case that a minimal side-length tile in a tiling appears in an edge-to-edge patch of more than one tile. Such patches can take a variety of shapes. But we will make a change to the tilings to simplify the possibilities. As mentioned in Section 1, we define a regular triangle $T$ with side-length $3\pi/5$ to be a {\bf magic triangle}. It has  $\angle(T) \approx 116.565^\circ$, which is the angle of the pentagon in the special tri-pent supp-same pair. It is magic in the sense that it can be decomposed into three pentagons and four triangles from the tri-pent supp-same pair, as in Figure \ref{magictriangle}. This triangle can be observed as a union of tiles from the icosidodecahdron.

If a tiling has such a magic triangle within it, we decompose the triangle into the four triangles and three pentagons. We call a tiling {\bf decomposed} if all magic triangles within it have been decomposed. Then we will show that any non-singleton edge-to-edge patch of minimal side-length tiles in such a tiling must be a bigon. After we have classified all of the tilings that result, we will return to the question of when tiles can be re-composed into magic triangle tiles, generating additional non-edge-to-edge tilings.
So from now on, in this section, all tilings are decomposed.

\begin{lemma} \label{minimaledgepatch}Let $\mathcal{P}$ be a maximal edge-to-edge patch of more than one tile, all of minimal side length $L$ among the tiles in a decomposed tiling $\mathcal{T}$. Then $\mathcal{P}$ forms a bigon. 
\end{lemma}

\begin{proof}  We consider the collection consisting of all tiles outside the patch $\mathcal{P}$ that share an edge with the patch. We build it one tile at a time. Note that because the tiles in the patch have minimal side-lengths, all added tiles will overhang the side they are glued along on at least one end. Note that there can be more than one component of $\partial \mathcal{P}$, and boundary components can touch one another or themselves at vertices. When they do so, we split the boundary open at the vertex, as in Figure \ref{touchingboundary} for the purpose of making arguments about the boundary, temporarily ignoring a tile in $\mathcal{P}$ that touches at a single vertex as we run along a portion of a boundary component of $\mathcal{P}$. 

\begin{figure}
    \centering
    \includegraphics[scale=.45]{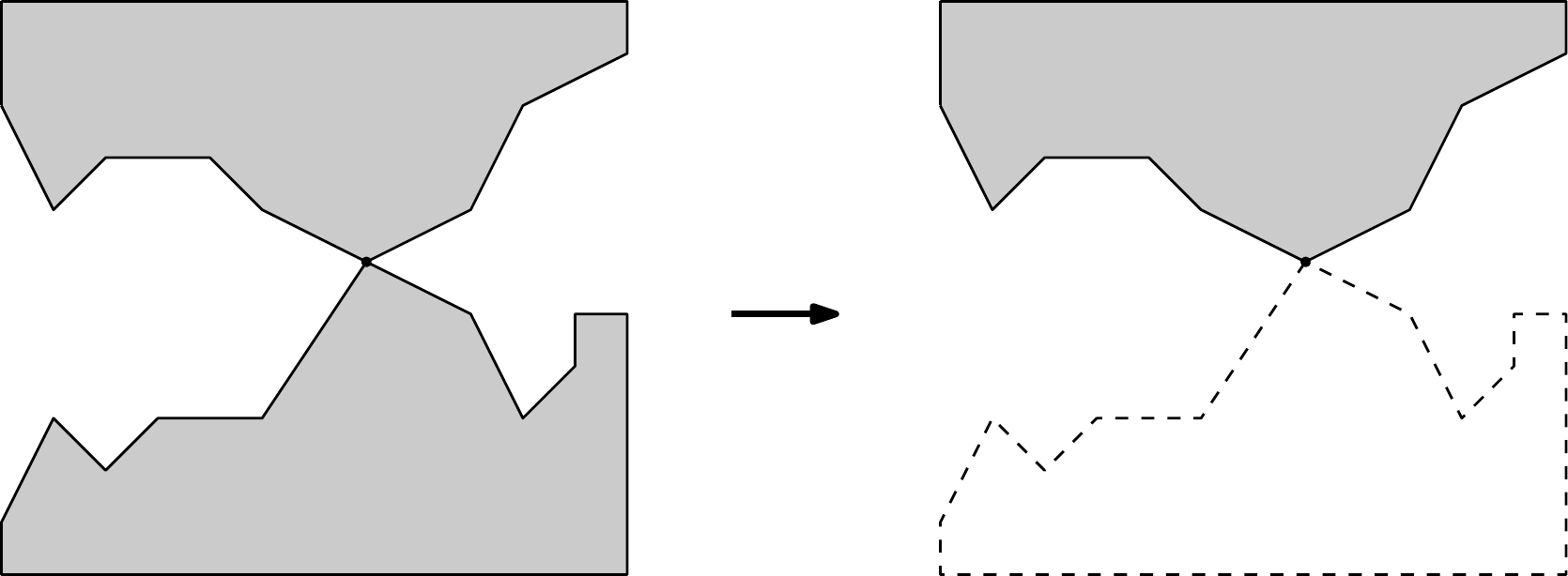}
    \caption{Splitting an edge-to-edge patch at a vertex where boundary touches.}
    \label{touchingboundary}
\end{figure}

Consider the vertices on the boundary of $\mathcal{P}$, including in this set the points at the corners of all of the tiles in $\mathcal{P}$ with sides on the boundary of $\mathcal{P}$. At each such vertex, 
the angle of $\mathcal{P}$ is either less than, equal to or greater than $\pi$. If there is a vertex with angle $\pi$, then there must be exactly two tiles in the patch involved in that vertex, and they are a supp-same pair. We call a vertex of the boundary of $\mathcal{P}$ where the angle is $\pi$ a {\bf $\pi$-vertex}. 
\medskip

\noindent {\bf Claim 1.} Every tile in $\mathcal{P}$  with side on $\partial \mathcal{P}$ is in a $\pi$-vertex on $\partial \mathcal{P}$.


\medskip

Let $T$ be a tile in $\mathcal{P}$ with a side in  $\partial \mathcal{P}$. If both vertices at each end of the side of $T$ have angle in $\mathcal{P}$ greater than $\pi$, this creates a well that can only be filled with a tile of the same length as $T$, contradicting the fact this side of $T$ was on the boundary of the maximal edge-to-edge patch.

If one vertex has angle greater than $\pi$, and one, denoted $v$, has angle less than $\pi$  then the tile $T'$ outside $\mathcal{P}$ glued onto this side of $T$ will overhang at the vertex $v$. Since we cannot have three corners of tiles coming together to form an angle of $\pi$, the only option is that there is exactly one tile $T''$ in the resulting crevice, and it must be supplementary to $T$. If it is supp-same with $T$, then it is part of $\mathcal{P}$, and this is a $\pi$-vertex, contradicting the fact this vertex has angle less than $\pi$. If it is not supp-same with $T$, then it is not part of $\mathcal{P}$, and its side-length must be larger than $L$. It overhangs at the next corner $v'$ of $T$. Repeat the process and we either have a supp-same pair on the boundary of $\mathcal{P}$ involving $T$, or we overhang the next corner of $T$. If we do not hit a supp-same pair on $\partial \mathcal{P}$, we eventually create a well, in one step if $T$ is a triangle, in two steps if $T$ is a square and in three steps if $T$ is a pentagon. That well can only be filled with a supp-same tile that will also be in $\mathcal{P}$, and that tile will have side in $\partial \mathcal{P}$, yielding the $\pi$-vertex on $T$.

If both corners have angle less than $\pi$, then at least one of them is overhung by a tile $T'$ glued to this side of $T$. The subsequent crevice is either filled with a supp-same tile in $\mathcal{P}$, and we are done, or it is filled with a tile of side-length greater than $L$ which is not in $\mathcal{P}$. As in the previous case, it overhangs the next corner of $T$. Repeating this process until we come around the boundary of $\mathcal{T}$, we either come back to the starting corner of $T$ in which case the entire patch $\mathcal{P}$ is a single tile, a contradiction, or we come to another tile in $\mathcal{P}$ that is supp-same on the boundary of $\mathcal{P}$ with $T$.

\medskip

\noindent {\bf Claim 2.} All tiles in $\mathcal{P}$ with sides on $\partial \mathcal{P}$ are special tiles that are either all tri-tri tiles or all tri-square tiles or all tri-pent tiles.

\medskip

This follows immediately from Claim 1, since all edges in the patch must have the same side-length. But we know that the side-lengths of the special tiles in each of the three supp-same cases are distinct. 

\medskip






We now consider the polygonal curve or curves that form the boundary of $\mathcal{P}$. For a side of that polygonal curve (now only treating as corners those points with angle not equal to $\pi$) we define its {\bf tile length} to be the number of sides of tiles in $\mathcal{P}$ that make it up. We call the set of tiles outside $\mathcal{P}$ that glue to a given side of $\mathcal{P}$ a {\bf perfect fit} if their total length matches that of the side of $\mathcal{P}$, and thereby causes no overhang at either end.

\medskip

\noindent {\bf Claim 3:} If there are no perfect fits on sides of $\mathcal{P}$, then $\mathcal{P}$ is a bigon.

\medskip

We first prove that in this case, all corners of $\mathcal{P}$ are convex and each corresponds to a corner of one triangle tile in $\mathcal{P}$.


Given that there are no perfect fits, we can start with any side and on at least one end of that side there will be overhang. The crevice formed must be filled by a tile that has angle supplementary to the angle of the corner. Hence, the corner on $\mathcal{P}$ can only contain one tile. If that corner tile is a square or pentagon, the only tile that can be supplementary to it is a triangle. But since the square or pentagon is special, that triangle would have exactly the same side-length as the corner tile, contradicting the fact that it is not in $\mathcal{P}$. So the corner tile in $\mathcal{P}$  must be a triangle. The next corner in this direction around the boundary of $\mathcal{P}$ must also be overhung, and the same argument shows that corner is made up of a single triangle. We continue in this way, showing all corners in $\partial \mathcal{P}$ must be made up of a single triangle in $\mathcal{P}$. In particular, no corners can be concave.

\medskip

Now we consider each type of supp-same pair separately. 

\medskip

If we have a patch constructed from tri-tri supp-same tiles on its boundary, then start at a corner. This must be a triangle of angle $\pi /2$.  Gluing a second triangle in $\mathcal{P}$ to its non-exposed side yields a bigon. If there is another triangle glued onto this one in $\mathcal{P}$, we create a crevice of angle $3\pi /2$ which must be filled with another triangle of angle $\pi/2$, and we obtain a hemisphere for $\mathcal{P}$, which counts as a bigon with no corners. If there is another triangle glued to the hemisphere, we obtain another crevice of angle $\pi /2$ which is again filed with a congruent triangle. This yields a  bigon with apex angle $3\pi /2$, which yields corners with three triangles, which we have showed cannot occur for a patch with no perfect fits. So both crevices would be filled with congruent triangles yielding an edge-to-edge tiling of the entire sphere.  So the only options for $\mathcal{P}$ are a bigon of apex angle $\pi/2$ and a hemisphere, which also counts as a bigon.

\medskip

If we have a patch $\mathcal{P}$ constructed from tri-square supp-same tiles on the boundary, then again start at a corner, which is a triangle $T_1$ of side-length $\pi/3$. It must be glued to one other tile in the patch, which must be a square $T_2$ by Claim 2. Because there can be no angles at corners of the boundary of $\mathcal{P}$ greater than $\pi$, neither of the adjacent sides of the square can have a tile from the patch glued on. So both sides appear in the boundary of $\mathcal{P}$. There must be another tile from $\mathcal{P}$ glued onto the remaining side of $T_2$ since $T_2$ cannot appear at a corner of $\mathcal{P}$. This must be a triangle $T_3$ to prevent a corner on $T_2$. Neither exposed side of $T_3$ can have a tile from $\mathcal{P}$ glued to it, as this would create a corner of angle greater than $\pi$. Thus the entire patch is a bigon with two triangles and one square.

\medskip

If we have a patch constructed from tri-pent supp-same tiles on the boundary, then again start at a corner, which is a triangle $T_1$ of side-length $\pi/5$. It must be glued to one other tile in the patch, which must be a pentagon $T_2$ by Claim 2. Because there can be no angles greater than $\pi$, neither of the adjacent sides of the pentagon can have a tile from the patch glued on. So both sides appear in the boundary of $\mathcal{P}$. In order to prevent a corner involving a corner of the pentagon, there must be triangles $T_3$ and $T_4$ glued to each of the remaining exposed sides of $T_2$. This generates a crevice with angle available for a new pentagon $T_5$. That pentagon must be present in the patch because otherwise we would have a corner of angle greater than $\pi$. Neither of the adjacent exposed sides of the pentagon can have a tile in $\mathcal{P}$ glued on as again, this would generate an unacceptable corner. But to avoid any other unacceptable corners on the pentagon $T_5$, the remaining side must have a triangle glued onto it. Neither of the exposed sides of this triangle can have a tile from the patch glued on as this would also generate an unacceptable angle. Thus, we have generated a bigon with two pentagons and four triangles, and this is the entire patch.

\medskip

We  define a perfect fit to be a {\bf supplementary perfect fit} if the angle of at least one end of the perfect fit tiles is supplementary to the angle at that corner of the patch. 

\medskip


\noindent {\bf Claim 4:} If there is at least one perfect fit on a side of $\mathcal{P}$, then $\mathcal{P}$ is a bigon.

\medskip

We prove this via a series of claims.

\medskip

\noindent {\bf Claim 4a:} There are no supplementary 
perfect fits on the sides of $\mathcal{P}$.

\medskip

We first consider the tri-tri supp-same patch. if there were a supplementary perfect fit, the tile in the perfect fit that is supplementary would have to have angle $\pi/2$, and therefore be a triangle in $\mathcal{P}$, a contradiction. 


Now consider the tri-square supp-same patch. Suppose there is a supplementary
perfect fit. Then the collection of tiles in the perfect fit must have length $m \pi/3$ where $1 \leq m \leq 5$. 

A corner of the patch that could have a supplementary angle for the perfect fit is either made up of a single triangle or a single square since three tiles cannot fit together at their corners to make an angle of $\pi$.  If the tile is a single square, then the angle to fill to make them supplementary is $\cos^{-1}(1/3) \approx 70.529^\circ$, which can only be filled with a triangle, and that triangle has side length $L$, a contradiction to it not being in the patch. So the only possibility is that the corner in the patch consists of one triangle, and the tile in the perfect fit must be a triangle of angle $\cos^{-1}(-1/3) \approx 109.471^\circ$. The subsequent tiles in the perfect fit would have to be supplementary to one another. So the next tile in the perfect fit would be a triangle of angle $70.529^\circ$ and side-length $L = \pi/3$. In order to avoid collision between its two neighbors in the perfect fit, the next tile, if it exists, must be a square of side-length $L$. It cannot be a pentagon, as such a pentagon would have side-length less than $L$, which is the minimum over all tiles in the tiling. 

The subsequent tile, if it exists, must be a triangle, also of side-length $L$. If there is a fifth tile, then it cannot be a triangle, as a triangle of angle $109.471^\circ$ has length $cos^{-1} (-1/4) \approx.5804 \pi$, and the first triangle would overlap with it, since both lie on the boundary of the bigon of length $\pi$ formed by the second, third and fourth tiles. 

The fifth tile cannot be a pentagon, as it would have side-length less than $L$. So it must be a square with side length exactly $L$. This forces the sixth possible tile in the perfect fit to be a triangle of side-length $L$. 

Thus, we have seen that the length of the perfect fit is $(.5804 + p/3) \pi$, which does not equal $m\pi/3$ for any $2 \leq m \leq 6$ and $1 \leq p \leq 5$. So there are no supplementary perfect fits in this case.


\medskip




Now we consider a  minimal side-length edge-to-edge tri-pent supp-same patch in a decomposed tiling. If there is a supplementary perfect fit, then the tiles in the perfect fit must have total length $m \pi/5$ for  some $2 \leq m \leq 10$. A corner of the patch that could have a supplementary angle for the perfect fit is either made up of a single triangle or a single pentagon since again, three tiles cannot fit together to make an angle of $\pi$.  If the tile is a single pentagon, then the angle to fill to make them supplementary is $63.434^\circ$, which can only be filled with a triangle, and that triangle has side length $L$, a contradiction to it not being in the patch. So the only possibility is that the corner in the patch consists of one triangle. The supplementary tile in the perfect fit cannot be a pentagon, as it would then be of side-length $L$, and would be in $\mathcal{P}$, a contradiction. It cannot be a  triangle, as if it were, it would be a magic triangle, which has been decomposed. So it must be a square.

     A square of angle $116.566^\circ$ has side length approximately $.3752 \pi$.
The subsequent tiles in the supplementary perfect fit must be supplementary to one another. So the next tile in the perfect fit, if it exists, must be a triangle of side-length $L = \pi/5$. In order to avoid collision between its two neighbors in the perfect fit, the next tile, if it exists, must be a pentagon of side-length $L$. The subsequent tile, if it exists, must be a triangle, also of side-length $L$. The next tile could either be a square or pentagon of angle $116.566^\circ$. Then we repeat the pattern of tiles. The length of the perfect fit then must be equal to $j (.3752\pi) + k (.2 \pi)$ for $j \geq 1$ and $k \geq 1$.  But there is no combination of $j (.3752\pi) + k (.2 \pi)$ for $j \geq 1$ and $k \geq 1$ that is equal to $m (.2 \pi)$ for $2 \leq m \leq 10$. Thus, there are no supplementary perfect fits in the tri-pent case.


\medskip

\noindent {\bf Claim 4b:} If there are any perfect fits along the sides of $\mathcal{P}$, then  all sides must be perfect fits.

\medskip

If not, then take a side of $\mathcal{P}$ that is not a perfect fit. Then there is overhang at at least one of its two vertices. Continue around the boundary of $\mathcal{P}$, filling in sides, and in each case generating a new overhang until we come to a side that is a perfect fit. Since the perfect fit cannot be supplementary, either the overhang and the perfect fit collide, or if not, there is not enough angle between the perfect fit and the overhang to fill the resulting crevice, since we can never have three corners of tiles adding to angle $\pi$.

\medskip

\noindent {\bf Claim 4c:} If all sides are perfect fits, then the patch must be a bigon.

\medskip

We  subdivide into the cases corresponding to the choice for a supp-same pair.

\medskip

\noindent {\bf Case 1:} If the tiles in $\partial \mathcal{P}$ are all tri-tri special tiles, then they each have angle $\pi/2$ and side-length $\pi/2$.  
Since these triangles are the smallest triangles in the tiling, the only possible perfect fit would be a square that has angle $\pi$, and that forms a hemisphere. But this can only occur along one side of the patch. It would cause an``overhang'' on both ends of the side. When we follow this overhang around the boundary of $\mathcal{P}$ in one direction we would run into the other overhang. So the only possibility for a perfect fit is if the patch is also a hemisphere corresponding to half of the octahedron. But a hemisphere is a bigon.


 \medskip
 
\noindent {\bf Case 2:} If the tiles in $\partial \mathcal{P}$ are all tri-square special tiles, then they have side-length $\pi/3.$ 
Each side of $\mathcal{P}$ has length $m \pi/3$ for $1 \leq m \leq 6$.

If  $ m=1$, then the perfect fit would itself be a tile of length $\pi/3$, which would contradict the fact it is not in the patch. If $m=2$, then the two tiles from $\mathcal{P}$ that are in the side are a triangle and a square.  The perfect fit must be a single tile glued to this side of length $2\pi/3$. The only such tile is a triangle that forms a hemisphere, with angles $\pi$. As a hemisphere, it overhangs both of the endpoints of the side of the patch. At the endpoint corresponding to the square, the resulting crevice must be filled with a supplementary tile of greater length. But the only option is a supp-same triangle, a contradiction to the fact this tile is not part of $\mathcal{P}$.

So all sides of $\mathcal{P}$ must have tile length three or greater. But then two adjacent sides are length $\pi$ or greater which means they intersect again on the far side of the sphere. Hence they cannot be greater than $\pi$, and we have a bigon.





In the case there are no corners, so  $m = 6$, the patch is an entire hemisphere, which is a bigon and we are done.


\medskip

      
\noindent {\bf Case 3:} If the supp-same pair is a triangle and pentagon, they each have side-length $\pi/5$, and any side of the patch $\mathcal{P}$ has length $m\pi/5$ for $m = 1, \dots, 10$. There cannot be a side corresponding to $m=1$ as a perfect fit on such a side would require the tile in the perfect fit to be in $\mathcal{P}$, a contradiction. 

We first consider the angles at the corners of $\partial \mathcal{P}$. We show that none of the angles are concave. Suppose first that two pentagons and a triangle from $\mathcal{P}$ meet at a corner. Then the crevice formed has angle $63.471^\circ$ and the only tile that fits is a triangle that would be in $\mathcal{P}$, a contradiction. 

Suppose next that two triangles and a pentagon from $\mathcal{P}$ meet at a corner of $\partial \mathcal{P}$. Then the crevice has angle $116.566^\circ$. This is not large enough for two tiles, so it must be filled with a single tile. It cannot be filled with a triangle, as such a triangle would be a magic triangle, all of which have been decomposed. It cannot be a pentagon, as such as pentagon would be in $\mathcal{P}$. So the only possibility is that it is a square. But a square of angle $116.566^\circ$ has side-length of $.3752\pi$.  Since this tile is at the corner of two sides of $\mathcal{P}$, it is in two perfect fits. Considering one of them, if there is a second tile in the perfect fit, it is supplementary to it and therefore must be a triangle of angle $63.472^\circ$. The subsequent tile in the perfect fit, if it exists, must be a special pentagon, so as not to overlap with the square. The next tile, if it exists,  must again be a triangle of angle $63.472^\circ$. The subsequent tile, if it exists, must be a special pentagon, so as not to overlap with the first square. The subsequent tile, if it exists, must again be a triangle. The subsequent tile, if it exists, could either be a square or a pentagon, and then the sequence starts over. In particular, the length of the side of the perfect fit tiles is $ j(.3752\pi) + k (.2\pi$) for $1 \leq j \leq 2$ and  $0 \leq k$. But this cannot equal $m\pi/5$ for $2 \leq m \leq 10$. Thus, there are no such corners on $\partial \mathcal{P}$.

Suppose that two pentagons from $\mathcal{P}$ make up a corner of $\partial \mathcal{P}$. Then the crevice has angle $126.868^\circ$. There cannot be two tiles filling this crevice as if there were, one would have angle less than or equal to $63.472^\circ$, which would imply its length is either $L$ or less than $L$ contradicting the fact it is not in $\mathcal{P}$ or cannot have side-length less than $L$. 

Let $T$ be the single tile filling the crevice of angle $126.868^\circ$. But such a tile cannot have a supplementary tile since the supplementary angle is too small. So there can only be one tile in this perfect fit.
It is either a triangle, square or pentagon. The side-lengths of these three are respectively $.6223 \pi$, $.4195 \pi$ and $.2805 \pi$, so they are not equal to $m \pi/5$ for $2 \leq m \leq 10.$

Thus all corners are convex. This leaves only three possibilities. A corner of the boundary of $\mathcal{P}$ could consist of a triangle from $\mathcal{P}$, a pentagon from $\mathcal{P}$ or two triangles from $\mathcal{P}$. 

We now consider the possible edge-to-edge patches, knowing the corners are convex and knowing that every side of $\mathcal{P}$ contains at least two tiles. We show there are five possibilities.

First suppose that there is a corner $v$ in $\mathcal{P}$ consisting of a single triangle $T$. Since the two sides of $\mathcal{P}$ meeting at $v$ must contain at least two tiles, there must be a pentagon $P$ glued to the third side of $T$. In order to avoid a side of tile length 1, there must be a triangle glued to one of the exposed sides of $P$ that is not adjacent to $T$. This is the first of our candidates, appearing as (1) in Figure \ref{penttripatches}. We could also glue another triangle onto the remaining exposed side of $P$ that is not adjacent to $T$. This creates a concave crevice which must then be filled with a pentagon. We now have two sides of length 4 that start at $v$. No other tiles in $\mathcal{P}$ can be glued to these sides since to do so would create a concave corner, which when filled would eventually cause the original corner to no longer consist of a single triangle. But currently, there is a side of tile length 1, so we must glue a triangle onto the end of the last pentagon, and we now have a bigon. This is our second option, appearing as (2) in Figure \ref{penttripatches}.

We now assume that we have a corner $v$ of $\mathcal{P}$ consisting of a single pentagon labelled $T_1$. (See Figure \ref{penttripatches}(3)). Since all sides of $\mathcal{P}$ have tile length at least 2, there must be a triangles $T_2$ and $T_3$ glued onto each of the two sides of $T_1$ that are adjacent to the sides that touch $v$. We have already seen this candidate patch. But we can add to it further. 

We could extend one of the sides touching $v$ by adding a pentagon $T_4$ to that side. This creates a concave crevice that must be filled with a triangle $T_5$. And this creates a crevice that must be filled with a pentagon $T_6$. This then results in another crevice between the pentagons that must be filled with a triangle. This is our third candidate patch. Note that it is a decomposed magic triangle, as in Figure \ref{penttripatches}(3).

We could extend the decomposed triangle further by adding a triangle $T_8$ to pentagon $T_4$. This creates a crevice that must be filled with pentagon $T_9$, which creates a crevice that must be filled with triangle $T_{10}$. The end result is a bigon, which is our fourth candidate edge-to-edge patch, as in Figure \ref{penttripatches}(4). If instead of adding a triangle to the side of $T_4$, we had added a pentagon to the exposed side of $T_7$, we would ultimately have ended up at the same bigon.

Instead, we could extend the tiles $T_1$, $T_2$ and $T_3$ in (1) by adding a triangle to the triangle $T_2$ or to the triangle $T_3$ or adding a pentagon to the pentagon $T_1$. In any of these cases, the resulting crevices force us to add in two more tiles, obtaining the patch in Figure  \ref{penttripatches}(5).

Finally, we consider what happens if we have two triangles making up a corner of $\mathcal{P}$. As in Figure \ref{penttripatches}(5), we are forced to add two pentagons to obtain sides of length at least 2, and then we are forced to add two triangles to fill the crevice and we have the same patch as the fifth mentioned previously. So these five are all the options. Note that therefore,  the only tile lengths for sides of the resultant patches that are not a bigon are $m = 2$ and $m=3$.

\begin{figure}
    \centering
    \includegraphics[scale=.58]{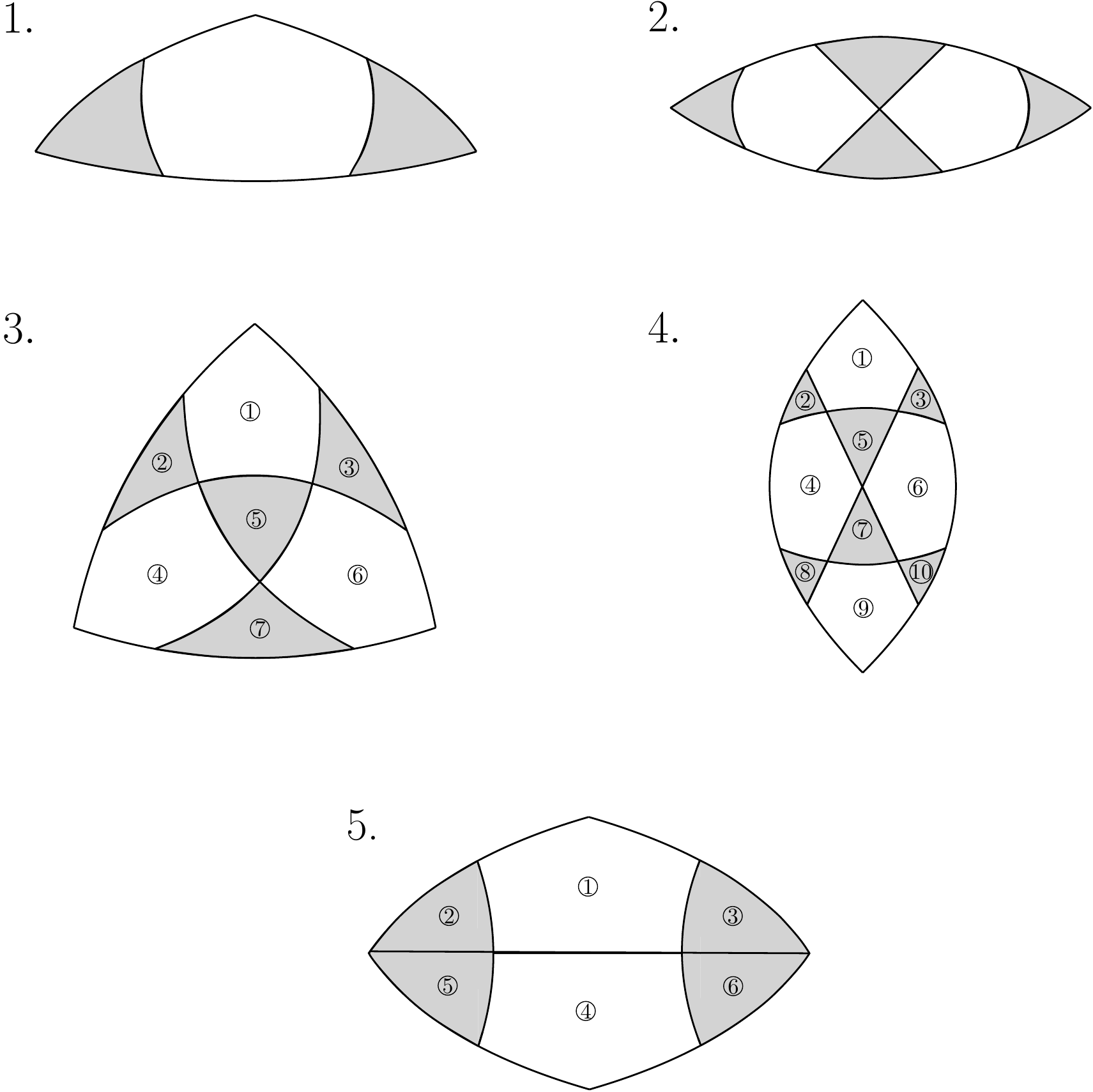}
    \caption{Possible edge-to-edge patches made from tri-pent supp-same tiles.}
    \label{penttripatches}
\end{figure}

We now consider possibilities for tiles glued to sides of $\mathcal{P}$ with these side-lengths. A side-length for $m = 2$ can only have a perfect fit of one tile. This could be either a triangle, square or pentagon of side-length $2\pi/5$. The angle on the triangle is $76.345^\circ$, the angle on the square is $121.86^\circ$, and the angle on a pentagon would be $180^\circ$.

A side-length of $m=3$ cannot correspond to a single tile in the perfect fit, as the side-length of that tile would be $3\pi/5$. All triangles of that length are magic triangles that have been decomposed, and that length is longer than the side length of any other regular polygon.

So there must be two tiles that make up the perfect fit, which because they are supplementary, must be two triangles, a triangle and a square or a triangle and a pentagon. But two supplementary triangles have total length at least $.60817\pi$, the lower bound of which which corresponds to one triangle with angle approaching $\pi/3$, so its side-length approaches 0, and the other triangle having angle approaching $2\pi/3$. So that case is out. 

The total length of a supplementary triangle and pentagon pair is a maximum of  $.436 \pi$, which is too small, so that is out. So it must be a triangle and a square. The sum of the side-lengths of a supplementary pair of triangle and square, with the angle of the triangle being $\alpha$ is $$f(\alpha) = \cos^{-1} \left( \frac{-1/2 + \cos^2(\alpha/2)}{\sin^2(\alpha/2)}\right) + \cos^{-1} ( \tan^2(\alpha/2))$$ This function is concave down with one local maximum. When it is set equal to $3\pi/5$, there are two possibilities for the pair. We could have a  triangle with angle $88.2545^\circ$ and square with angle $91.7455^\circ$, but then the length of the side of the square is $.11\pi$, smaller than the smallest length of a side of a polygon in this tiling, which is $.2 \pi$. The only other option is  a triangle $T_1$ with angle $64.6756^\circ$ and square $S_1$ with angle $115.3243^\circ$. So this is the only possibility for a perfect fit on a side of $\mathcal{P}$ of tile length 3.

We now consider each of the three types of possible edge-to-edge patches that are not bigons. We start with the decomposed magic triangle, with angle at the corners of $116.566^\circ$ as in Figure \ref{penttripatches}(3). Since each side has tile length 3, all three perfect fits must correspond to the square-triangle supplementary pair just discussed. If two squares come together at a corner of the patch, then the crevice has angle $360 - 116.566 - 2(115.324) = 12.786^\circ$, too small for a tile to fill the crevice. 

So it must be that the subsequent perfect fits are arranged so that the square $S$ at one side at a given corner is paired with a triangle $T$ on the next side at that corner. Then the crevice formed at a corner has angle $360 - 116.566 - 115.324 - 64.675 = 63.435^\circ$. This crevice can only be filled with a triangle of side length $\pi/5$, as any other tile would have side-length less than $L$. But this creates a well on the triangle $T$ that has length $.03 \pi$, too small for any tile. See Figure \ref{penttri1}.

\begin{figure}
    \centering
    \includegraphics[scale=.55]{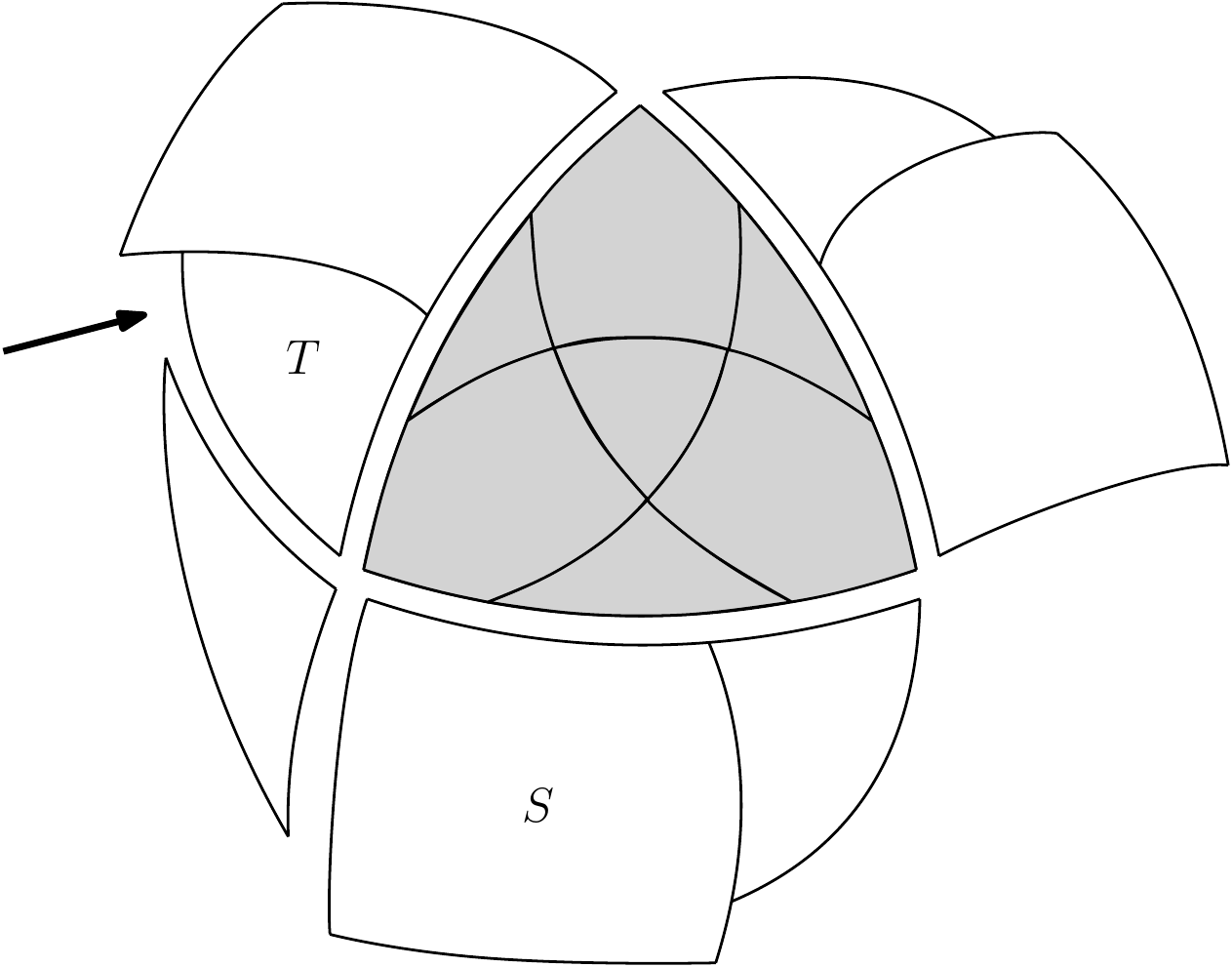}
    \caption{A decomposed magic triangle cannot be a maximal edge-to-edge patch.}
    \label{penttri1}
\end{figure}

We now consider the patch with two triangles and one pentagon as in Figure \ref{penttripatches}(1). The perfect fits on the sides of tile length 2 could be a triangle, square or pentagon. But the pentagon would have angle $\pi$, which would not allow room for the other perfect fits. Neither perfect fit for these sides can be a square either, as such a square of side-length $2\pi/5$ has angle $121.861^\circ$, and  there would not be enough angle remaining at the corner shared by the two of them for a tile to fill the resulting crevice. So the only possibility is that both of these perfect fits are triangles $T_1$ and $T_2$ with angle $76.345^\circ$. And the perfect fit on the side of tile length 3 is the triangle-square pair mentioned above, as in Figure \ref{penttri2}. 

Starting with the corner corresponding to the pentagon in $\mathcal{P}$, we have a crevice of angle $360 - 116.566 - 2(76.345) = 90.744^\circ$. This cannot be a square as it would have side-length less than $L$. So it is a triangle that has side-length larger than the two perfect fit triangles to either side of it.

Concentrating on the corner where the perfect fit has a square, the crevice that remains after adding the two perfect fits here has angle $360 - 115.324 - 63.472 - 76.345 = 104.859^\circ$. This can only be a single tile. If it is a triangle, then its side-length is longer than $T_1$, and it collides with $T_3$,  a contradiction, as in Figure \ref{penttri2}. 

\begin{figure}
    \centering
    \includegraphics[scale=.7]{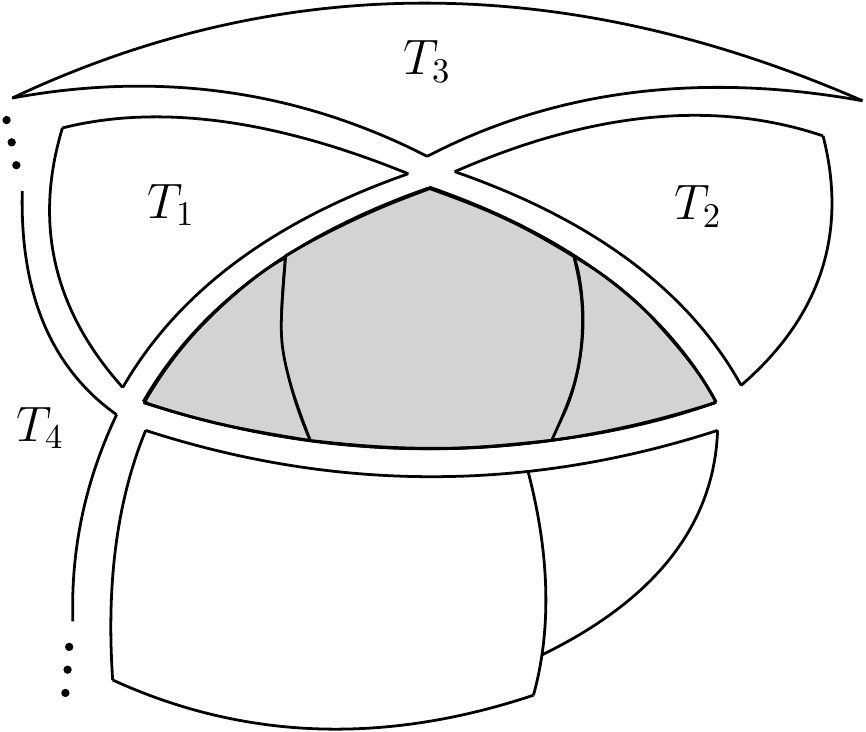}
    \caption{The shaded patch cannot be a maximal edge-to-edge patch.}
    \label{penttri2}
\end{figure}

If it is a square, then its side length is $.2984 \pi$. It creates a well of length $.1016 \pi$, which is smaller than $L = .2 \pi$, which is the smallest size of a tile. So no tile can fill it and we have another contradiction. So this patch cannot exist.

Finally we consider the edge-to-edge patch that consists of two pentagons and four triangles as in Figure \ref{penttripatches}, which has all sides of tile length 2. None of the perfect fits can be squares by considering one of the corners corresponding to the pentagon. As occurred in the previous case, if either perfect fit were a square, the crevice formed by the corner and the two perfect fits would not have angle enough for any tile. 

So all perfect fits are congruent triangles of side length $2\pi/5$ and angle $76.345^\circ$. But then the crevice formed at the corner in $\mathcal{P}$ with two triangles has angle $360 - 2(63.472) - 2(76.345) = 80.366^\circ$. hence, this must be filled with a triangle $T$ of side length longer than the side-length of the perfect fit triangles, as in Figure \ref{penttri3}.

The crevice formed at at the pentagon corner in $\mathcal{P}$  has angle $360 - 116.566 - 2(76.345) = 90.744^\circ$. This can only be filled with a single tile. If it is a triangle, it has side length longer than the perfect fits and it collides with $T$. If it is a square, it has side length shorter than $L = \pi/5$, a contradiction.

\begin{figure}
    \centering
    \includegraphics[scale=.65]{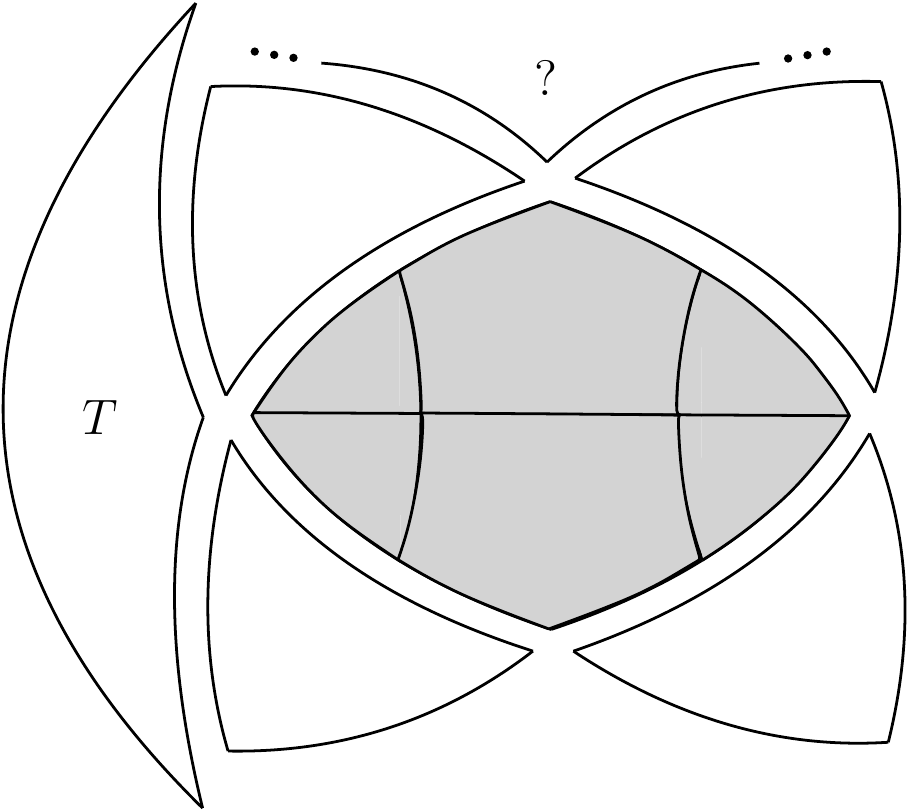}
    \caption{The shaded patch cannot be a maximal edge-to-edge patch.}
    \label{penttri3}
\end{figure}



 This completes the proof of Lemma \ref{minimaledgepatch}.

\end{proof}

An {\bf indivisible edge-to-edge bigon} in a tiling is a bigon tiled edge-to-edge that contains no other proper edge-to-edge tiled bigons.

\begin{lemma} \label{bigonoptions} An indivisible edge-to-edge bigon in a tiling must be one of the five types in Figure \ref{fig:bigontilings}.
\end{lemma}

\begin{proof} Since all tiles have side-length at most $2\pi/3$, there must be vertices on each side of the bigon that are $\pi$-vertices that correspond to a supp-same pair of tiles. Thus the only possibilities for tiles on the boundary correspond to the alternate pairs of the types of supp-same pairs from Lemma \ref{suppsame}, all of the same type. But by the lengths delineated in that lemma, the boundaries must appear as in the five types shown. In the case of Type V, there is exactly room for the two additional triangles in the interior. So this is a complete list of indivisible edge-to-edge bigons.
\end{proof}

We can glue two copies of Type I bigons together along a side to obtain an edge-to-edge tiling of a hemisphere, which is exactly the tiling of a hemisphere obtained from the octahedron. We can glue a side of a Type II bigon to a side of a Type III bigon to obtain another edge-to-edge tiling of a hemisphere, which is the tiling of a hemisphere obtained from the cuboctahedron. The same holds for Types IV and Type V bigon,  yielding a hemisphere that comes from the icosidodecahdron.

\begin{lemma} \label{bigontilings}
If there is a smallest tile that appears in a maximal edge-to-edge bigon with apex angle less than $\pi$, then the corresponding tiling must be one of three lunar tilings or one of the three sporadic tilings.
\end{lemma}

\begin{proof}
The bigon under consideration must be one of the five delineated in Lemma \ref{bigonoptions}, or a union of two copies of Type II, or two copies of Type IV. Since the apices must be vertices of the tiling, there will be a tile not in the bigon with its side contained in the edge of the bigon and such that the tile has a corner at the apex of the bigon. Let $\alpha$ be the angle at the apex of the bigon and $\beta$ the angle of this tile.

\medskip

\noindent {\bf Case 1.} Suppose that there is such a tile $T$ at an apex of the bigon $B$ that is supplementary with the tile that is at the apex of the bigon. So $\alpha + \beta = \pi$. In this case, the bigon cannot be a union of two copies of  Type II or two copies of Type IV bigons because the supplementary angle to the bigon apex angle is too small for a tile.

For a Type I bigon, the apex has angle $\alpha = \pi/2$, and therefore the only possibility for $T$ is a triangle of angle $\pi/2$. But then there are three edge-to-edge triangles in the corresponding edge-to-edge patch, contradicting that the bigon was supposed to be a maximal edge-to-edge patch.

In the case of a Type II bigon, the apex has angle $\alpha \approx 70.529^\circ$, and the tile $T$ cannot be a square for the same reason, as it would yield a bigger edge-to-edge patch. If it is a triangle, that triangle will have angle $\beta \approx 109.471^\circ$, and it generates a crevice of angle $\alpha$. 

The tile, call it $T'$ that fills that crevice is another triangle of angle $\alpha$, since the angle is too small for a square or pentagon. Then we have a well of angle $\beta$. Since the $\alpha$ triangle is minimal side-length, there cannot be two or more tiles filling the well (as half-vertices must be incident to triangles). So the only way to fill the well is with a square with the same length as the $\alpha$-triangle. Thus, the $\alpha$-triangle is in an edge-to-edge patch that contains at least two tiles. So it must be in a bigon, and the only possibility is a Type II bigon. Repeating exactly this argument for the new crevice formed by the new bigon and the tile $T$ yields a third Type II bigon, the three bigons of which surround the triangle $T$. The complement of these three bigons and the tile $T$ is a triangle of angle $\beta$ congruent to $T$, which cannot be decomposed further into smaller regular tiles by Theorem \ref{polygondecomposition}. Therefore, we have the first of the lunar tilings.

Note that for the Type II bigon, if $T$ was a pentagon with angle $\beta$, then it has side-length smaller than that of the triangle of angle $\alpha$. This case is not covered by the lemma, since the smallest tiles are not in the bigon, but it does lead to the second of the lunar tilings, which we saw when we considered pentagonal singletons.

If the bigon is Type III, then its apex has angle $\alpha \approx 109.471^\circ$ and the supplementary tile $T$ must have angle $\beta \approx 70.529^\circ$.  But then it must be a triangle and that cannot occur because that would extend the edge-to-edge patch. So this bigon cannot occur in a tiling with a supplementary tile at the apex.

If the bigon is Type IV, it has apex angle $\alpha \approx 63.435^\circ$, and the supplementary tile $T$ has angle $\beta \approx 116.565^\circ.$ Again here, $T$ cannot be a pentagon as that would extend the edge-to-edge patch beyond the bigon. 

     Suppose $T$ is a triangle. Then repeating exactly the argument we used in this case for the Type II bigon, we obtain the third of the lunar tilings. 
     
     Suppose now that $T$ is a square, then it has side-length greater than that of the triangle of angle $\alpha$. Then the crevice formed by the bigon and $T$ must be filled with a triangle $T'$ of angle $\alpha$. This generates a well with base along $T'$ and angles $\beta$. Since the triangle with angle $\alpha$ is the smallest triangle, there cannot be more than one tile to fill this well. So it must be filled with a tile of the same length as the triangle $T'$. Hence, $T'$ must be in a bigon of Type IV. 

Repeating the above arguments as we move around the edges of $T$, we obtain four bigons of Type IV around $T$. The union of the four bigons and $T$ fill the entire sphere except for a square of angle $\beta$, which by Theorem \ref{polygondecomposition} cannot be further decomposed. Hence, we obtain the fourth of the lunar tilings.

\medskip

\noindent{\bf Case 2.} Suppose there is no tile that shares a vertex with an apex of $B$, and that has a side on a side of $B$ and that is supplementary with the angle at the apex of $B$. Then each side of $B$ must have a perfect fit, which means that the tiles that glue to each of the two sides of $B$ have total length equal to $\pi$.

We begin with the Type II bigon $B$. Since the tiles in $B$ are minimal side-length, there must be two tiles in each perfect fit. In order that their side-length add to $\pi$, they must both be right-angled triangles. So the bigon $B$ shares sides with two Type I bigons. This leaves room on the sphere for a single bigon of apex angle approximately $109.471^\circ$. In order that the side-lengths of this bigon be $\pi$, this bigon must be a Type III bigon, and we obtain the first of the sporadic tilings. A similar argument for a Type III bigon yields the same resulting tiling.

Suppose now that $B$ is a Type IV or a Type V bigon. Then again, we need perfect fits along each of its sides, though now those perfect fits can contain two, three or four tiles. In the case of two tiles, once again, to be supplementary and add to $\pi$ in length, they must be two triangles in a 
Type I bigon.

Note that every triangle that appears in the perfect fit must have angle at least $70.529^\circ$ in order that it not be smaller side-length than the tiles in the original bigon. 

Define

\begin{align*}
g(\alpha) = & A \cos^{-1}\left(\frac{-1/2 + \cos^2(\alpha/2)}{\sin^2(\alpha/2)}\right) +  B \cos^{-1}\left(\tan^2(\alpha/2)\right) \\ 
 + &  C \cos^{-1}\left(\frac{\cos(2 \pi/5) + sin^2(\alpha/2)}{\cos^2(\alpha/2)}\right) +  D 
\cos^{-1} \left(\frac{(-1/2 + \sin^2(\alpha/2)}{\cos^2(\alpha/2)}\right) 
\end{align*}

 For various nonnegative integer entries in $(A, B, C, D)$, where $A$ is the number of initial triangles, $B$ is the number of squares supplementary to the initial triangles, $C$ is the number of pentagons, supplementary to the initial triangles, and $D$ is the number of additional triangles supplementary to the initial triangles, this function yields the length of the side of the union of the tiles for a given angle $\alpha$ corresponding to the initial triangle tile, one of which will always be present. For all relevant choices for the parameters, $g$ is concave down with a single maximum. 

In the case of three tiles, we must consider each of the possibilities. First, suppose all three tiles in the perfect fit are triangles, in which case $(A, B, C, D) = (2, 0, 0, 1)$. Then the middle triangle must be larger than the other two, or they will overlap. There is only one possibility, which is if the middle one is a magic triangle, and the two smaller ones are special triangles from the tri-pent supp-same pair. But because all magic triangles are decomposed, the perfect fit would have tiles of the same side-length as our initial edge-to-edge patch, meaning they would be included in the edge-to-edge patch and the initial bigon was not the entire patch, a contradiction.  

If the three tiles are two triangles with a square in the middle, so that $(A, B, C, D) = (2, 1, 0, 0)$,  then we know that they could form a Type II bigon, and over the range of angles $\pi/3 < \alpha < \pi/2$,  this is the only possibility. If the three tiles are two squares with a triangle between them, so $(A, B, C, D) = (1, 2, 0, 0)$, then we know that the supp-same tri-square pair would yield a total length of exactly $\pi$. In that case, the crevice where the three tiles meet can only be filled with a triangle that then yields a Type III bigon. All other angles applied to the three tiles yields a total side-length less than $\pi$.

If the three tiles are two triangles with a pentagon in the middle, so $(A, B, C, D) = (2, 0, 1, 0)$, then the maximum value for the function is less than 2.4, and we cannot obtain a perfect fit against a side of length $\pi$. Similarly, if we have two pentagons with a triangle in the middle, the maximum value of the function is no more than 1.9.

If the three tiles consist of a square and a pentagon, with a triangle between them, the maximum value for the function is less than 2.5. This takes care of all of the cases of three tiles in the perfect fit. 

We now consider four tiles. Note that there must be at least two triangles. If there are four triangles, then to avoid overlap, they must all be the same size, so their angles are all $\pi/2$, and their total side-length is $2\pi$, too large.

If there are two triangles alternating with two squares, so $(A, B, C, D) = (2, 2, 0, 0)$, then in the limit as $\alpha$ approaches $\pi/2$, the value of the function approaches $\pi$. Otherwise, there is one value at which the function is $\pi$, but it requires the triangles to be smaller than the triangles in the original bigon.

Two triangles alternating with two pentagons yields a maximum for $g$ of 2.75, too small. Three triangles and a square, which is given by $(A, B, C, D) = (2, 1, 0, 1)$ leaves $g$ always too large.

Two triangles alternating with a square and then a pentagon, so $(A, B, C, D) = (2, 1, 1, 0)$ does have  one solution for $g(\alpha) = \pi$, however the triangle size both causes the pentagon and the square to overlap and is smaller than the triangle tiles in the bigon.

If we consider three triangles and a square, so $(A, B, C, D) = (2, 1, 0, 1)$, then $g$ is too large over the range of allowable values. And if we consider three triangles and a pentagon, so $(A, B, C, D) = (2, 0, 1, 1)$, there is a unique solution, however, again the angle $\alpha$ is smaller than the angle in the tri-pent triangle, so this solution is disallowed.

Thus, the only possibilities for perfect fits are bigons of Type I, II and III. 
We first consider $B$ a Type IV bigon. For a choice of the two bigons from this set to border $B$ that are either both Type I or one Type II and one Type III, there remains a bigon $B'$ of apex angle approximately $116.565^\circ$. This can only be filled with a Type V bigon, as we show in Lemma \ref{singletileatcorner}. This yields two more of the sporadic tilings. For any other choice of the adjacent bigons, the apex angle of $B'$ is such that it cannot be filled in with tiles, which also follows from Lemmas \ref{singletileatcorner} and \ref{twotilesatcorner}. 

Finally, if we start with a Type V bigon, a similar argument yields exactly the same pair of sporadic tilings.

\end{proof}

\begin{lemma}
If there is a minimal side-length tile in a maximal edge-to-edge hemisphere, then the tiling is a 2-hemisphere tiling.
\end{lemma}

\begin{proof} If a tiling $\mathcal{T}$ contains an edge-to-edge tiling of a hemisphere, then by doubling that hemisphere across its boundary, we obtain an edge-to-edge tiling $\mathcal{T}'$ of the sphere, all the possibilities for which have been classified. So by considering the Platonic, Archimedean and Johnson spherical tilings, we can find all the possibilities. They are the dodecahedral, cuboctahedral, icosidocdecahedral hemispheres and a hemisphere that is a single $n$-gon of angle $\pi$ for $n \geq 3$. 

The other hemisphere can also be doubled, but then add in a small twist along the equator so the tiles on either side no longer line up. Call this tiling $\mathcal{T}''$.  If the smallest tile only appears as a singleton, then by  Section \ref{singletonsection}, the tiling is a kaleidoscope tiling or one of the lunar tilings. But none of those possess a great circle made up of edges, so this cannot occur. Therefore there is an edge-to-edge bigon made up of the smallest tiles. Because of the twist, that bigon cannot cross the equator. If that bigon has angle less than $\pi$, then by Lemma \ref{bigontilings}, the tiling is either a lunar or sporadic tiling, none of which have a great circle made up of edges. If it is the entire hemisphere, then we are done, as both hemispheres are edge-to-edge and the tilings are 2-hemisphere tilings.  
\end{proof}

The final case to consider is if a minimal side-length tile appears in a edge-to-edge bigon of angle greater than $\pi$. Then the complement of that bigon is a bigon of angle less than $\pi$ that must be tiled. We begin with a lemma.

\begin{lemma} \label{singletileatcorner}If in a decomposed tiling, a bigon $B$ of angle less than $\pi$ is tiled such that there is a single tile $T$ at at least one of the apices, then it must be an edge-to-edge tiling of types I, II, III, IV or V.
\end{lemma}

\begin{proof}
 Let the angle of the bigon and therefore tile $T$ be $\alpha$. Let $\beta$ be the supplementary angle which occurs at both of the other corners of $T$ along the sides of the bigon $B$. 
 
 Suppose first that $T$ is a triangle. If there is a single tile $T'$ glued to the exposed side of $T$ on the inside of $B$, then it must be a supp-same tile.  If it is a triangle, both have angle $\pi$, and  the two tiles make up the entire bigon. If it is a square, it must be special, and then the exposed side of the square in the bigon has supplementary angles $\alpha$ and we are forced to glue in another copy of $T$, giving us a Type II bigon. 
 
If  $T'$ is a pentagon, then it is a special pentagon, and we are forced to glue copies of $T$ to each of the exposed sides of the pentagon in $B$. Since a different choice of $n$-gon would have too large side-length, we are forced to glue another copy of $T'$ to the two exposed sides of the triangles in $B$. And then this leaves room for another copy of $T$ glued to the exposed side of this pentagon. So we obtain a tiling of $B$ of Type IV.
    
We now consider the possibility that there is more than one tile glued to the exposed side of $T$. Then reading across the exposed side of $T$, we have a sequence of supplementary tiles, alternately with angles $\alpha$ and $\beta$. Since they have side-length less than $T$, none of the ones with angle $\alpha$ can be triangles, and therefore all with angle $\beta$ must be triangles, and $\beta < \alpha$. In particular, the two outer ones must be triangles, and there must be an odd number total of them. Then as in Figure \ref{bigoncorneronetri}, if $T'$ is such a triangle of angle $\beta$, one of the tiles $D$ or $E$ glued to its exposed side  must be a triangle, which then  has side-length too long to fit along the triangle of angle $\beta$. (Note that $D$ may be the only tile along the side of $T'$.)  

\begin{figure}
    \centering
    \includegraphics[scale=.5]{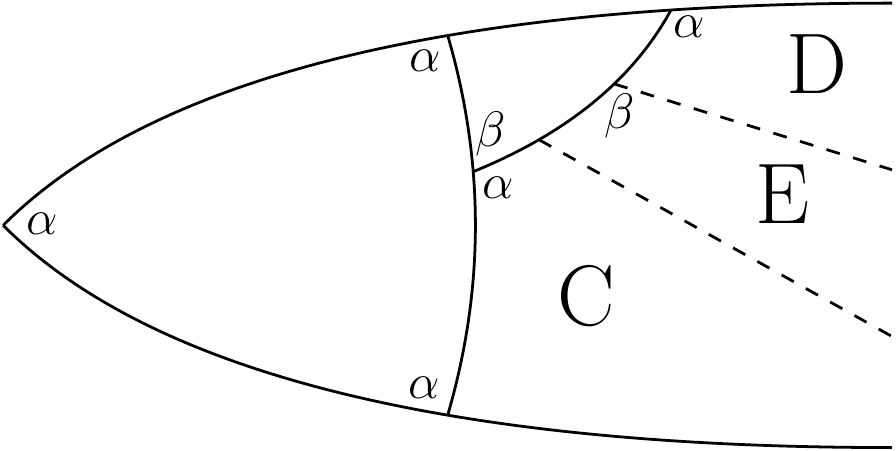}
    \caption{A single triangle in the apex of a bigon cannot have more than one tile glued to it in the bigon.}
    \label{bigoncorneronetri}
\end{figure}

Now, suppose that $T$ is a square of angle $\alpha$. Then at each corner of $T$ along the sides of $B$, there must be a triangle of angle $\beta$. If the triangles have side-length greater than $T$, they collide with one another. If they are supp-same with $T$, then the bigon tiling is completed with a single square of angle $\alpha$, which cannot be further decomposed by Theorem \ref{polygondecomposition}, and we obtain the Type III tiling of a bigon. If the triangles have side-length less than that of $T$, then each triangle generates a crevice of angle $\alpha$ with $T$.
 At least one of these must be filled with a pentagon or else the two tiles filling these crevices will collide. If the pentagon is supp-same with the triangle of angle $\beta$, then $\alpha = 116.566^\circ$, but then the side-length of the square is $.3752 \pi$ which is less than $.4 \pi$, which is the side-length of the triangle plus pentagon, a contradiction.

If the pentagon has side-length less than the triangle, as in Figure \ref{onesquareapex}, then the crevice created by the pentagon on the triangle must be filled with a triangle of angle $\beta$, for which there is not enough room.

\begin{figure}
    \centering
    \includegraphics[scale=.7]{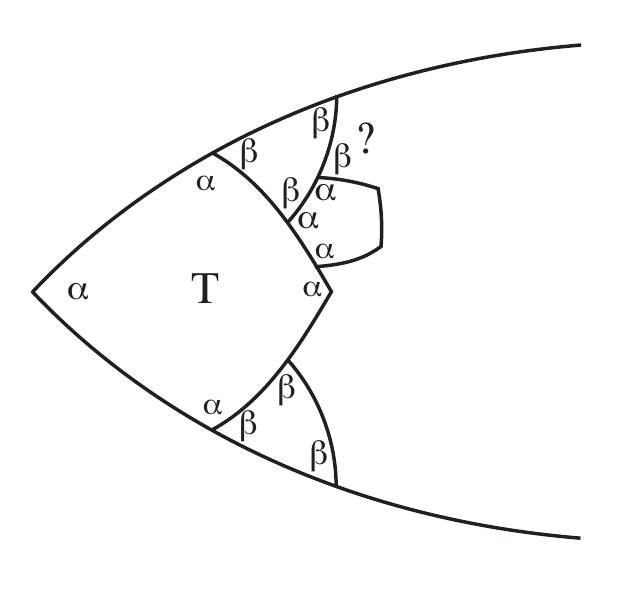}
    \caption{A square at the apex of a bigon when the supplementary triangles have smaller side-length.}
    \label{onesquareapex}
\end{figure}

Finally, if $T$ is a pentagon, the two crevices it creates on the sides of $B$ must be filled with triangles of angle $\beta$. If the triangles have the same side-length as $T$, they are supp-same with it, and then the next tile on each side has angle $\alpha$. Call these two tiles $C$ and $D$. If either one is a pentagon, this forces another triangle of angle $\beta$ to the center, which forces the other to be a pentagon as well. This then forces an additional triangle between them. Along each side of the bigon, there must be another supplementary triangle, which then leaves room for one more pentagon, which cannot be decomposed by Theorem \ref{polygondecomposition}. Thus we have a Type V tiling of the bigon $B$.

On the other hand, if neither $C$ nor $D$ is a pentagon, then if one is a triangle, in fact it is a magic triangle, with side-length $3\pi/5$, forcing $C = D$. But this contradicts the fact we have decomposed all magic triangles. So the only possibility is if both $C$ and $D$ are squares of angle $116.565^\circ$. But such as square has side-length $.3752 \pi$, meaning two cannot fit across a distance of $3\pi/5$ formed by the exposed sides of the union of the pentagon and supp-same triangles.

If the triangles have shorter side-length than $T$, they create crevices of angle $\alpha$ with $T$ that yield wells when we include the sides of $B$.  But any tile of angle $\alpha$ cannot fit in such a well. 

If the triangles have longer side-length than $T$, they create a well with $T$ that must be filled with another triangle of angle $\beta$. But such a triangle is has side-length too large to fit in the well.

\end{proof}

\begin{lemma}\label{twotilesatcorner}
If in a decomposed tiling, a bigon $B$ of angle less than $\pi$ is tiled such that there are two tiles at at least one of the apices, then it must be a tiling obtained by gluing together two of the five possible bigons in the following combinations:

I-II, I-IV, II-II, II-IV, III-IV, IV-IV.

\end{lemma}


\begin{proof} 





Suppose the two tiles in $B$ at a corner are triangles, or a triangle and square, or a triangle and a pentagon such that the triangle $T$ is smaller. Then, treating the boundary of the bigon as a wall for tiles to the inside,  we obtain a well, as in Figure \ref{bigontwotilesapex}(a). If that well is filled with a single tile $T'$, then it must be a supp-same tile with $T$. If $T'$ is a triangle, then both $T$ and $T'$ must be of angle $\pi/2$. But then, after removing the bigon they create, we are left with a bigon with one tile at its apex. Therefore, by Lemma \ref{bigoncorneronetri}, the remainder is a bigon of one of the five types, as we wished to show. 

If $T'$ is a square, then $T$ and $T'$ must be a supp-same tri-square pair. The next tile along that side of $B$ must be a triangle, which is congruent to $T$. Thus, we have a bigon of Type II contained properly within $B$. Again, if we remove it, we have a bigon that satisfies the hypotheses of Lemma \ref{bigoncorneronetri} and therefore must be one of the five types of edge-to-edge bigons.

If $T'$ is a pentagon, then $T$ and $T'$ are a supp-same tri-pent pair. The tile supplementary to $T'$ on the side of $B$ must be a triangle $T''$ congruent to $T$. The next tile on the side of $B$ is supplementary to $T''$, and therefore has angle $116.565^\circ$. It cannot be a triangle because such a triangle is a magic triangle and therefore would have been decomposed. If it were a square, its side-length would be $.3752 \pi$. But then the next tile on the side of $B$ would have to be a triangle. Then the sum of the side-lengths of the tiles on this side of $B$ would be at least $4\pi/5 + .3752\pi$ which is greater than $\pi$, a contradiction.

Thus, $T''$ must be a pentagon congruent to $T'$ and the next tile on the side of $B$ must be a triangle. Thus we have built a bigon tiled as a Type IV bigon, except for a missing triangle in the middle of the far side. If we attempted to fill that slot with a square or pentagon, they would have side-length less that of the supp-same tri-pent pair. This creates two crevices of angle $116.565\circ$ that must be filled with two triangles. But these triangles would be magic triangles that are decomposed. So this cannot occur.

Thus, there is a bigon of Type IV that we can peel off $B$, leaving a bigon that by Lemma \ref{singletileatcorner} must be one of the five types.

It could be that instead of a single tile $T'$ glued to the exposed side of $T$, there is more than one tile glued there. But as in Figure \ref{bigontwotilesapex}(a), two of the tiles have angle $\beta$ where $\beta + \mu = \pi$. But since $\mu + \gamma < \pi$, we know $\beta > \gamma$. But we are assuming $\gamma > \mu$, so $\beta > \mu$. Hence neither tile with angle $\beta$ can be a triangle, as they would have too great side-length. So there must be at least three tiles glued to the exposed side of $T$. They alternate angles $\beta$ and $\mu$, but every other one must be a triangle, so there will be a triangle that is too long to fit against the side of $T$.

\begin{figure}
    \centering
    \includegraphics[scale=.8]{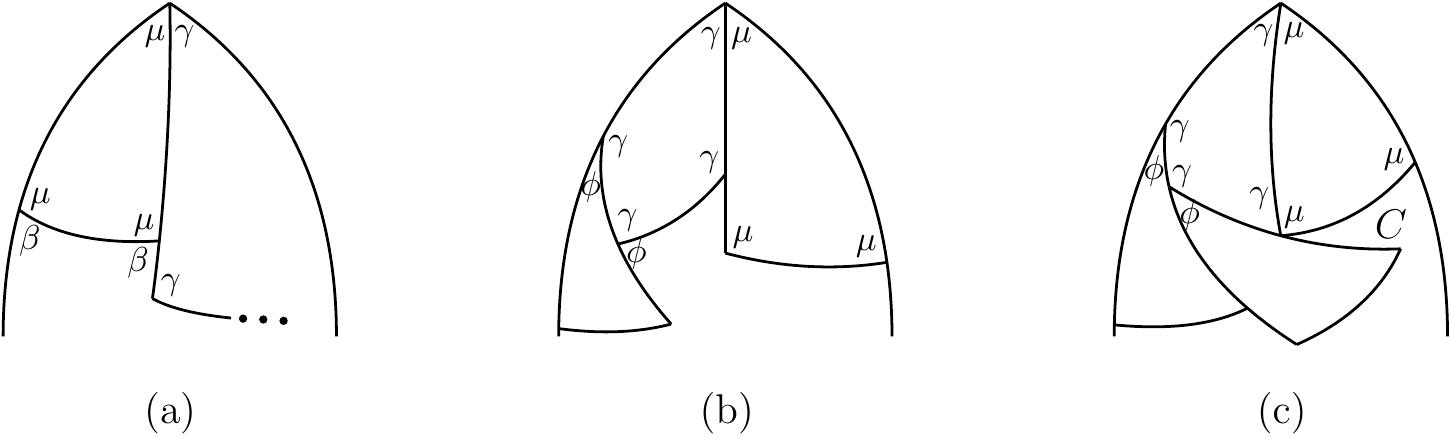}
    \caption{Two tiles at the apex of a bigon does not occur.}
    \label{bigontwotilesapex}
\end{figure}

If the two tiles at the apex are a triangle $T$ and square $S$ and the square is of smaller side-length than the triangle, as in Figure \ref{bigontwotilesapex}(b), then the supplementary tile to the square on the side of $B$ is a triangle of angle $\phi$ greater than $\mu$. It creates a crevice with $S$ also of angle $\phi$ that must be filled with a triangle which is of too great side-length to fit in the corresponding well.

A similar argument with one more triangle deals with the case of a triangle and pentagon at the apex with the pentagon of smaller side-length.

In the case the two tiles are a triangle $T$ and a square $S$ of equal side-length, then again, as in Figure \ref{bigontwotilesapex}(c), we must have two supplementary triangles to the square of angle $\phi$, each of which is of greater side-length than $T$. This creates a crevice $C$ that does not have enough angle for a tile. In the case of a triangle and pentagon of equal side-length, the same argument with one more triangle also yields a contradiction.

The last case is if  both tiles at a corner are triangles $T_1$ and $T_2$ of the same side-length. Then they are both of equal angle $\mu < 90^\circ$. The supplementary tiles along the sides of the bigon cannot both be triangles, as they would be of angle greater than $90^\circ$ and therefore of greater side-length than $T_1$ and $T_2$ and they would collide.

If there is more than one tile glued to the exposed sides of at least one of $T_1$ and $T_2$, then since adjacent tiles are supplementary, there will be at least one triangle with angle either $\beta$ or $\mu$, where $\beta > \mu$. Thus, that triangle will have side-length at least as long as the side-lengths of $T_1$ and $T_2$, and the tiles glued to the exposed side of $T_1$ will collide with the tiles glued to the exposed side of $T_2$, a contradiction. 

So the tiles glued to the exposed sides of $T_1$ and $T_2$ must be the same side-length as $T_1$ and $T_2$, and therefore be congruent, since their angles are the same. If they are both squares, then they are tri-square squares with the triangles,and each exposed side inside the bigon must be filled with triangles congruent to $T_1$ and $T_2$. In other words, the bigon is two Type II bigons glued along an edge. 

If the tiles glued to the exposed sides of $T_1$ and $T_2$ are pentagons, then the triangles and pentagons must be tri-pent tiles. Continuing along each boundary of the bigon, we then must glue on anoher tri-pent triangle, and then a tri-pent pentagon and then a tri-pent triangle. At the center, we have room for two more tri-pent triangles, and we now see the bigon as the union of two Type IV bigons. 
Other than these last two cases, we know that the bigon must be a Type I or Type IV bigon paired with one of the other bigons. Considering which can be paired to obtain a bigon of apex angle less than $\pi$, we obtain the list given in the theorem.

\end{proof}

\begin{corollary} \label{bigonbigangle}
There are no spherical tilings by regular polygons where an edge-to-edge patch containing a minimal side-length tile is a bigon $B$ of angle greater than $\pi$. \end{corollary}

\begin{proof} By Lemmas \ref{singletileatcorner} and \ref{twotilesatcorner}, if such a tiling existed, the complementary bigon $B'$ would have to be one of the five types of bigons or combinations of two such bigons with the sum of their apex angles less than $\pi$. But $B$ cannot contain tiles of the same size as those in $B'$, where $B$ is made up of a union of Type I bigons, or a union of Type II and Type III bigons or a union of Type IV and Type V bigons. But then there is no combination of angles at the apices of the bigons in $B$ and $B'$ that will add to $360^\circ$.   
 
\end{proof}





\section{Magic Triangles}\label{magictrianglesection}

We now have a complete classification  of decomposed tilings of the sphere by regular polygons, including the tilings from this paper, and the additional edge-to-edge tilings that were known to exist before this. If any one of those contains a collection of tiles that can be composed into a magic triangle, then we obtain a new non-edge-to-edge tiling. We list the additional tilings here.

Starting with the icosidodecahedron, each triangle in the tiling is the center of a set of tiles that together compose to yield a magic triangle. So composing any one of those sets yields a non-edge-to-edge tiling. 

We can also pick any two such sets that do not overlap in the interior of tiles, and compose them to obtain a tiling with two magic triangles. The magic triangles can be opposite one another on the icosidodecaedron. The resulting tiling is in fact the same as the lunar tiling with Type V bigons. Note that if we take that lunar tiling and decompose one of the polar triangles, we obtain the tiling from the preceding paragraph. 

We can also have the two magic triangles on the icosidodecahedron touch on their boundaries, overlapping a length of $2\pi/5$ on an side of each. Or they can be chosen to share a single vertex. In both these cases note how irregular the boundary of the edge-to-edge patch becomes.

We can also choose three sets of tiles that each compose to a magic triangle. In this case, each magic triangle has a side that overlaps with the side of another one in a length of $2\pi/5$. 

There are also four sets of such tiles that do not overlap in their interior, however if we compose all four, we obtain a tiling we have already seen, which is the tri-pent kaleidoscope tiling with triangle of side-length exactly $3\pi/5$. We also note that certain of these new tilings could have been obtained by starting with the tri-pent kaleidoscope tiling with triangle of side-length $3\pi/5$ and subdividing one, two or three of the magic triangles within it.

Moving to the sporadic tilings by bigons, note than any such that contains a bigon of Type V has a set of tiles that compose into a magic triangle. So since two of the sporadic bigon tilings contain one copy of that bigon, this generates two more tilings.

Finally, note that if we have a hemisphere tiled edge-to-edge by tri-pent supp-same tiles, which is the same as an icosidodecahedral hemisphere, there is room for a subset of tiles that compose into a magic triangle, which has one side on the boundary of the hemisphere. So any tiling that contains a icosidodechaedral hemisphere also generates a tiling with a magic triangle. This includes the edge-to-edge Johnson tilings corresponding to Johnson solids J6 and J34, in addition to the 2-hemisphere tilings that use an icosidodecahedral hemisphere.

A final case is if we take a tiling by two tiles, one a magic triangle and one the complement. Then by decomposing the magic triangle, we obtain a tiling called the {\bf magic triangle tiling}. Note that this is the only tiling obtained in this manner, as if there is a tile of angle greater than  $\pi$, then its complementary tile has angle less than $\pi$. By Theorem \ref{polygondecomposition} from the appendix, there are four options for a tiling a regular polygon of angle less than $\pi$. The only one that generates a non-edge-to-edge tiling is the magic triangle decomposition.

\section{Conclusion}

We can now prove Theorem \ref{maintheorem}.
Given a non-edge-to-edge tiling $\mathcal{T}$ of the sphere by regular polygons, all of angle less than $\pi$, we saw in Lemmas 3.2, 3.3, and 3.4 that if there is a minimal side-length prototile that appears only as a singleton, then the tiling must be one of the five families of kaleidoscope tilings or it must be the lunar tiling with polar pentagons. In Lemma 4.1, we proved that if $\mathcal{T}$ is a decomposed tiling, and a minimal side-length tile appears in an edge-to-edge patch, that patch must be a bigon. In Lemma 4.3, we proved that if the bigon angle is less than $\pi$, the tiling must be one of the remaining three lunar tilings or one of the three sporadic tilings. In Lemma 4.4, we proved that if the bigon has angle $\pi$, so it is a hemisphere, we obtain one of the ten 2-hemisphere families of tilings. And in Corollary \ref{bigonbigangle}, we showed there are no tilings corresponding to the bigon having angle greater than 
$\pi$.  In Section 5, we found all tilings that could be obtained by composing the tiles in magic triangles that exist  in either edge-to-edge tilings or any of the tilings we have listed. Further in  Theorem A.1, we proved that a regular polygonal tile of angle less than $\pi$ can be decomposed into regular polygons with at least one  $\pi$-vertex if and only if it is a magic triangle. Thus, there is only one additional non-edge-to-edge tiling that comes from allowing polygons of angle greater than $\pi$. 
    
    This completes the proof of Theorem 1.1.
    
    \appendix
    \section{Tiling Regular Polygons}

The following theorem is useful at various points in the paper, but it is also of independent interest. 

\begin{theorem} \label{polygondecomposition} A regular spherical $n$-gon with $n \geq 3$ and angle less than $\pi$ can be decomposed into regular polygonal tiles if and only if  it is one of the following:
\begin{enumerate}
\item The magic triangle 
\item The pentagon of angle $4\pi/5$ decomposing edge-to-edge into five triangles
\item The octagon decomposing edge-to-edge into four triangles and four squares surrounding a fifth square
\item The decagon decomposing edge-to-edge into five triangles and five squares surrounding a pentagon 
\end{enumerate}
\end{theorem}

The first appears as a subset of the tiles in the spherical projection of the icosidodecahedron. The second appears as a subset of the tiles in the spherical projection of the icosahedron. The third appears as a subset of the tiles in the spherical projection of the rhombicuboctahedron. The fourth appears as a subset of the tiles in the spherical projection of the rhombicosidodecahedron.

\begin{proof} 

\noindent {\bf Case 1.} We first consider edge-to-edge tilings of such an $n$-gon $P$ with angle $\lambda < \pi$ and with no $\pi$-vertices on its boundary.
Then each corner must consist of the corners of at least two tiles, since if not, the tile at that corner would have the same angle and because of there being no $\pi$-vertices, would have the same side-length as $P$, and therefore be $P$. Further, because  each corner has angle less than $\pi$, and each angle of a tile is greater than $\pi/3$, there are at most two tiles meeting at a corner of $P$, and they must either be two triangles, a triangle and a square or a triangle and a pentagon.

In the first case, all the triangles have the same side-length and therefore must meet at a central vertex. The angle of one such triangle at the central vertex is $2\pi/n$, making the angles of $P$ equal to $4\pi/n$. But those angles must be less than $\pi$, forcing $n > 4$. But if $n \geq 6$, the angle of the triangle is less than or equal to $\pi/3$,  a contradiction. So the only possibility is $n = 5$, yielding the pentagon with angle $4\pi/5$, decomposing into five triangles. 

If we allow triangles and squares and triangles and pentagons to meet at each corner of $P$, then the choice must be the same at every corner, as all these tiles have the same side-length, and the tiles to either side of a triangle have the same angle, since the angles of the two tiles meeting at a corner of $P$ must add to the angle of that corner.

If we alternate squares of angle $\beta$ and triangles of angle $\alpha$, then $n$ must be even. Once we include this outer layer of the tiling of $T$,  each interior vertex is incident to two squares and one triangle, with angle sum greater than $\pi/3 + 2 (\pi/2) = 4\pi/3$. Thus, the interior unfilled $n/2$-gon must have an angle less than $2\pi/3$, implying that $n/2 < 6.$ If $n/2 = 4$, $P$ is the octagon in the conclusion of the theorem. If $n/2 = 5$, $P$ is the decagon in the conclusion to the theorem. In both cases, that the interior polygon cannot be further decomposed will follow from the rest of the proof in this theorem. If $n/2 = 3$, the central triangle would have the same side-length and therefore angle as the rest of the triangles. So we would have $2\alpha + 2\beta = 2\pi$, implying $\alpha + \beta = \pi$ and $P$ has angle $\pi$, a contradiction. 

If we alternate pentagons of angle $\beta$ and triangles of angle $\alpha$ in the outer layer, then each interior corner of a triangle is shared with the corners of two pentagons. Let $\gamma$ be the angle that is to be filled. So $\alpha + 2 \beta + \gamma = 2 \pi$. But since triangles have angle greater than $\pi/3$ and pentagons have angle greater than $3\pi/5$, $\gamma < 7\pi/15 < \pi/2$. So we can only fill this crevice with a triangle. However, $\alpha + \beta < \pi$ implies $\beta + \gamma > \pi$, which implies $\gamma > \alpha$. So the triangles filling these $n/2$ crevices all have side-length longer than the side-lengths of the other tiles we have already inserted. This forces the new tiles, appearing as tiles $B$ and $C$ in Figure \ref{tripentcontradiction}, to overlap, a contradiction.

\begin{figure}
    \centering
    \includegraphics[scale=.7]{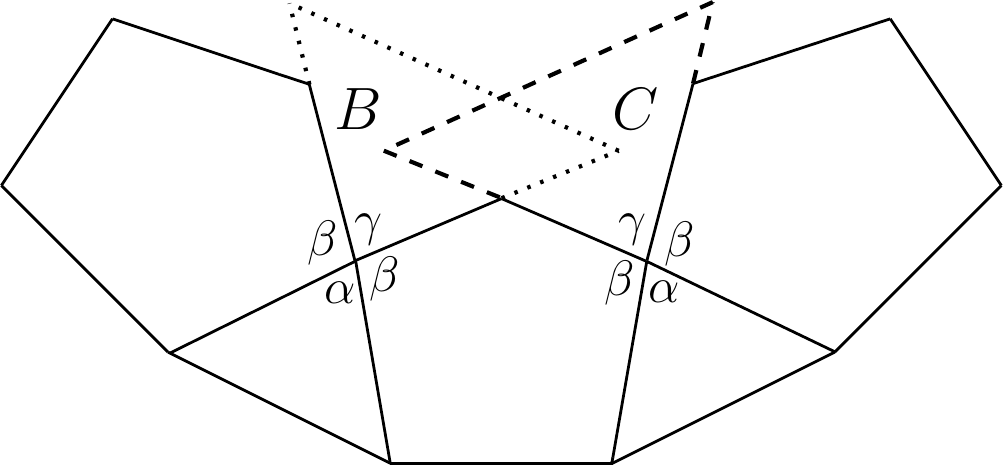}
    \caption{Alternating pentagons and triangles with no $\pi$-vertices leads to tile overlap.}
    \label{tripentcontradiction}
\end{figure}

\medskip
\noindent{\bf Case 2.} We now consider the case that there are $\pi$-vertices on at least some sides of $P$. Every $\pi$-vertex must involve a triangle. We define a corner of $P$ to be {\bf split} if two tiles in the tiling of $P$ meet at that vertex.

\medskip



\noindent{\bf Case 2a.} There is a side with a single $\pi$-vertex and neither of the two corners of the side are split. Then we cannot have two triangles, as this would prevent the polygon from being regular unless the triangles were the same side-length, and therefore supp-same, meaning the side-length of $P$ would be $\pi$, which is too large for the side-length of any regular polygon of three or more sides. 

If there is a triangle and a square, they cannot have the same side-length, as then the angles at the two corners of $P$ corresponding to this side would be different. The square cannot have longer side length than the triangle, as this would prevent $P$ from being regular. So this situation must appear as in Figure \ref{pivertexnosplit}(a). So tile $B$ must be a triangle since it is supplementary to a square. But then there is not room for it along the initial triangle. The same argument shows that we cannot have a triangle and a pentagon unless the pentagon has side-length equal to the triangle, as in Figure \ref{pivertexnosplit}(b) and (c). But then the side-length of the triangle is the same as the side-length of the pentagon. So we have a supp-same pair. Thus, all the sides of the polygon must be of length $2\pi/5$. But then $\alpha$ and $\beta$ are both angles of $P$ and hence equal to each other and therefore to $\pi/2$. But this is not possible for a pentagon. 

\begin{figure}
    \centering
    \includegraphics[scale=.45]{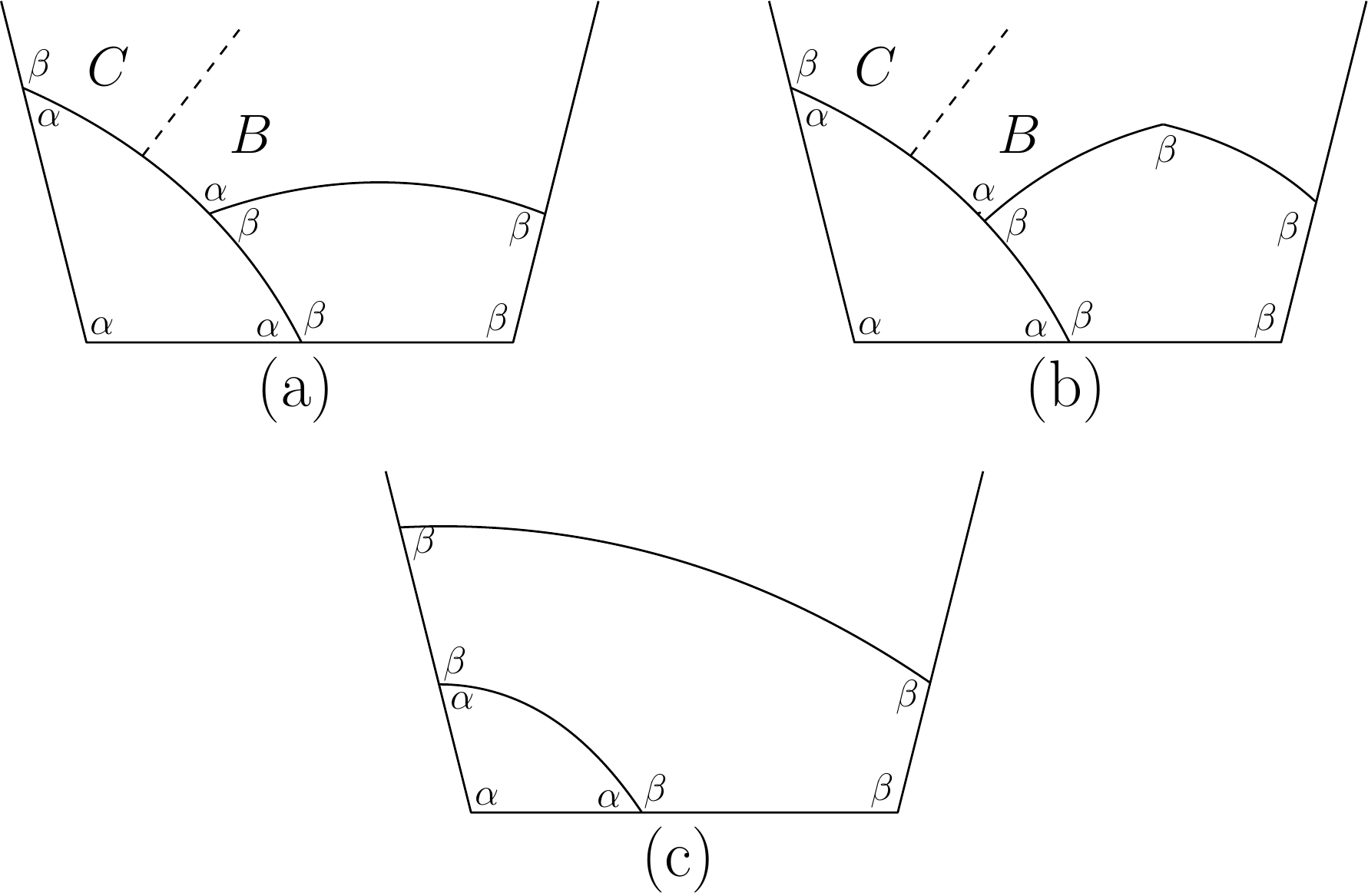}
    \caption{A side cannot contain a single $\pi$-vertex if both corners are not shared with more than one tile.}
    \label{pivertexnosplit}
\end{figure}

\medskip

\noindent{\bf Case 2b.} There is a side with a single $\pi$-vertex and both of the two corners of the side are split. There are a variety of cases to consider. When the pair are not of equal side-length, there are two triangles, a triangle and square with the triangle smaller side-length, a triangle and pentagon, with the triangle of smaller side-length, a triangle and square with square the smaller and a triangle and pentagon with pentagon the smaller. These are depicted in Figure \ref{pivertexsplits}. 

\begin{figure}
    \centering
    \includegraphics[scale=.7]{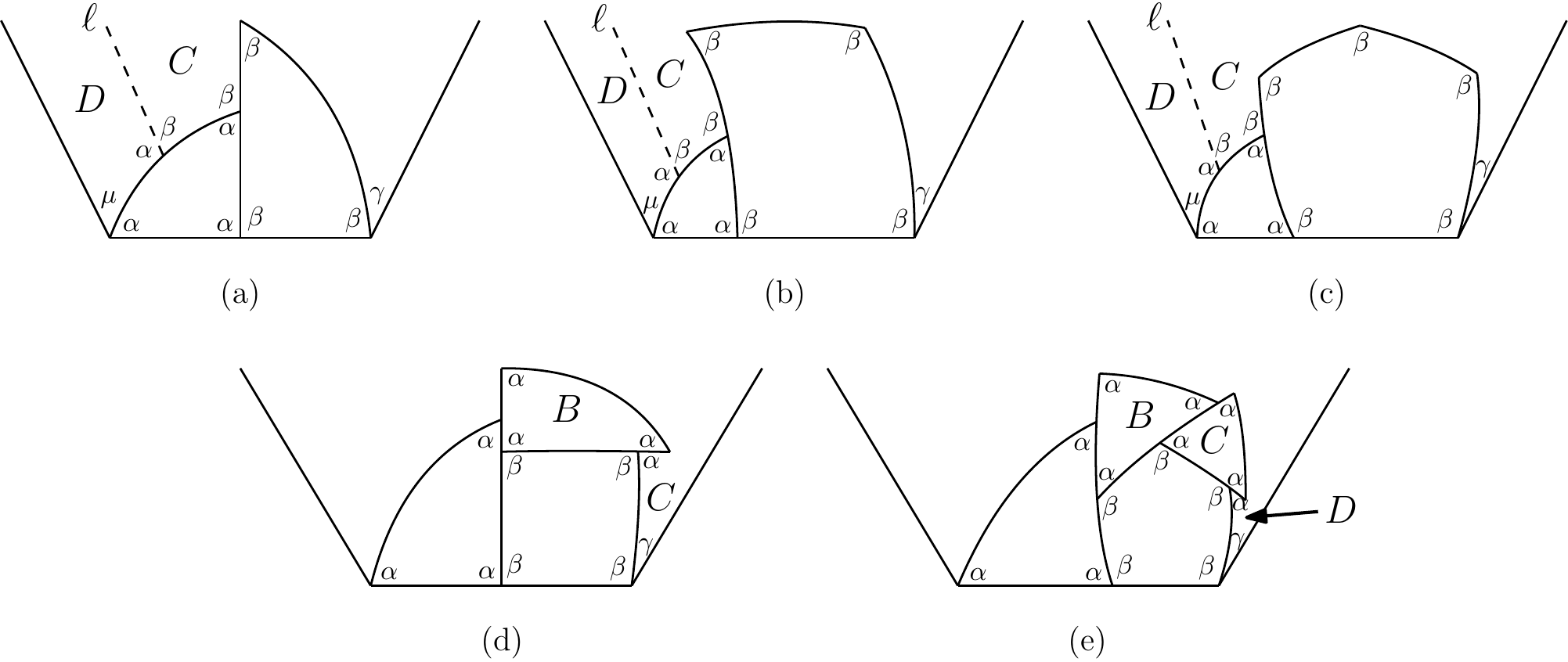}
    \caption{A side cannot contain a single $\pi$-vertex if both corners are split.}
    \label{pivertexsplits}
\end{figure}

In cases (a), (b), and (c), the angles along the exposed side of the triangle of angle $\alpha$ are $\beta$ and $\mu$. But $\beta$ cannot equal angle $\mu$, as $\alpha + \beta =\pi$ but $\alpha + \mu < \pi$. So there must be at least one half-vertex along this side. The tiles glued to this side alternate between angles $\alpha$ and $\beta$. Every other one must be a triangle. But then there is a triangle of angle $\alpha$ or $\beta$, both of which have side-length too large to be glued along this side. 

In cases (d) and (e), the tile labelled $B$ must be a triangle since it is supplementary to a non-triangle. Then in (d), $C$ is also supplementary to a square and hence must be a triangle of angle $\alpha$. However than there is not room for it on the side of the square. 

Similarly in (e), tile $C$ must also be a triangle of angle $\alpha$ and again here, tile $D$ must also be a triangle of angle $\alpha$ and there is not room for it on the side of the pentagon.

If the two tiles are of equal side-length, then we have a supp-same pair. In the case of triangle and triangle their total length is $\pi$, too large for a side of $P$. If they are a triangle and a square, their total side length is $2\pi/3$, which can only occur if $P$ is a triangle of angle $\pi$, which we are not allowing here.  In the case of a triangle and a pentagon, $\alpha \approx 63.435^\circ$ and $\beta \approx 116.565^\circ$. This implies $\gamma < 63.435^\circ$ which means that the tile filling the crevice between the pentagon and the side is a triangle that has side-length shorter than the side of the pentagon. This creates a well that has angles greater than $116.565^\circ$. No tile of that angle can fit in the well, since the smallest tile of angle $116.565^\circ$ is the very pentagon we are using.

\medskip
\noindent{\bf Case 2c.} There is a side with a single $\pi$-vertex and one of its two corners is split and the  other is not.  The arguments for the cases (a), (b) and (c) from Figure \ref{pivertexsplits} did not depend on whether or not the rightmost tile was in a corner by itself or with another tile. So the only cases we need to consider appear in Figure \ref{pivertexonesplit}. 

\begin{figure}[h]
    \centering
    \includegraphics[scale=.4]{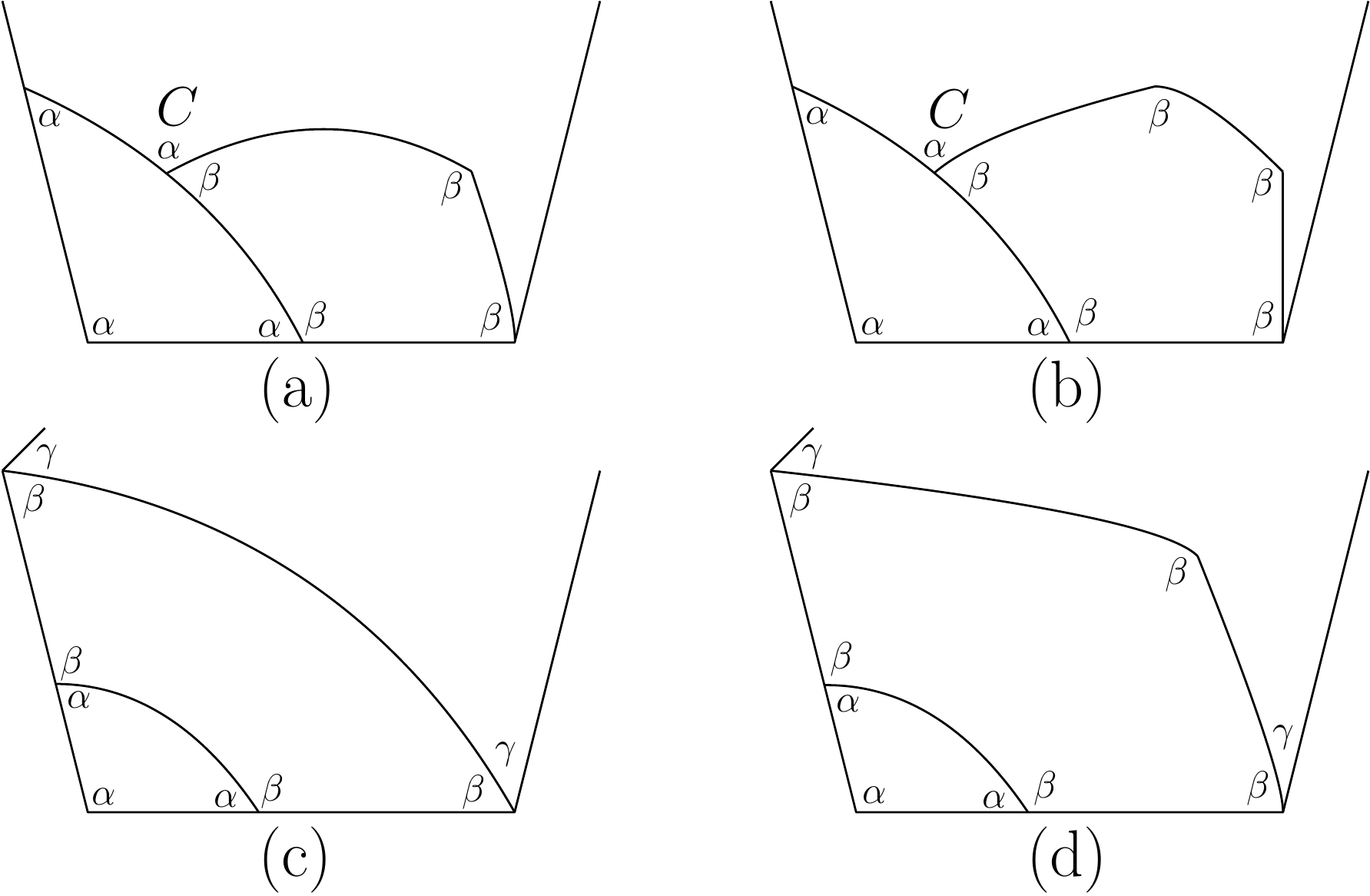}
    \caption{A side cannot contain a single $\pi$-vertex if one corner is split.}
    \label{pivertexonesplit}
\end{figure}

But in cases (a) and (b), tile $C$ is supplementary to a non-triangle, and so it must be a triangle of angle $\alpha$, which has side-length too long to fit here. In cases (c) and (d), $\beta + \gamma = \alpha$ so $\beta < \alpha$, which contradicts the angles necessary for these supp-same pairs.

\medskip
\noindent{\bf Case 2d.} There is a side with two $\pi$-vertices, and neither of the corners of this side are split. The possibilities with tiles that are not supp-same appear in Figure \ref{twopiverticesnosplit}.

\begin{figure}[h]
    \centering
    \includegraphics[scale=.5]{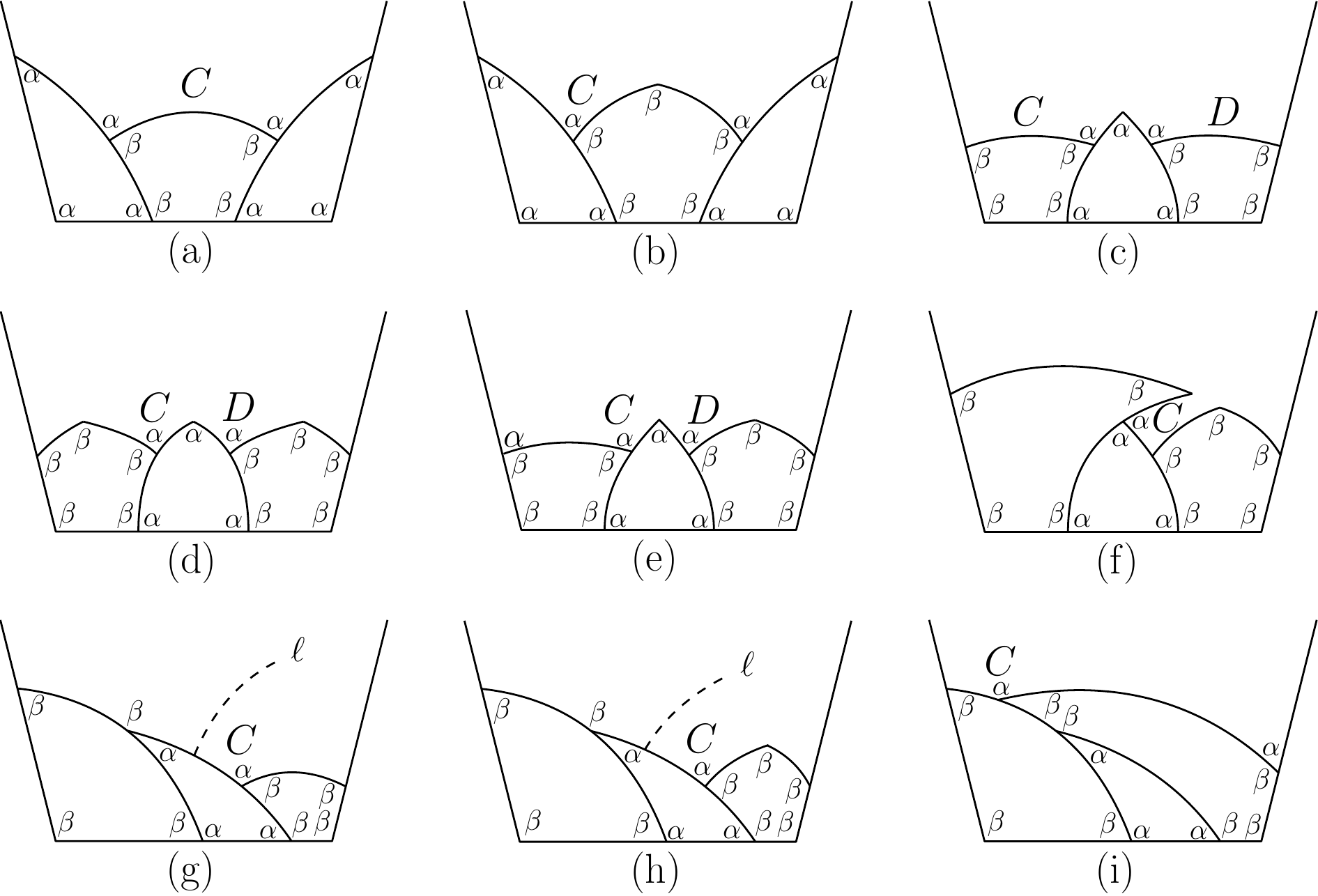}
    \caption{A side cannot contain two $\pi$-vertices and no splits when the side-lengths are different.}
    \label{twopiverticesnosplit}
\end{figure}

In all of the cases,  tiles $C$ and $D$ must be a triangles of angle $\alpha$. In cases (a), (b), (f), (g) and (h), there is not enough room in a well for the tile $C$. In cases (c), (d), and (e), tiles $C$ and $D$ collide.  In the case of (i), the triangles of angle $\alpha$ have the same length as the pentagon of angle $\beta$. Hence these are a supp-same tri-pent pair. This forces the triangle of angle $\beta$ to have side-length $3\pi/5$, making it a magic triangle. The polygon is forced to be a bigon, which, if the magic triangle were decomposed, would be of Type V. But in particular, this  means $P$ is not an $n$-gon for $n \geq 3$.

If we allow tiles of the same side-length, then there are six cases as in Figure \ref{twopiequallengthnosplit}. 
\begin{figure}
    \centering
    \includegraphics[scale=.65]{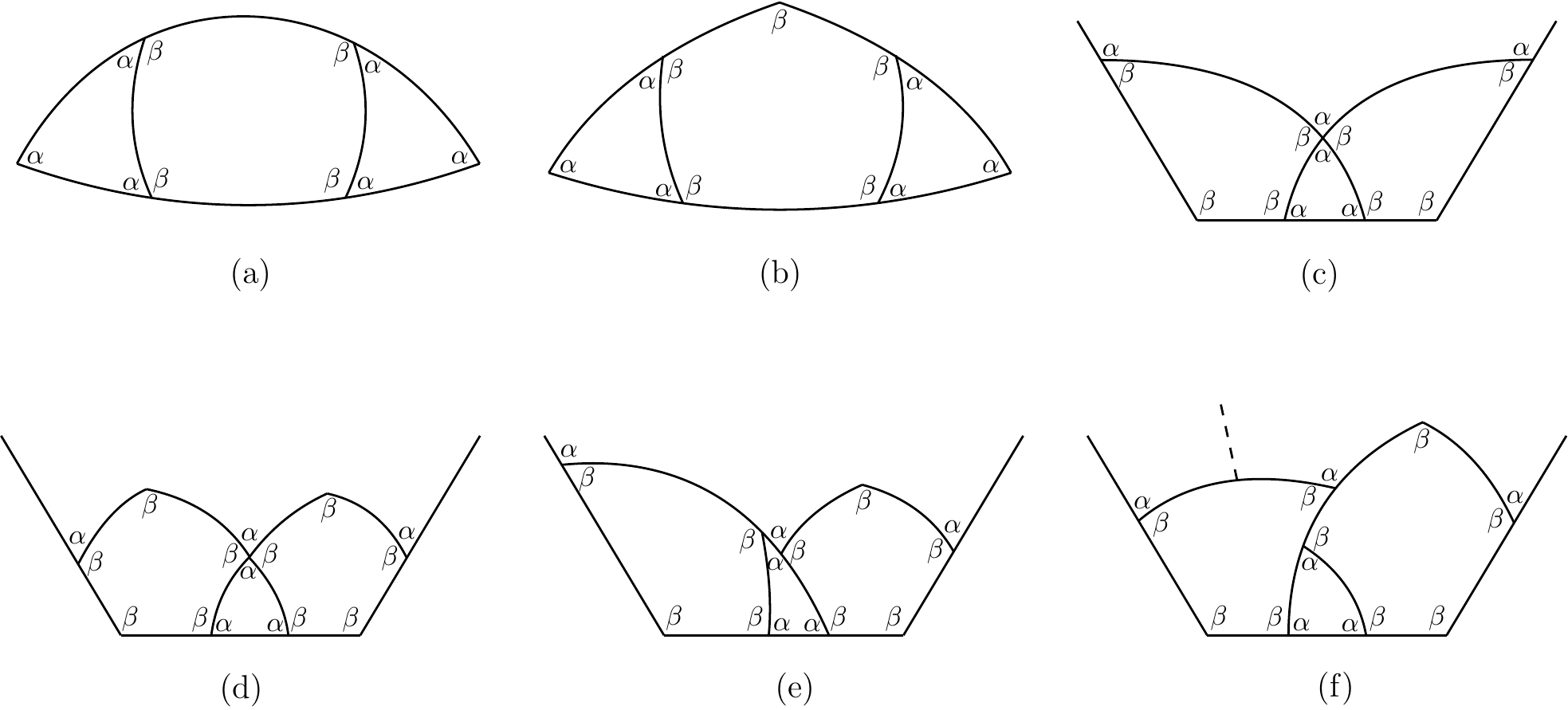}
    \caption{There is only one option when there are two $\pi$-vertices, no splits and tiles of the same length.}
    \label{twopiequallengthnosplit}
\end{figure}
In case (a), the three tiles can only generate a bigon. In case (b), the bottom side of $P$  has length $3\pi/5$. But there is no way to add tiles to these three to expand both of the other two sides to that length. In case (c), since $\alpha  \approx 70.528^\circ$, it can only be a triangle that fills the crevice at the center. But this completes the patch into a bigon. In case (e), both crevices from the pentagon must be filled with triangles of angle $\alpha$ which will collide with one another. In case (f), the top side of the square is too long for one triangle of angle $\alpha$ but too short for two. But the two corresponding crevices must be filled with triangles, a contradiction. 

Thus, the only possibility is case (d). But then because the side-length is $3\pi/5$ and the corner angle $\beta \approx 116.565^\circ$, this must be a magic triangle. To see that it can only be tiled in one way, note that all three crevices must be filled with triangles that are supp-same. This leaves room at the top for a single  supp-same-pentagon, as expected.

\medskip
\noindent{\bf Case 2e.} There is a side with two $\pi$-vertices, and one or two of the two corners is split. Then both corners must be split as the angles for the  outer two tiles are the same, as are the angles of $P$.

This does not impact any of the arguments in the previous case for the situations pictured in Figure \ref{twopiverticesnosplit}.
(In case (i), the contradiction comes from the fact the side-length of $P$ is forced to be $\pi$.)
In the case of Figure \ref{twopiequallengthnosplit} but with both corners split, we see 
cases (a) and (c) have side-length $\pi$ which is too long. For cases (e) and (f), the central crevice must be filled with a triangle of angle $\alpha$. This forces a new crevice on the pentagon with angle $\alpha$ which must also be filled with a triangle. And this creates an additional crevice of angle $\alpha$ on the pentagon, for which we do not have room for the requisite triangle. 

     In cases (b) and (d), the residual angle $\gamma$, not including $\beta$, at the corners of $P$ must be less than $\alpha$. So those crevices must be filled with triangles of angle $\gamma$. This creates a well on each side, with angles greater than $\beta$. Hence there are no tiles that can fit in those wells, as pentagons of angle $\beta$ are already too large.

\medskip

\noindent{\bf Case 2f.} We show that there are never more than two $\pi$-vertices on a side of $P$. If there are three or more $\pi$-vertices, then these correspond to a sequence of four or more tiles, each supplementary to the preceding one. In particular there are at least two triangles of the same angle $\alpha$ and two other tiles that have angle $\beta$ where $\alpha + \beta = \pi$. One of the triangles $T$ is surrounded by two tiles in the sequence. The triangle must have side-length no smaller than one of them, or those two tiles will overlap. For a given angle $\alpha$, the least length for the four tiles occurs if both neighbors are pentagons. If the triangle forms a supp-same pair with them, then the sequence of four tiles has a side-length of $4\pi/5$ which is longer than the largest possible side-length of a regular polygon, which is $2\pi/3$. If the pentagons shrink, the length first grows, passes through a maximum and then shrinks toward the length of two triangles of angle $2\pi/5$, each of which have length $.3524 \pi$. So the total length remains larger than $2\pi/3$. If we replace the pentagons with two squares or a square and a triangle, the angle $\alpha$ can shrink to $\pi/2$, but this still yields a total side-length of  at least $\pi$. If all four are triangles the only possibility is that they all have angle $\pi/2$ and the total length is $2\pi$.








\end{proof}


\bibliographystyle{amsalpha}
\bibliography{biblio.bib}

\medskip
\medskip

{\footnotesize \noindent COLIN ADAMS, DEPARTMENT OF MATHEMATICS, WILLIAMS COLLEGE, WILLIAMSTOWN, MA 01267 \\
{\it E-mail address:} {\bf cadams@williams.edu}

\medskip

\noindent CAMERON EDGAR, DEPARTMENT OF MATHEMATICS, WILLIAMS COLLEGE, WILLIAMSTOWN, MA 01267 \\
{\it E-mail address:} {\bf cse1@williams.edu}

\medskip

\noindent PETER HOLLANDER, DEPARTMENT OF MATHEMATICS, WILLIAMS COLLEGE, WILLIAMSTOWN, MA 01267 \\
{\it E-mail address:} {\bf  pjh1@williams.edu}

\medskip

\noindent LIZA JACOBY, DEPARTMENT OF MATHEMATICS, WILLIAMS COLLEGE, WILLIAMSTOWN, MA 01267 \\
{\it E-mail address:} {\bf ljj1@williams.edu}}

\end{document}